\documentclass[12pt,reqno]{amsart}

\usepackage[utf8]{inputenc}
\usepackage{a4wide}
\usepackage{amssymb,amsmath,amsthm,newlfont,enumerate}

\usepackage{appendix}
\usepackage{dsfont}
\usepackage{stmaryrd} 
\usepackage{bm}
\usepackage{amsfonts}
\usepackage{amssymb}
\usepackage{cancel}
\usepackage{amsbsy}
\usepackage[dvips]{graphicx}
\usepackage{psfrag}
\usepackage[hang,center]{caption}
\usepackage{verbatim} 
\usepackage{cite} % à inclure si on utilise \usepackage[french]{babel}
\usepackage{enumitem}
\usepackage[colorlinks=true,linkcolor=blue,urlcolor=blue]{hyperref}
\usepackage{color}
\usepackage[normalem]{ulem}
\usepackage{fancyhdr}
\usepackage{bm}
\usepackage{xcolor}
\usepackage{empheq}
\usepackage{placeins}
\usepackage{subcaption}
\usepackage{booktabs}
\usepackage{siunitx}
\usepackage{float}
\usepackage{cite}
\usepackage{geometry}
\usepackage{booktabs}
\usepackage{multirow} %this is for multiple rows you don't need if you wont use it.
\usepackage{multicol}
\usepackage{rotating}
\usepackage{mathtools}    
\DeclarePairedDelimiterX\setc[2]{\{}{\}}{\,#1 \;\delimsize\vert\; #2\,}

\def\paragraph#1{{\bf #1\ }}

 \def\OO{\rm \hbox{O\kern-.34em\raise.47ex
         \hbox{$\scriptscriptstyle |$}\kern-.46em\raise.47ex
         \hbox{$\scriptscriptstyle |$}\kern+0.5 em }}
\def\RR{\mbox{\mathrm I\hspace{-0.50ex}R} }

\def\hcboxcm#1#2{\hbox to #1{\hfill #2 \hfill}}

\def\null{\hbox{}}

\def\tn1{\widetilde n_1}
\def\tn2{\widetilde n_2}
\def\tn{\widetilde n }

\let\ds\displaystyle

\def\be{\begin{equation}}
\def\ee{\end{equation}}
\def\bea{\begin{eqnarray}}
\def\eea{\end{eqnarray}}
\def\bean{\begin{eqnarray*}}
\def\eean{\end{eqnarray*}}

\def\RR{{\mathrm{ I~\hspace{-1.15ex}R}}}

\def\qquad{{\quad\quad}}

\def\={\, = \, }

\def\Box{\leavevmode\vbox{\hrule
     \hbox{\vrule\kern4pt\vbox{\kern4pt}%
           \vrule}\hrule}}
\def\blackbox{\leavevmode\vrule height 5pt width 4pt depth 0pt\relax}
\catcode`@=11

\def\eqalign#1{\null\,\vcenter{\openup1\jot \m@th
   \ialign{\strut \hfil$\displaystyle{##}$ & $\displaystyle{{}##}$\hfil
      \crcr#1\crcr}}\,}
\def\eqalignrll#1{\null\,\vcenter{\openup1\jot \m@th
   \ialign{\strut \hfil$\displaystyle{##}$ & $\displaystyle{{}##}$\hfil
    & $\displaystyle{{}##}$\hfil
      \crcr#1\crcr}}\,}
\def\eqalignrcl#1{\null\,\vcenter{\openup1\jot \m@th
   \ialign{\strut \hfil$\displaystyle{##}$ &\hfil $\displaystyle{{}##}$\hfil
    & $\displaystyle{{}##}$\hfil
      \crcr#1\crcr}}\,}
\def\eqalignno#1{\displ@y \tabskip\@centering
  \halign to\displaywidth{\hfil$\@lign\displaystyle{##}$\tabskip\z@skip
    &$\@lign\displaystyle{{}##}$\hfil\tabskip\@centering
    &\llap{$\@lign##$}\tabskip\z@skip\crcr
    #1\crcr}}
\newcounter{appendix}
\newcounter{sectionz}
\def\appendix{\advance\c@appendix by 1
   \def\thesectionz {\Alph{appendix}}
\def\thesection{\Alph{appendix}} 
   \ifnum\c@appendix=1 \setcounter{section}{-1} \fi
   \@startsection {section}{1}{\z@}{-3.5ex plus -1ex minus 
  -.2ex}{2.3ex plus .2ex}{\large\bf}}

\catcode`@=12
\newtheorem{lemme}{Lemma}[section]  

\newtheorem{theorem}[lemme]{Theorem}

\newtheorem{corollary}[lemme]{Corollary}

\newtheorem{definition}[lemme]{Definition}

\newtheorem{proposition}[lemme]{Proposition}

\newtheorem{remark}[lemme]{Remark} 

\newcounter{hypothesiscounter}

\def\deblem{\begin{lemme}\it }
\def\finlem{\end{lemme}}
\def\debthm{\begin{theorem}\it }
\def\finthm{\end{theorem}}
\def\debprop{\begin{proposition} \it}
\def\finprop{\end{proposition}}
\def\debcor{\begin{corollary}\it }
\def\fincor{\end{corollary}}
\def\debdef{\begin{definition}\it}
\def\findef{\end{definition}}
\def\debrem{\begin{remark}\em}
\def\finrem{\null\hfill\blackbox\end{remark}}
\def\OO{\mathbb{O}}

\def\RR{\mathbb{R}}

\def \eps{\epsilon}

\def\bfu{\mathbf{u}}
\def\bfv{\mathbf{v}}
\def\bfx{\mathbf{x}}

\def\bfE{\mathbf{E}}
\def\bfB{\mathbf{B}}
\def\bfb{\mathbf{b}}

\newcommand{\converge}[1]{%
\ensuremath{\displaystyle{\smash{\,\mathop{\rightharpoonup}\limits_{\epsilon \to0}^{\star}\,}}}}

\newcommand{\convergebis}[1]{%
\ensuremath{\displaystyle{\smash{\,\mathop{\rightarrow}\limits_{\epsilon \to0}^{#1}\,}}}}

\newcommand{\converges}[1]{%
\ensuremath{\displaystyle{\smash{\,\mathop{\rightharpoonup}\limits_{\sigma \to0}^{#1}\,}}}}

\newcommand{\convergess}[1]{%
\ensuremath{\displaystyle{\smash{\,\mathop{\rightharpoonup}\limits_{\epsilon \to0}^{#1}\,}}}}

\graphicspath{{images/}}

\title[Stiff transport equations and their AP-resolution]{Asymptotic-Preserving scheme for the resolution of evolution equations with stiff transport terms}

\author[B. Fedele, C. Negulescu, S. Possanner]{Baptiste Fedele, Claudia Negulescu, Stefan Possanner}

\address{Universit\'e de Toulouse \& CNRS, UPS, Institut de Math\'ematiques de Toulouse UMR 5219, F-31062 Toulouse, France.}
\email{baptiste.fedele@math.univ-toulouse.fr\\
claudia.negulescu@math.univ-toulouse.fr}
\date{\today}

\begin{document}

\maketitle

\begin{abstract}
\noindent We develop an asymptotic-preserving scheme to solve evolution problems containing stiff transport terms. This scheme is based to a micro-macro decomposition of the unknown, coupled with a stabilization procedure. The numerical method is applied to the Vlasov equation in the gyrokinetic regime and to the Vlasov-Poisson 1D1V equation, which occur in plasma physics. The asymptotic-preserving properties of our procedure permit to study the long-time behavior of these models. In particular, we limit drastically by this method the numerical pollution, appearing in such time asymptotics when using classical numerical schemes.
\end{abstract}

\bigskip


%%%%%%%%%%%%%%%%%%%%%%%%%%%%%%%%%%%%%%%%%%%%%%
\section{Introduction} \label{SEC1}
%%%%%%%%%%%%%%%%%%%%%%%%%%%%%%%%%%%%%%%%%%%%%%
The main objective of this work is to introduce and subsequently investigate an efficient numerical scheme for the resolution of evolution equations containing stiff transport terms, namely
\be \label{ANI_TR}
\partial_t f^\eps + {\mathcal L} f^\eps+{\bfb \over \eps} \cdot \nabla f^\eps =0\,, \quad t \in \RR^+ \,, \,\,\, \bfx \in \Omega \subset \RR^d\,,
\ee
where $\bfb : \RR^+ \times \Omega \rightarrow \RR^d$ is a known (passive, linear transport model) or self-consistently computed (active, nonlinear transport model) vector-field satisfying $\nabla \cdot \bfb =0$\,, and $\mathcal L$ is a given operator (for ex. transport or diffusion operator). The small parameter $\eps \ll 1$ represents the stiffness of the problem and signifies that we have to cope with a very strong vector-field $\bfb$. It brings up the main difficulties in the numerical resolution of \eqref{ANI_TR} and this due to the introduction of multiple scales in the problem. Indeed, the dynamics along the $\bfb$-field is very rapid, as compared to its perpendicular evolution. In the formal limit $\eps \rightarrow 0$, the problem reduces to the constraint 
\be \label{Red}
\bfb \cdot \nabla f^0 =0\,,
\ee
which signifies that the unknown $f^0$ is constant along the field-lines of $\bfb$. However, in general, problem \eqref{Red} does not permit to determine this constant, for example when $\bfb$ has closed field lines. Thus the reduced problem \eqref{Red} is ill-posed, information has been lost while setting formally $\eps =0$ in \eqref{ANI_TR}. This feature is typical for singularly-perturbed problems or multi-scale problems (see \cite{lebris, wein}).\\

The study of multi-scale problems is very arduous from a mathematical as well a numerical point of view. Standard explicit numerical schemes require very small time steps, dependent on the $\eps$-parameter, in order to accurately account for the microscopic information (living at the $\eps$-scale). This procedure, even if accurate, has however the big disadvantage of being numerically very costly in simulation time and memory. Fully implicit schemes or IMEX-schemes are also not of use for $\eps \ll 1$, due to the ill-conditioned limit model. Alternative methodologies are thus required taking into account for the various scales present in the problem. Asymptotic analysis will be one of the mathematical tools used in this paper, permitting to recover the microscopic information lost in the reduced model \eqref{Red} and the numerical scheme presented here is based on such developments.\\

Evolution equations of the type \eqref{ANI_TR} arise often in applications coming from fluid dynamics (see \cite{majda}) and plasma physics (see \cite{CHENF,HM,Ruther}). To mention only some examples, in thermonuclear tokamak plasmas, the evolution of ions is described via the non-dimensional Vlasov ($\eta=0$) or Fokker-Planck ($\eta \neq 0$) equation
\begin{equation} \label{Boltz}
\partial_t f_{i}^{\epsilon} + \bfv \cdot \nabla_{\bfx} f_{i}^{\epsilon} +  \Big(\bfE+ {1 \over \eps}\,\bfv \times \bfB \Big) \cdot \nabla_{\bfv} f_{i}^{\epsilon} =\eta \, \nabla_{\bfv} \cdot \left[ \bfv f_i^\eps + \nabla_{\bfv} f_i^\eps\right]\,,
\end{equation}
where $f_{i}^{\epsilon}(t,\bfx,\bfv)$ represents the ion distribution, dependent on time, space and velocity. This equation is coupled via the electromagnetic fields $(\bfE(t,\bfx),\bfB(t,\bfx))$ to an equation describing the electron evolution. The coupling is done by means of Maxwell's equations or Poisson equation in the electrostatic case.
The magnetic field is very strong in tokamak experiments in the aim to confine the plasma and to render the fusion possible. This feature is translated in \eqref{Boltz} in the magnitude of the scaling parameter $\eps \ll 1$.\\

The second example we shall be interested in here, concerns the long-time asymptotic study of the electron $1D1V$ Vlasov-Poisson system
\be \label{VP1D1V}
\left\{
\begin{array}{l}
\ds \partial_t f_{e} + v \, \partial_x f_{e} - E(t,x) \, \partial_vf_{e} = 0\,, \qquad \forall t \in \RR^+\,, \,\,\, \forall (x,v)\in \Omega \subset \RR^2\\[3mm]
\ds - \partial_{xx} \varphi = 1 -n_{e}\,, \qquad n_{e}(t,x)=\int_{\RR} f_{e}(t,x,v)\, dv\,, \qquad E=-\partial_x \varphi\,.
\end{array}
\right.
\ee
Introducing the field $\bfu:=(v,-E(t,x))^t$ and the stream function $\Psi:={1 \over 2}|v|^2 - \varphi(t,x)$, one has $\bfu=^\perp \! \nabla \Psi$, where $^\perp  \nabla:=(\partial_v, - \partial_x)$. Considering additionally long-time scales, the Vlasov-Poisson system \eqref{VP1D1V} transforms into the nonlinear, coupled system
\be \label{VP1D1V_bis}
\left\{
\begin{array}{l}
\ds \partial_t f_{e}^\eps + {\bfu^\eps \over \eps} \cdot \nabla_{x,v} f_{e}^\eps  = 0\,, \qquad \forall (t,x,v) \in \RR^+ \times \Omega\\[3mm]
\ds - \Delta_{x,v} \Psi^\eps = n_{e}^\eps-2\,, \qquad n_{e}^\eps(t,x)=\int_{\RR} f_{e}^\eps(t,x,v)\, dv\,, \qquad \bfu^\eps=^\perp \! \nabla \Psi^\eps\,.
\end{array}
\right.
\ee

\vspace{0.2cm}
Finally our last example (which shall not be treated in this paper) comes from fluid mechanics : consider the incompressible Euler equations in the long-time scaling, describing a bi-dimensional, inviscid flow  with velocity $\bfu :=(u_1, u_2, 0)$ and pressure $p$
\be \label{EULER_EQ}
\left\{
\begin{array}{rcl}
\ds \eps\,\partial_t \bfu^\eps + (\bfu^\eps \cdot \nabla)\bfu^\eps + \nabla p^\eps&=&0\,, \qquad \forall (t,\bfx) \in \RR^+ \times \Omega\,,\\[3mm]
\ds \nabla \cdot \bfu^\eps&=&0\,.
\end{array}
\right.
\ee
Introducing the vorticity $\omega^\eps:=\nabla \times \bfu^\eps$, the Euler system leads to the following nonlinear, coupled system 
\be \label{VOR}
\left\{
\begin{array}{l}
\ds \partial_t \omega^\eps + {\bfu^\eps \over \eps} \cdot \nabla \omega^\eps=0\,,\\[3mm]
\ds - \Delta \Psi^\eps = \omega^\eps\,, \qquad \bfu^\eps = ^\perp \! \nabla \Psi^\eps\,,
\end{array}
\right.
\ee
constituted of a transport equation for the vorticity, which is self-consistently coupled with a Poisson equation for the determination of the stream-function $\Psi^\eps$, result of the divergence-free constraint of $\bfu^\eps$. Sometimes one can add on the right hand side of the first equation in \eqref{VOR} a small viscosity term $\nu\, \Delta \omega^\eps$, $\nu$ being the reciprocal of the Reynolds number. The new modified equation \eqref{VOR} is coming then from the incompressible Navier-Stokes equations.\\

The goal of this work is now to present and investigate an efficient, uniformly accurate and stable (wrt. $\eps$) numerical scheme for the resolution of the following linear, stiff transport problem
\be \label{ADV_bis}
(V)^\eps\,\,\,
\left\{
\begin{array}{l}
\ds \partial_t f^\eps + {\bfb \over \eps} \cdot \nabla f^\eps =0\,, \qquad \forall  t \in (0,T) \,, \,\,\, \forall \bfx =(x,y) \in \Omega \subset \RR^2\,,\\[2mm]
\ds f^\eps(0,\bfx)=f_{in}(\bfx) \qquad \forall \bfx \in \Omega\,,
\end{array}
\right.
\ee
with given, smooth and time-independent vector-field $\bfb : \Omega \rightarrow \RR^2$, satisfying $\nabla \cdot \bfb =0$. This simplified transport equation contains all the numerical difficulties arising also in the original equation \eqref{ANI_TR}. Given an efficient numerical algorithm for the resolution of \eqref{ADV_bis}, the treatment of the examples mentioned above is straightforward. Indeed, the nonlinear coupling can be treated iteratively, as shall be shown in Section \ref{SEC6} for the Vlasov-Poisson test case, and the discretization of the general not-stiff term ${\mathcal L} f^\eps$ of \eqref{ANI_TR} can be done via standard schemes suited for this particular operator. The scheme we propose in this paper shall be verified and validated in two test cases, corresponding firstly to a simplified version of the gyrokinetic scaling \eqref{Boltz}, containing only the stiff magnetic term ${1 \over \eps}\,(\bfv \times \bfB) \cdot \nabla_{\bfv} f_{i}^{\epsilon} $, and to the long-time asymptotics of the Vlasov-Poisson system \eqref{VP1D1V_bis}.\\

Due to the divergence constraint of $\bfb$, there exists a stream-function $\Psi$ such that $\bfb=^\perp \!\! \nabla \Psi$. Using the Poisson-bracket notation for two functions $\chi,\theta$, namely
$$
\{\chi,\theta\}:=\partial_x \chi\, \partial_y \theta - \partial_y \chi\, \partial_x \theta= \nabla \chi \cdot ^\perp \!\! \nabla \theta\,,
$$
the transport equation \eqref{ADV_bis} can be simply rewritten as
\be \label{ADV_duo}
(V)^\eps\,\,\,\left\{
\begin{array}{l}
\ds \partial_t f^\eps + {1 \over \eps} \{f^\eps, \Psi\} =0\,, \quad t \in \RR^+ \,, \,\,\, \bfx \in \Omega \subset \RR^2\,,\\[2mm]
\ds f^\eps(0,\cdot)=f_{in}\,,
\end{array}
\right.
\ee
and shall be completed with adequate boundary conditions, depending on the shape of the domain $\Omega$ and on the vector-field $\bfb$. In order to recover the examples presented above, we shall investigate two different cases, resumed in the following Hypothesis.\\

{\bf Hypothesis A} : {\it The domain $\Omega$ will be either the whole $\RR^2$ (case  \eqref{Boltz}) or an infinite strip $(L_1,L_2)\times \RR$ (case \eqref{VP1D1V_bis}) of the $(x,y)$-plane. In the second case, we shall assume periodic boundary conditions in $x$ and the field $\bfb$ is supposed to be also periodic in $x$.}\\

Our main goal is to understand in detail the features of the Asymptotic-Preserving scheme we intent to propose for the resolution of \eqref{ADV_bis}. In particular, we aim to:
\begin{itemize}
\item design a simple and robust numerical scheme, working on a Cartesian grid;
\item design a scheme which enjoys the Asymptotic-Preserving properties (AP-scheme), in the sense that it has to be uniformly stable and accurate wrt. $\eps$;
\item give a detailed explanation why the proposed AP-method behaves better than standard methods (explicit, implicit, IMEX);
\item design a scheme which has to be simply ``generalizable'' to more dimensions and various advection fields.  
\end{itemize}
Let us underline at this point one important fact. We were interested in designing a scheme working on a Cartesian grid.  One can imagine that for stiff problems of the type \eqref{ADV_bis} (or more generally \eqref{ANI_TR}), it could be better to adapt the coordinate system, choosing field-aligned variables, and transforming thus the problem into an evolution problem with a strong anisotropy aligned with one coordinate axis, problem which is much simpler to solve (via IMEX schemes for ex., see \cite{FN}). Our aim however was rightly to avoid a coordinate transformation and to design a simple scheme based on a Cartesian grid. The advantage is that the numerical treatment is very simple, the disadvantage will be mentioned in Section \ref{SEC2}, namely the introduction of a second, auxiliary unknown. Our scheme is hence an alternative to the existing schemes for such evolution problems with stiff transport terms, and marries at the same time simplicity and Asymptotic-Preserving property.\\ 

% Indeed, this strategy has several advantages, and is extensively used in applications \cite{}.

Several AP-schemes were designed in the last years for various types of problems, including anisotropic elliptic \cite{DLNN,DDLNN} or parabolic \cite{MN} equations, Vlasov equation in the hydrodynamic regime \cite{filbet} or drift-diffusion regime \cite{crestetto2,klar}, Vlasov equation in the high-field limit \cite{crou3, jin}, Euler equation in the low-Mach regime \cite{Degond,dimarco}. Briefly, an AP-scheme is a numerical scheme specially designed for singularly-perturbed problems $P^{\eps}$, containing some small parameter $\eps \ll 1$, and which enjoy the following properties (see commutative diagram \ref{diag1}):
\begin{itemize}
\item for fixed $\eps >0$, the AP-scheme, denoted in this diagram $P^{\eps,h}$, is a consistent
  discretization of  the continuous problem $P^{\eps}$, where
  $h=(\Delta t, \Delta \bfx)$;
\item the stability condition is independent of $\eps$;
\item for fixed discretization parameters $h=(\Delta t, \Delta \bfx)$,
  the AP-scheme $P^{\eps,h}$ provides in the limit $\eps \rightarrow
  0$ a consistent discretization of the limit problem $P^{0}$.
\end{itemize}
\vspace{0.3cm}
\begin{figure}[!ht]
  \centering
  \psfrag{T1}[][][1.]{$P^{\varepsilon,h}$}
  \psfrag{T2}[][][1.]{$P^{\varepsilon}$}
  \psfrag{T4}[][][1.]{$P^{0,h}$}
  \psfrag{T3}[][][1.]{$P^{0}$}
  \psfrag{L1}[][][1.]{$\varepsilon\rightarrow 0$}
  \psfrag{H1}[][][1.]{$h\rightarrow 0$}
  \includegraphics[width=0.5\textwidth]{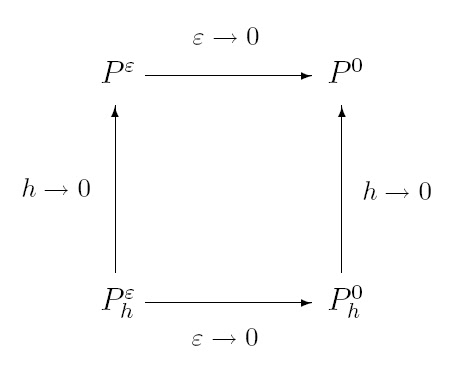}
  \caption{Properties of AP-schemes.}
  \label{diag1}
\end{figure}

One can put these schemes in the category of multi-scale numerical methods. At the end, let us also remark here that standard schemes for the resolution of \eqref{ANI_TR} exist in literature, based on Galerkin methods \cite{gamba}, IMEX-techniques \cite{BFR} or spectral methods \cite{FP}.\\

The outline of this paper is the following. In Section \ref{SEC2}, the asymptotic-preserving reformulation of the singularly-perturbed advection equation \eqref{ADV_bis} is detailed. The reformulation is based on a micro-macro decomposition and a stabilization method. The Section \ref{SEC3} deals with some mathematical aspects of the AP-reformulation, in order to show the well-posedness of this latter. Section \ref{SEC4} presents the numerical discretization of our asymptotic-preserving procedure. Section \ref{SEC5} focuses on a mathematical test case and its numerical resolution by our AP-scheme. In particular, we study deeply the stabilization of the numerical scheme. Finally Section \ref{SEC6} is dedicated to the numerical resolution of the Vlasov-Poisson 1D1V system. We focus notably in this part on the long-time behavior of the two-stream instability, leading to BGK-like equilibria. The last section concludes the paper with some remarks and perspectives.

%%%%%%%%%%%%%%%%%%%%%%%%%%%%%%%%%%%%%%%%%%%%%%%
\section{Asymptotic-Preserving reformulation} \label{SEC2}
%%%%%%%%%%%%%%%%%%%%%%%%%%%%%%%%%%%%%%%%%%%%%%%
We shall present in this section an AP-reformulation of the singularly-perturbed advection problem 
\eqref{ADV_bis} completed with adequate boundary conditions, explicited in {\it Hypothesis A}, scheme which shall behave better (regularly) in the limit $\eps \rightarrow 0$. For this, the well-posed limit-model has firstly to be identified by investigating the asymptotic behaviour of the solutions $f^\eps$, as $\eps \ll 1$. We underline here that $\bfb$ is time-independent in the following, if not explicitly mentioned, as in Section \ref{SEC6}.
%%%%%%%%%%%%%%%%%%%%
\subsection{Identification of the limit model}\label{SUB21}
%%%%%%%%%%%%%%%%%%%%
As mentioned in the introduction, letting formally $\eps \rightarrow 0$ in \eqref{ADV_bis}, leads to an ill-posed problem, which does not permit to compute in a unique manner the limit solution $f^0(t,x,y)$. The only information we get is that $f^0$ is constant along the field-lines of $\bfb$. 

\bigskip

\noindent In order to establish the limit model $(V)^0$ corresponding to \eqref{ADV_bis}, let us suppose that $f^\eps$ admits the following Hilbert expansion
\be \label{HI}
f^{\eps} = f^0 + \eps f^1 + \eps ^2 f^2 + ... \; .
\ee
Injecting this Ansatz in \eqref{ADV_bis} leads to the infinite hierarchy of equations

\begin{align}
\bfb \cdot \nabla f^0 &=0, \label{project1} \\
\partial_t f^0 + \bfb \cdot \nabla f^1 &=0, \label{project2} \\ 
\partial_t f^1 + \bfb \cdot \nabla f^2 &=0, \\
& \hspace{-1cm} \vdots \nonumber
\end{align}

\noindent Equation \eqref{project1} reveals that $f^0$ belongs to the kernel of the dominant operator $\mathcal{T} := \mathbf{b} \cdot \nabla$. However, this information is not enough to determine completely $f^0$. It is necessary to use the next equation \eqref{project2}, to get the missing information. To eliminate $f^1$ from this equation, one projects \eqref{project2} on the kernel of $\mathcal{T}$. This projection is nothing else than the average of a quantity $q$ along the field lines of $\bfb$ and will be denoted by $\langle q \rangle$. Briefly, if $Z(s;\bfx)$ is the characteristic flow associated to the field $\bfb$, {\it i.e.}
$$
\left\{
\begin{array}{l}
\ds {d \over ds} Z(s;\bfx)=\bfb(Z(s;\bfx))\,,\\[2mm]
\ds Z(0;\bfx)=\bfx\,,
\end{array}
\right.
$$
the average of a function $q \in L^2(\Omega)$ over the field lines of $\bfb$ is defined as
$$
\langle q \rangle(\bfx):= \lim_{ S \rightarrow \infty} {1 \over S} \int_0^S q(Z(s;\bfx))\, ds\quad \forall \bfx \in \Omega\,.
$$
One can show (after some hypothesis on the regularity of $\bfb$, see \cite{mb}) that $\langle \cdot \rangle$ is a well-defined application, furthermore that $\langle q \rangle$ is constant along the field lines of $\bfb$ and $\langle \bfb \cdot \nabla q \rangle=0$. The above mentioned procedure permits then to obtain a well-posed limit model for $f^0$. We already know that $f^0$ belongs to the kernel of $\mathcal{T}$, meaning $f^0=\langle f^0 \rangle$, such that the limit model $(V)^0$ writes 
\be \label{ADV_bis_0}
(V)^0\,\,\, \left\{
\begin{array}{l}
\ds \partial_t f^0=0, \qquad \bfb  \cdot \nabla f^0 =0\,, \qquad \forall (t,\bfx) \in (0,T) \times \Omega\,,\\[3mm]
\ds f^0(0,\bfx)=\langle f_{in}(\bfx) \rangle \in \ker \mathcal{T} \,, \qquad \bfx \in \Omega\,.
\end{array}
\right.
\ee

\noindent The following theorem proves rigorously the convergence of the solution $f^\eps$ of \eqref{ADV_bis} towards the solution $f^0$ of the limit model \eqref{ADV_bis_0}, as $\eps \rightarrow 0$.
\begin{theorem} \cite{mb}
Consider a subset $\Omega$ of $\mathbb{R}^2$ satisfying Hypothesis A. Assume $\mathbf{b} \in W_{loc}^{1,\infty}(\RR^2)$ (where in the case $\Omega$ is a strip, we extend $\bfb$ periodically to the whole $\RR^2$) satisfying $\nabla \cdot \bfb =0$ as well as the growth condition
$$
\exists C>0 \,\,\, \textrm{s.t.}\,\,\, |\bfb(\bfx)| \le C \, (1+ |\bfx|)\,, \quad \forall \bfx \in \Omega\,.
$$
Suppose furthermore that $f_{in} \in L^2(\Omega)$. Then \eqref{ADV_bis} resp. \eqref{ADV_bis_0} admit unique weak solutions $f^\eps,f^0 \in L^{\infty}(0,T ; L^2(\Omega))$ and one has $f^\eps \converge{}f^0$, weakly-$\star$ in $L^{\infty}(0,T ; L^2(\Omega))$.\\
If the initial conditions are well prepared in the sense that $f^\eps_{in}$ is smooth enough and satisfies $f^\eps_{in} \convergebis{} f^0_{in} \in \ker {\mathcal T}$ in $L^2(\Omega)$, then one has even  $f^\eps \convergebis{} f^0$ in $L^{\infty}(0,T ; L^2(\Omega))$.
\end{theorem}
%%%%%%%%%%%%%%%%%%
\subsection{Micro-Macro reformulation}
%%%%%%%%%%%%%%%%%%

The design of a multiscale numerical procedure for the resolution of problem \eqref{ADV_bis} is now inspired by the asymptotic study performed in Section \ref{SUB21}. To recover the missing microscopic information in the reduced model \eqref{Red}, we shall decompose $f^\eps$ into a macroscopic and a microscopic part, as follows
\be \label{DECOMP}
f^\eps=p^\eps + \eps\, q^\eps\,, \quad \textrm{with} \,\,\, \bfb  \cdot \nabla f^\eps= \eps \,{\bfb}\cdot \nabla q^\eps\,.
\ee
This signifies that $p^\eps$ belongs to the kernel of the dominant operator $\mathcal{T} = \bfb  \cdot \nabla$ and is considered as the macroscopic part. This decomposition is not unique as one has still to fix the values of $p^\eps$ or equivalently $q^\eps$ on the field-lines, fact which shall be done in the next subsections.\\
Plugging for the moment \eqref{DECOMP} into \eqref{ADV_bis} leads to the following augmented system for the two unknowns $(f^\eps,q^\eps)$
\be \label{ADV_ref}
\left\{
\begin{array}{l}
\ds \partial_t f^\eps + {\bfb} \cdot \nabla q^\eps =0\,, \qquad \forall (t,\bfx)\in (0,T) \times \Omega\,,\\[3mm]
\ds \bfb  \cdot \nabla f^\eps= \eps\, \bfb \cdot \nabla q^\eps\,, \qquad \forall (t,\bfx)\in (0,T) \times \Omega\,,
\end{array}
\right.
\ee
associated with the initial condition $f^\eps(0,\cdot)=f_{in}$ and adequate boundary conditions ({\it Hypothesis A}). Now several possibilities are conceivable to fix the values of $q^\eps$ on the field-lines, fact which is nothing else than rendering the decomposition \eqref{DECOMP} unique. Let us observe here that the values of $q^\eps$ on these lines are of no importance for the computation of our physical unknown $f^\eps$, as only ${\bfb} \cdot \nabla q^\eps$ is occurring in the system \eqref{ADV_ref}. Thus any arbitrary choice could do the work.

%%%%%%%%%%%%%%%%%%%%%%%%%%%%%%%%%%%%%%%%%%%%%%%
\subsection{Zero mean value} \label{SEC21}
%%%%%%%%%%%%%%%%%%%%%%%%%%%%%%%%%%%%%%%%%%%%%%%
From a purely mathematical point of view, one first idea is to fix the average of $q^\eps$ along the field lines of $\bfb$, by enforcing zero mean, {\it i.e.}
\be \label{mean}
\langle q^\eps \rangle=0\,.
\ee
Imposing \eqref{mean} can be done by slightly changing the system, adding an additional ``subtle'' term, with $\sigma \in \RR$ an arbitrary constant, namely
\be \label{ADV_mean}
\left\{
\begin{array}{l}
\ds \partial_t f^\eps + {\bfb} \cdot \nabla q^\eps =0\,, \\[3mm]
\ds \bfb  \cdot \nabla f^\eps= \eps \, \bfb\cdot \nabla q^\eps - \sigma\, \langle q^\eps \rangle\,.
\end{array}
\right.
\ee
Indeed, one can remark immediately that taking the average of the second equation over the field-lines yields automatically the constraint $\langle q^\eps \rangle=0$. The new introduced term is hence a tricky zero, rendering $q^\eps$ unique by fixing its average values along the $\bfb$-lines to zero. One can show now that \eqref{ADV_mean} is completely equivalent to \eqref {ADV_bis}, for each $\eps >0$. Indeed, two ingredients help to prove this equivalence between both formulations. On one hand, for given $f^\eps$ and $\eps >0$, the  equation
$$
\left\{
\begin{array}{l}
\ds {\bfb} \cdot \nabla q^\eps = {1 \over \eps } {\bfb} \cdot \nabla f^\eps,\\[3mm]
\ds \langle q^\eps \rangle =0\,,
\end{array}
\right.
$$
has a unique solution $q^\eps$. On the other hand, the second equation in \eqref{ADV_mean} yields immediately, as mentioned above, the constraint $\langle q^\eps \rangle=0$.\\

This idea is very nice from a mathematical point of view, however, if one is thinking at the numerical implementation, one has to average over the field lines of $\bfb$, in order to discretize the new term $\sigma\, \langle q^\eps \rangle$ in the second equation of \eqref{ADV_mean}. This procedure is rather hard (we are working on Cartesian grids with not-aligned fields $\bfb$) and can introduce moreover $\eps$-dependent error terms in the results. Thus we shall leave this idea behind, and search for a more practical one.
%%%%%%%%%%%%%%%%%%%%%%%%%%%%%%%%%%%%%%%%%%%%%%%
\subsection{Regularization} \label{SEC22}
%%%%%%%%%%%%%%%%%%%%%%%%%%%%%%%%%%%%%%%%%%%%%%%
In order to render $q^{\epsilon}$ unique in \eqref{ADV_ref}, one can imagine to use a regularization technique. Regularization is a very broad field in mathematics, and is devoted to the design and analysis of methods for obtaining stable solutions of ill-posed problems. In particular, the usual regularization technique consists in replacing the ill-posed problem by a nearby (slightly-perturbed) well-posed problem, whose resolution poses no difficulties (uniqueness, stability of the solution). The original solution is recovered only in the limit of vanishing regularization/perturbation parameter. The choice of the perturbation term as well as the strength of the perturbation parameter is essential and constitutes the key point of the method. There is a rich literature on regularization techniques, we refer the interested reader to the references \cite{bochev,brezin,calvetti,engl}.\\

Coming now to our problem, one can regularize \eqref{ADV_ref} either by adding on the left of the second equation a term of the type ``$\alpha\, \partial_t q^\eps$'' or of the form ``$\sigma \, q^\eps$''. These two regularizations permit to get a unique $q^\eps$, however the corresponding solutions behave very differently in the desired asymptotic limit $\eps \rightarrow 0$. To see this difference, let us simplify by putting formally $\eps=0$ in \eqref{ADV_ref} and take a look at both regularizations, namely
\be \label {REG}
(W)\,\,\, \left\{
\begin{array}{l}
\ds \partial_t f^0 + \partial_x q^0 =0\,, \\[3mm]
\ds \alpha\, \partial_t q^0 + \partial_x f^0 =0\,,
\end{array}
\right.
\quad \quad \qquad 
(P)\,\,\,\left\{
\begin{array}{l}
\ds \partial_t f^0 + \partial_x q^0 =0\,, \\[3mm]
\ds \partial_x f^0 + \sigma q^0 =0\,,
\end{array}
\right.
\ee
where the constants $\alpha >0$ resp. $\sigma >0$ have to be sufficiently small in order not to perturb too much the original problem. Now, one can eliminate in both systems the auxiliary unknown $q^0$ and get an equation involving only $f^0$, which reads
$$
(W)\,\,\, \partial_{tt} f^0 - {1 \over \alpha} \partial_{xx} f^0=0\,; \qquad \quad (P)\,\,\, \partial_{t} f^0 - {1 \over \sigma} \partial_{xx} f^0=0\,.
$$
As one can observe,  regularizing \eqref{ADV_ref} by adding a term of the form ``$\alpha \partial_t q^\eps$'' leads in the limit $\eps \rightarrow 0$ to a wave-equation, whereas the regularization by adding ``$\sigma q^\eps$'' leads to a parabolic equation.\\

Which one is better suited for our anisotropic transport problem can be understood by remembering the asymptotic behaviour of the unique solution $f^\eps$ of the original advection equation \eqref{ADV_bis} as $\eps$ becomes smaller and smaller. As shown in Section \ref{SUB21}, one gets in the (weak) limit $\eps \rightarrow 0$ a function $f^0$ which is constant along the field lines of $\bfb$. This gives us a hint that the regularization $(P)$ is better suited for our problem, as the corresponding limit problem is a diffusion problem, with very strong diffusivity along the field lines. Hence one is expecting to get a better approximation of $f^0$ via $(P)$ than via a wave-equation of the form $(W)$, which describes rather a very rapid wave-motion.
%%%%%%%%%%%%%%%%%%%%%%%%%%%%%%%%%%%%%%%%%%%%%%%
\subsection{The stabilized AP-reformulation} \label{SEC23}
%%%%%%%%%%%%%%%%%%%%%%%%%%%%%%%%%%%%%%%%%%%%%%%
To summarize, our Asymptotic-Preserving scheme for an efficient resolution of the anisotropic transport equation \eqref{ADV_bis} is based on the resolution of the following reformulated system
\be \label{ADV_AP}
(MM)_\eps^{\sigma}\,\,\, \left\{
\begin{array}{l}
\ds \partial_t f^{\eps,\sigma} + {\bfb} \cdot \nabla q^{\eps,\sigma} =0\,, \qquad \forall (t,\bfx) \in (0,T) \times \Omega\,, \\[3mm]
\ds \bfb  \cdot \nabla f^{\eps,\sigma}= \eps \, \bfb \cdot \nabla q^{\eps,\sigma} - \sigma\, q^{\eps,\sigma}\,, \qquad \forall (t,\bfx) \in (0,T) \times \Omega\,,
\end{array}
\right.
\ee
with $\sigma >0$ a small parameter to be fixed later on.
This system is completed by an initial condition $f^{\eps,\sigma}(0,\cdot)=f_{in}(\cdot)$ and adequate boundary conditions ({\it Hypothesis A}). Let us underline here the difference between \eqref{ADV_mean} and \eqref{ADV_AP}. Both procedures are fixing the value of the auxiliary variable $q$ on the field lines of $\bfb$ by imposing zero mean $\langle q \rangle =0$. To see this in \eqref{ADV_AP}, it is enough to take the average $\langle \cdot \rangle$ of the second equation.
However, while \eqref{ADV_mean} is completely equivalent to the starting model \eqref{ADV_bis}, the system \eqref{ADV_AP} introduces an error, as the supplementary term we introduced, $\sigma q^{\eps,\sigma}$, is no more zero but contains also the non-zero fluctuation part of $q^{\eps,\sigma}$. The big advantage of \eqref{ADV_AP} with respect to \eqref{ADV_mean} is that this time we have no more to discretize the average procedure $\langle \cdot \rangle$.
\\ 

The $\eps$-regularity of the system \eqref{ADV_AP} allows now to pass directly to the $\eps \rightarrow 0$ limit in \eqref{ADV_AP} to get the corresponding limit model, {\it i.e.}

\be \label{ADV_AP_LIM}
(MM)_0^{\sigma}\,\,\, \left\{
\begin{array}{l}
\ds \partial_t f^{0,\sigma} + {\bfb} \cdot \nabla q^{0,\sigma} =0\,, \\[3mm]
\ds \bfb  \cdot \nabla f^{0,\sigma}+ \sigma\, q^{0,\sigma}=0\,.
\end{array}
\right.
\ee
Eliminating $q^{0,\sigma}$ from this system, yields the degenerate diffusion equation
\be \label{Diff}
\begin{array}{l}
\ds \partial_t f^{0,\sigma}- { 1 \over \sigma}\, \nabla \cdot \left[ (\bfb \otimes \bfb) \, \nabla f^{0,\sigma} \right]=0\,,
\end{array}
\ee
which shows clearly what the regularization term is doing in the limit $\eps \rightarrow 0$. For future numerical discretizations, it will be more convenient to rewrite these systems by using the Poisson bracket. Introducing the stream function $\Psi$ such that  $\bfb = (\partial_y \Psi, -\partial_x \Psi)=^\perp\! \nabla \Psi$, the previous Micro-Macro system \eqref{ADV_AP}  reads 

\be \label{ADV_AP_poisson}
(MM)_\eps^{\sigma}\,\,\, \left\{
\begin{array}{l}
\ds \partial_t f^{\eps,\sigma} +\{q^{\eps,\sigma},\Psi \} =0\,, \\[3mm]
\ds \{ f^{\eps,\sigma},\Psi \} = \eps\,\{q^{\eps,\sigma},\Psi \}  - \sigma\, q^{\eps,\sigma}\,.
\end{array}
\right.
\ee

%\be \label{ADV_vect}
%\left\{
%\begin{array}{l}
%\ds \mathcal{A}\, \partial_t \mathcal{U}^{\eps,\sigma} + \mathcal{B}\, \partial_x \mathcal{U}^{\eps,\sigma}+ \mathcal{C}\, \partial_y \mathcal{U}^{\eps,\sigma} +\mathcal{D}\, \mathcal{U}^{\eps,\sigma}= 0\,,\\[3mm]
%\ds \mathcal{U}^{\eps,\sigma}(0,x,y)=\mathcal{U}^{\eps,\sigma}_{in}(x,y)\,,
%\end{array}
%\right.
%\ee
%with the matrices defined by
%$$
%\mathcal{A:}=
%\left(
%\begin{array}{cc}
%1&0\\0&0
%\end{array}
%\right)\,, \quad
%\mathcal{B}:=b_1(x,y)\, 
%\left(
%\begin{array}{cc}
%0&1\\1&-\eps
%\end{array}
%\right)\,, \quad
%\mathcal{C}:=
%b_2(x,y)\, \left(
%\begin{array}{cc}
%0&1\\1&-\eps
%\end{array}
%\right)\,, \quad
%\mathcal{D}:=
%\left(
%\begin{array}{cc}
%0&0\\0&\sigma
%\end{array}
%\right)\,.
%$$
Before introducing the numerical discretization of the AP-reformulation \eqref{ADV_AP_poisson}, let us first analyze the existence and uniqueness of a solution $(f^{\eps,\sigma},q^{\eps,\sigma})$ as well as the asymptotic properties of this solution, when $\epsilon \to 0$, $\sigma \rightarrow 0$, {\it etc}.
%%%%%%%%%%%%%%%%%%%%%%%%%%%%%%%%%%%%%%%%%%%%%%%
\section{Some mathematical observations} \label{SEC3}
%%%%%%%%%%%%%%%%%%%%%%%%%%%%%%%%%%%%%%%%%%%%%%
The rigorous mathematical study of the existence and uniqueness of a solution to the AP-reformulation \eqref{ADV_AP} along with the rigorous study of its limit towards \eqref{ADV_AP_LIM} is a delicate question and shall be treated in a supplementary, more mathematical work. To give however some ideas about the well-posedness of this model, and underline the difficulties of the mathematical study, we shall concentrate in this paper only on the study of the implicit time semi-discretization of \eqref{ADV_AP}, namely
\be \label{ADV_AP_h}
(MM)_{\eps}^{\sigma,n}\,\,\, \left\{
\begin{array}{l}
\ds f^{\eps,\sigma,n+1} + \Delta t \;  {\bfb} \cdot \nabla q^{\eps,\sigma,n+1} =f^{\eps,\sigma,n}\,, \\[3mm]
\ds  \bfb  \cdot \nabla f^{\eps,\sigma,n+1}= \eps \; \bfb \cdot \nabla q^{\eps,\sigma,n+1} - \sigma \, q^{\eps,\sigma,n+1}.\,
\end{array}
\right.
\ee
Here, we discretized the time interval $[0,T]$ with $T>0$, as follows
$$
t^n := n \, \Delta t\,, \quad \quad \Delta t := T/N_t\,, \quad \quad n \in \llbracket0, N_t\rrbracket\,, \quad N_t \in \mathbb{N}\,,
$$
and denoted by $f^{\eps,\sigma,n}$ an approximation of $f^{\eps,\sigma}(t^n,\cdot)$.
This system is associated with boundary conditions for $f^{\eps,\sigma,n+1}$ and $q^{\eps,\sigma,n+1}$, following Hypothesis A. To be more precise, for a normed space $X$, we shall denote by $X_{bc}$ the space of all the functions of $X$, satisfying the boundary conditions precised in Hypothesis A.\\

Let us now specify the mathematical framework of problem \eqref{ADV_AP_h}. For this, choose firstly the Hilbert spaces $V = L^2(\Omega)$ and $Q = \left \{ v \in L^2_{bc}(\Omega), \; \; \bfb \cdot \nabla v \in L^2(\Omega)  \right \}$, associated with the standard $L^2$ scalar-product for $V$ and $(u,v)_{Q}:=(u,v)_{L^2}+( \mathbf{b} \cdot \nabla u,\mathbf{b} \cdot \nabla v)_{L^2}$ for $Q$. We introduce now the following bi-linear forms $\mathcal{A}$, $\mathcal{B}$, $\mathcal{C}_{\eps,\sigma}$:
\begin{equation} \label{Bil}
\begin{array}{llll}
\ds \mathcal{A}(u,v) &:=& \ds \int_{\Omega} uv d \bfx, &\qquad \ds \mathcal{A}: V \times V \rightarrow \RR\,,  \\[3mm]
\ds \mathcal{B}(v,r) &:=&\ds  \int_{\Omega} v\, (\bfb \cdot{\nabla} r) d \bfx, &\qquad\ds \mathcal{B}: V \times Q \rightarrow \RR\,,  \\[3mm]
\ds \mathcal{C}_{\eps,\sigma}(r,s) &:=&\ds - \eps \;  \mathcal{B}(s,r) + \; \sigma\, (r,s)_{L^2}, &\qquad\ds \mathcal{C}_{\eps,\sigma}: Q \times Q \rightarrow \RR\,,
\end{array}
\end{equation}
and their associated linear operators $A$, $B$, $C_{\eps,\sigma}$:
$$
A : V \longrightarrow V^{\star}\,, \quad B : V \longrightarrow Q^{\star}\,, \quad C_{\eps,\sigma} : Q \longrightarrow Q^{\star}\,, 
$$
$$
\langle Au,v \rangle_{V^{\star},V}:=\mathcal{A}(u,v)\,, \quad \langle Bv,r \rangle_{Q^{\star},Q}:=\mathcal{B}(v,r)\,, \quad \langle C_{\eps,\sigma}r,s \rangle_{Q^{\star},Q}:=\mathcal{C}_{\eps,\sigma}(r,s) \,.
$$
\begin{remark}
The bi-linear form $\mathcal{B}$ defines also the adjoint linear operator $B^{\star} : Q \longrightarrow V^{\star}$ via $\mathcal{B}(v,r)=\langle Bv,r \rangle_{Q^{\star},Q}=\langle v,B^{\star}r \rangle_{V,V^{\star}}$ for all $(v,r) \in V \times Q$. Observe also that 
$B^{\star} r=  \bfb \cdot \nabla r$ for all $r \in Q$, whereas in the distributional sense $B v= - \bfb \cdot \nabla v$ for all $v \in V$.
\end{remark}
With these definitions, the variational formulation of the previous problem \eqref{ADV_AP_h} writes now : for fixed $\eps \ge 0,\, \sigma >0$, $\Delta t >0$ and $f^{\eps,\sigma,n} \in V^{\star}$, find $(f^{\eps,\sigma,n+1},q^{\eps,\sigma,n+1}) \in V \times Q$, such that :

\begin{equation}\label{VAR_FORM}
\left\{
\begin{array}{ll}
\ds \mathcal{A}(f^{\eps,\sigma,n+1},\theta) + \Delta t \, \mathcal{B}(\theta,q^{\eps,\sigma,n+1}) = \langle f^{\eps,\sigma,n}, \theta \rangle_{V^{\star} \times V}, \quad \forall \theta \in V, \\[3mm]
\ds \mathcal{B}(f^{\eps,\sigma,n+1},\chi) - \mathcal{C}_{\eps,\sigma}(q^{\eps,\sigma,n+1},\chi) = 0, \quad \forall \chi \in Q.
\end{array}
\right.
\end{equation}

To prove the existence and uniqueness of a weak solution to \eqref{VAR_FORM}, we shall need some properties of these operators.
\begin{lemme} Let Hypothesis A be satisfied. Then, $\mathcal{A}(\cdot,\cdot)$, $\mathcal{B}(\cdot,\cdot)$ resp. $\mathcal{C}_{\eps,\sigma}(\cdot,\cdot)$ are continuous bi-linear forms on $V\times V$, $V \times Q$ resp. $Q \times Q$. Furthermore, $\mathcal{A}(\cdot, \cdot)$ is coercive on $V \times V$ and $\mathcal{C}_{\eps,\sigma}(\cdot, \cdot)$ is positive semi-definite on $Q \times Q$.
\end{lemme}

\begin{remark}
Let us remark here that $\mathcal{C}_{\eps,\sigma}(\cdot, \cdot)$ is not coercive on $Q \times Q$. However, as we will see later, this hypothesis is not crucial for both existence and uniqueness of a solution of the variational formulation \eqref{VAR_FORM}. 
\end{remark}

Without any other hypothesis on $\mathcal{B}(\cdot,\cdot)$ (as for example an inf-sup condition) we are now able to present the following existence/uniqueness  result of a solution to  \eqref{VAR_FORM}, and this due to the presence of the regularization term $\sigma q$.
\begin{theorem} 
Let Hypothesis A be satisfied and let $\mathcal{A}(\cdot,\cdot)$, $\mathcal{B}(\cdot,\cdot)$ and $\mathcal{C}_{\eps,\sigma}(\cdot,\cdot)$ be the continuous bi-linear forms defined in \eqref{Bil}. Then, for every $f^{\eps,\sigma,n} \in V^{\star}$ the problem \eqref{VAR_FORM} has for each fixed $\eps \ge 0$, $\sigma >0$, $\Delta t >0$ and $n\in \llbracket 0,N_t \rrbracket$, a unique weak solution $(f^{\eps,\sigma,n+1}, q^{\eps,\sigma,n+1})$ in $V \times Q$.

\noindent  This solution satisfies the following estimates, for all $ \eps \geqslant 0,\,\, \sigma >0,\,\, , \Delta t >0,\,\, n\in \llbracket 0,N_t \rrbracket$:
\be \label{thm_f}
||f^{\eps,\sigma,n+1}||_V  \le   ||f^{\eps,\sigma,n}||_{V^{\star}}\,,
\ee
\be \label{thm_q}
||q^{\eps,\sigma,n+1}||_V^2 \le \frac{1}{\sigma\, \Delta t}\,||f^{\eps,\sigma,n}||^2_{V^{\star}}\,, \quad ||\bfb \cdot \nabla q^{\eps,\sigma,n+1}||_{V}\le {2 \over \Delta t}\,||f^{\eps,\sigma,n}||_{V^{\star}}\,.
\ee
\end{theorem}

\begin{proof} 
Due to the lack of coercivity of $\mathcal{C}_{\eps,\sigma}$, we shall start by considering the regularized problem : for $\alpha>0$, find $f^{\eps,\sigma,n+1}_{\alpha} \in V$ and $q^{\eps,\sigma,n+1}_{\alpha} \in Q$, such that :

\begin{equation}\label{VAR_FORM_R}
\left\{
\begin{array}{ll}
\ds \mathcal{A}(f^{\eps,\sigma,n+1}_\alpha,\theta) + \Delta t\,\mathcal{B}(\theta,q^{\eps,\sigma,n+1}_\alpha) = \langle f^{\eps,\sigma,n}, \theta \rangle_{V^{\star} \times V}, \quad \forall \theta \in V, \\[3mm]
\ds \mathcal{B}(f^{\eps,\sigma,n+1}_\alpha,\chi) - \alpha(q^{\eps,\sigma,n+1}_\alpha,\chi)_{Q} - \mathcal{C}_{\eps,\sigma}(q^{\eps,\sigma,n+1}_\alpha,\chi) = 0, \quad \forall \chi \in Q.
\end{array}
\right.
\end{equation}
Multiplying the second equation by $\Delta t$ and subtracting both equations, one can show via Lax-Milgram theorem that \eqref{VAR_FORM_R} admits, for each fixed $\alpha>0$, $\eps \ge 0$, $\sigma >0$, $\Delta t >0$ and $n\in \llbracket 0,N_t \rrbracket$, a unique weak solution. Indeed, thanks to the term $\alpha(q^{\eps,\sigma,n+1}_\alpha,\chi)_{Q}$ in the second equation of \eqref{VAR_FORM_R}, the regularized operator $\mathcal{C}_{\eps,\sigma}(\cdot, \cdot) + \alpha(\cdot,\cdot)_{Q}$ is now coercive on $Q \times Q$.\\ 

Our aim is now to bound $f^{\eps,\sigma,n+1}_\alpha$ and $q^{\eps,\sigma,n+1}_\alpha$ uniformly in $\alpha$. Then, passing to the limit $\alpha \to 0$ in \eqref{VAR_FORM_R} would permit to conclude about both existence and uniqueness of a weak solution of \eqref{VAR_FORM}.\\

First, let us choose $\theta = f_{\alpha}^{\eps,\sigma,n+1}$ and $\chi= q_{\alpha}^{\eps,\sigma,n+1}$ in \eqref{VAR_FORM_R}, multiply the second equation with $\Delta t$ and subtract both equations to get
$$
 \mathcal{A}(f_{\alpha}^{\eps,\sigma,n+1},f_{\alpha}^{\eps,\sigma,n+1}) +\alpha\, \Delta t\, ||q_{\alpha}^{\eps,\sigma,n+1}||^2_Q + \Delta t\,\mathcal{C}_{\eps,\sigma}(q_{\alpha}^{\eps,\sigma,n+1},q_{\alpha}^{\eps,\sigma,n+1}) = \langle f^{\eps,\sigma,n}, f_{\alpha}^{\eps,\sigma,n+1} \rangle_{V^{\star} \times V}.
$$
By using the coercivity of $\mathcal{A}(\cdot,\cdot)$ and the fact that $\mathcal{C}_{\eps,\sigma}(q,q)= \sigma ||q||^2_{L^2}$, we get
$$
 ||f_{\alpha}^{\eps,\sigma,n+1}||^2_V + \alpha\, \Delta t\, ||q_{\alpha}^{\eps,\sigma,n+1}||^2_Q  + \Delta t\, \sigma \, ||q_{\alpha}^{\eps,\sigma,n+1}||^2_V   \le || f^{\eps,\sigma,n}||_{V^*}\,|| f_{\alpha}^{\eps,\sigma,n+1}||_V\,,
$$
leading to
\be \label{est_f}
||f_{\alpha}^{\eps,\sigma,n+1}||_V  \le ||f^{\eps,\sigma,n}||_{V^{\star}}\,, \qquad ||q_{\alpha}^{\eps,\sigma,n+1}||^2_V \le {1 \over \Delta t\, \sigma} \, ||f^{\eps,\sigma,n}||_{V^{\star}}^2\,.
\ee
Now, thanks to the first equation of \eqref{VAR_FORM_R}, {\it i.e.}
$$
\Delta t\, \mathcal{B}(\theta, q^{\eps,\sigma,n+1}_{\alpha}) = - \mathcal{A}(f^{\eps,\sigma,n+1}_\alpha,\theta) + \langle f^{\eps,\sigma,n}, \theta \rangle_{V^* \times V}, \quad \forall \theta \in V\,,
$$
we have
\be \label{est_q}
||\bfb \cdot \nabla q^{\eps,\sigma,n+1}_{\alpha}||_{L^2(\Omega)}= \sup_{\theta \in L^2(\Omega)} {(\bfb \cdot \nabla q^{\eps,\sigma,n+1}_{\alpha},\theta)_{L^2(\Omega)} \over ||\theta||_{L^2(\Omega)}}\le {2 \over\Delta t}\,||f^{\eps,\sigma,n}||_{V^{\star}}\,.
\ee

The estimates \eqref{est_f} as well as \eqref{est_q} are independent on $\alpha>0$ and thus one can extract weakly convergent subsequences and pass to the limit $\alpha \rightarrow 0$ in the variational formulation \eqref{VAR_FORM_R} to conclude the proof.
\end{proof}

\begin{remark}
One can also observe from the $\eps$-independent estimates \eqref{thm_f}-\eqref{thm_q} that up to extracting a subsequence of $\{f^{\eps,\sigma,n+1},q^{\eps,\sigma,n+1} \}_{\eps >0}$ we have the $\eps$-convergences
$$
 f^{\eps,\sigma,n+1} \convergess{} f^{0,\sigma,n+1} \quad \text{in} \;  V\,, \qquad q^{\eps,\sigma,n+1} \convergess{} q^{0,\sigma,n+1} \quad \text{in} \;  Q\,,
$$
which underlines the fact that the AP-reformulation \eqref{ADV_AP} is regular and tends towards the limit-model  \eqref{ADV_AP_LIM} as $\eps$ goes to zero.\\
The $\sigma$-convergences are more delicate. One has only
$$
 f^{\eps,\sigma,n+1} \converges{} f^{\epsilon,0,n+1} \quad \text{in} \;  V\,,
$$
however there is no convergence for $q^{\eps,\sigma,n+1}$. This again was to be expected as in the limit $\sigma \rightarrow 0$ one looses the uniqueness of $q$, the term $\sigma q$ being a regularization term.
\end{remark}

%%%%%%%%%%%%%%%%%%%%%%%%%%%%%%%%%%%%%%%%%%%%%%%
%%%%%%%%%%%%%%%%%%%%%%%%%%%%%%%%%%%%%%%%%%%%%%%
\section{Numerical discretization} \label{SEC4}
%%%%%%%%%%%%%%%%%%%%%%%%
%%%%%%%%%%%%%%%%%%%%%%%%%%%%%%%%%%%%%%%%%%%%%%%
%%%%%%%%%%%%%%%%%%%%%%%%%%%%%%%%%%%%%%%%%%%%%%%

Let us come now to the numerical part of our work and introduce here a numerical discretization of our reformulated system:

\be \label{NR}
(MM)_\eps^{\sigma}\,\,\, \left\{
\begin{array}{l}
\ds \partial_t f^{\eps,\sigma} + \{ q^{\eps,\sigma}, \Psi^{\eps,\sigma} \} =0\,, \\[3mm]
\ds \{ f^{\eps,\sigma}, \Psi^{\eps,\sigma} \}= \eps \, \{ q^{\eps,\sigma}, \Psi^{\eps,\sigma} \} - \sigma\, q^{\eps,\sigma}\,,
\end{array}
\right.
\ee

\noindent where the stream function $\Psi^{\eps,\sigma}$ is supposed to be given in this section, linked to the given vector field ${\bfb}$ through ${\bfb}:=\nabla^{\bot} \Psi^{\eps,\sigma}$.

%%%%%%%%%%%%%%%%%%%%%%%%%%%%%%%%%%%%%%%%%%%%%%%
\subsection{Discretization parameters} \label{SEC41}
%%%%%%%%%%%%%%%%%%%%%%%%%%%%%%%%%%%%%%%%%%%%%%%

In what follows, one assumes a bounded simulation domain $\Omega_S := [-L_x, L_x] \times [-L_y, L_y]$\,. Concerning the time interval $[0,T]$\,, $T>0$\,, we employ the discretization:

\begin{equation*}
t^n := n \, \Delta t\,, \quad \quad \Delta t := T/N_t\,, \quad \quad n \in \llbracket0, N_t\rrbracket\,, \quad N_t \in \mathbb{N}\,.
\end{equation*}

\noindent Similarly, let us supply the domain $\Omega_S$ with a uniform spatial discretization:
\begin{equation*}
x_i := (i-1) \Delta x -L_x\,, \quad  y_j := (j-1) \Delta y -L_y\,, \quad  \Delta x := 2\, L_x/N_x\,,  \quad \Delta y := 2\, L_y/N_y\,,
\end{equation*}
\noindent with $i \in \llbracket1, N_x+1\rrbracket$\,, and $j \in \llbracket1,N_y+1\rrbracket$\,. For any function $f : [0, T] \times \Omega_S \to \mathbb{R}$\,, $f_{i,j}^n$ refers to the numerical approximation of $f(t^n, x_i,y_j)$, and $f_h^n$ shall simply denote the discrete grid-function $(f^n_{i,j})_{i,j}$\,.\\

The domain $\Omega_S$ is a truncation of the real physical domain $\Omega=\RR^2$ or $\Omega=(L_1,L_2) \times \RR$. To be close to the physical reality, we took in our simulations a sufficiently large bounded domain $\Omega_S$ and supposed that the distribution function $f^{\eps,\sigma}$ is vanishing on the truncated infinite boundary, whereas on the other boundary, periodic boundary conditions are imposed. To be more precise, we imposed for the truncation of the physical domain $\Omega=\RR^2$ that
$$
f_{ij}^n=0\,\, \textrm{for}\,\, i=1;\, j=1; \, i=N_x+1;\, j=N_y+1\,,
$$
whereas for the truncation of the physical domain $\Omega=(L_1,L_2) \times \RR$ we imposed
$$
f_{ij}^n=0\,\,\, \forall i\,\, \textrm{and}\,\, j=1; \, j=N_y+1\,\, \quad \textrm{as well as} \quad f_{1,j}^n=f_{N_x+1,j}^n\,\,\, \forall j\,.
$$

%%%%%%%%%%%%%%%%%%%%%%%%%%%%%%%%%%%%%%%%%%%%%%%
\subsection{Space semi-discretization} \label{SEC42}
%%%%%%%%%%%%%%%%%%%%%%%%%%%%%%%%%%%%%%%%%%%%%%%
For the Poisson brackets appearing in the Micro-Macro reformulation \eqref{NR}, let us adopt the second order Arakawa discretization \cite{ara}. For two functions $u, v : \Omega_S \rightarrow \mathbb{R}$, the discrete version of the Poisson bracket $\{u,v\}$ calculated at the point $(x_i,y_j)$ is expressed by:

\begin{align*}
[u_h,v_h]_{i,j}:= \frac{1}{12 \Delta x \Delta y} \Big(u_{i+1,j}\,\mathcal{A}_{i,j} +u_{i-1,j}\,\mathcal{B}_{i,j}+u_{i,j+1}\,\mathcal{C}_{i,j}+u_{i,j-1}\,\mathcal{D}_{i,j} \\[3mm] +u_{i+1,j+1}\,\mathcal{E}_{i,j}+u_{i-1,j-1}\,\mathcal{F}_{i,j}+u_{i-1,j+1}\,\mathcal{G}_{i,j} + u_{i+1,j-1}\,\mathcal{H}_{i,j}\Big)\,.
\end{align*}

\noindent where the coefficients write

\begin{align*}
&\mathcal{A}_{i,j} := v_{i,j+1} - v_{i,j-1} + v_{i+1,j+1} - v_{i+1,j-1}\,, \quad \quad \mathcal{E}_{i,j} := v_{i,j+1} - v_{i+1,j}\,, \\[3mm]
&\mathcal{B}_{i,j} := v_{i,j-1} -v_{i,j+1} - v_{i-1,j+1} + v_{i-1,j-1}\,, \quad \quad \mathcal{F}_{i,j} := v_{i,j-1} -v_{i-1,j}\,, \\[3mm]
&\mathcal{C}_{i,j} := v_{i-1,j} -v_{i+1,j} - v_{i+1,j+1} + v_{i-1,j+1}\,, \quad \quad \mathcal{G}_{i,j} := v_{i-1,j} -v_{i,j+1}\,, \\[3mm]
&\mathcal{D}_{i,j} := v_{i+1,j} - v_{i-1,j} + v_{i+1,j-1} - v_{i-1,j-1}\,, \quad \quad \mathcal{H}_{i,j} := v_{i+1,j} - v_{i,j-1}\,.
\end{align*}

Thus, the semi-discretization in space of the Micro-Macro problem \eqref{NR} reads:

\begin{equation} \label{MM_scheme_space} 
(MM)_{\eps,h}^{\sigma}\,\,\, \left\{
\begin{array}{ll}
\ds \partial_t f_{i,j}^{\eps,\sigma} +[q_h^{\eps,\sigma},\Psi_h^{\eps,\sigma} ]_{i,j} =0\,, \\[3mm]
\ds [f_{h}^{\eps,\sigma},\Psi_{h}^{\eps,\sigma}]_{i,j}-\eps \,  [q_{h}^{\eps,\sigma},\Psi_{h}^{\eps,\sigma}]_{i,j} +\sigma \, q_{i,j}^{\eps,\sigma}\,=0\,.
\end{array}
\right.
\end{equation}

%%%%%%%%%%%%%%%%%%%%%%%%%%%%%%%%%%%%%%%%%%%%%%%
\subsection{Time discretization} \label{SEC43}
%%%%%%%%%%%%%%%%%%%%%%%%%%%%%%%%%%%%%%%%%%%%%%%
We shall use a DIRK (diagonally-implicit Runge-Kutta) approach in order to achieve second-order accuracy in time for the problem \eqref{NR}.  The general form of a RK-method is recalled here for the following equation
$$
\partial_t u  = Lu\,,
$$

\noindent  where $L$ refers to some differential operator. An r-stage Runge-Kutta approach is determined by its Butcher table

\begin{center}
\begin{tabular}{c|ccc}
$c_1$ & $a_{11}$ & $\ldots$ & $a_{1r}$\\
$\vdots$ & $\vdots$ & & $\vdots$ \\
$c_r$ & $a_{r1}$ & $\ldots$ & $a_{rr}$\\
\hline
&$b_1$ & $\ldots$ & $b_r$
\end{tabular}
\end{center}

\noindent For a given $u^{n}$, the subsequent $u^{n+1}$ is defined by the formula
$$
u^{n+1}  = u^n + \Delta t \sum_{j=1}^r b_j u_j\,,
$$
where each $u_i$ is defined by
$$
u_i = u^n + \Delta t \sum_{j=1}^r a_{ij} L u_j\,.
$$
\noindent In the case where $b_j = a_{rj}$ for $j=1,\,...\,,r$\,, then $u^{n+1}$ is equal to the last stage of the method, namely $u_r$. For the Micro-Macro problem \eqref{NR}, we consider the following 2-stage Butcher table

\begin{center}
\begin{tabular}{c|ccc}
$\lambda$ & $\lambda$  & $0$\\
$1$ & $1-\lambda$ & $\lambda$\\
\hline
&$1-\lambda$  & $\lambda$\,
\end{tabular}
\end{center}

\noindent For $\lambda := 1- 1/\sqrt{2}$, the method is $L-stable$. For all $n \in \llbracket0, N_t\rrbracket$, the full discretization of the Micro-Macro problem \eqref{NR} writes now

\begin{equation} \label{MM_scheme_full} 
(MM)_{\eps,h}^{\sigma,n}\,\,\, \left\{
\begin{array}{ll||}
\text{{Stage 1 :}} \\[3mm]
\ds f_{1,i,j}^{\eps,\sigma} +\lambda\,\Delta t\,[q_{1,h}^{\eps,\sigma},\Psi_{h}^{\eps,\sigma} ]_{i,j} =f^{\eps,\sigma,n}_{i,j}\,, \\[3mm]
\ds [f_{1,h}^{\eps,\sigma},\Psi_{h}^{\eps,\sigma}]_{i,j}-\eps \,[q_{1,h}^{\eps,\sigma},\Psi_{h}^{\eps,\sigma}]_{i,j} +\sigma\, q_{1,i,j}^{\eps,\sigma}\,=0\,.
\\[3mm]
\text{{Stage 2 :} } \\[3mm]
\ds f_{2,i,j}^{\eps,\sigma} +\lambda\,\Delta t\,[q_{2,h}^{\eps,\sigma},\Psi_{h}^{\eps,\sigma}]_{i,j}=f^{\eps,\sigma,n}_{i,j}\,+ \frac{1-\lambda}{\lambda}\,( f_{1,i,j}^{\eps,\sigma} -  f_{i,j}^{\eps,\sigma,n})\,, \\[3mm]
\ds [f_{2,h}^{\eps,\sigma},\Psi_{h}^{\eps,\sigma}]_{i,j} -\eps \,[q_{2,h}^{\eps,\sigma},\Psi_{h}^{\eps,\sigma}]_{i,j}  +\sigma \, q_{2,i,j}^{\eps,\sigma}\,=0\, . \\[3mm]  
\text{{Final Stage :} } \\[3mm]
(f_{i,j}^{\eps,\sigma,n+1}, q_{i,j}^{\eps,\sigma,n+1}) = (f_{2,i,j}^{\eps,\sigma}, q_{2,i,j}^{\eps,\sigma})\,.
\end{array}
\right.
\end{equation}

\begin{remark}
In the following, we shall simply call our Micro-Macro scheme \eqref{MM_scheme_full}, obtained with Arakawa space discretization and DIRK time discretization, the (DAMM)-scheme.
\end{remark}

%%%%%%%%%%%%%%%%%%%%%%%%%%%%%%%%%%%%%%%%%%%%%%
%%%%%%%%%%%%%%%%%%%%%%%%%%%%%%%%%%%%%%%%%%%%%%
\section{Verification of the AP-scheme in a mathematical test case} \label{SEC5}%%%%%%
%%%%%%%%%%%%%%%%%%%%%%%%%%%%%%%%%%%%%%%%%%%%%%
%%%%%%%%%%%%%%%%%%%%%%%%%%%%%%%%%%%%%%%%%%%%%%

In this section we investigate the numerical properties of our asymptotic-preserving (DAMM)-scheme \eqref{MM_scheme_full} for the resolution of \eqref{NR} with given field $\mathbf{b}$. The section is devoted to the investigation of a linear case where the stream function $\Psi$ does not depend on $f^{\epsilon}$ and is static. In mind, we have as an application the Vlasov/Fokker-Planck equation with strong given magnetic field \eqref{Boltz}. Convergence results regarding the discretization parameters and numerical study of the asymptotic limit $\epsilon \to 0$ are presented. Moreover, the influence of the stabilization parameter $\sigma$ on the numerical results is discussed.

%%%%%%%%%%%%%%%%%%%%%%%%%%%%%%%%%%%%%%%%%%%%%%
\subsection{Analytical solution for both $\eps$-regimes} \label{SEC51}%%%%%%%%%%%
%%%%%%%%%%%%%%%%%%%%%%%%%%%%%%%%%%%%%%%%%%%%%%

Let us choose in this section the stationary stream function $\ds \Psi(x,y) = \frac{1}{2}(x^2+y^2)$, corresponding to $\bfb=(y,-x)^T$. In this case, we can compute explicitly the characteristics corresponding to the transport equation \eqref{ADV_bis}.
Indeed, the characteristic curve $\mathcal{C}_{\epsilon}^{t,x,y}(s) := \Big(X(s), Y(s)\Big)$ passing at instant $t$ through $(x,y)$, solves the ODE 
$$
\left\lbrace
\begin{array}{ll}
\displaystyle \dot{X}(s) = \frac{Y(s)}{\epsilon}\,, \\ \\
\displaystyle \dot{Y}(s) = - \frac{X(s)}{\epsilon}\,,
\end{array}\right.
\label{charac}
\quad \quad  (X(t),Y(t)) = (x,y)\,,
$$
and has the explicit form 

$$
\mathcal{C}_{\epsilon}^{t,x,y}(s) := \ds \left(\begin{array}{c}X \\Y\end{array}\right)(s;t,x,y)  =  \mathcal{R}\Big(\frac{s-t}{\epsilon}\Big) \left(\begin{array}{c}x \\y\end{array}\right) \,.
$$
with the rotation matrix given by
$$
\mathcal{R}(\theta) = 
\left(\begin{array}{cc}\ds \cos(\theta) & \ds \sin(\theta) \\ \\\ds -\sin(\theta)& \ds  \cos(\theta)\end{array}\right) \,.
$$
\noindent These characteristic curves are nothing else than spirals in the $(t,x,y)$-phase-space. All characteristics are $2\, \pi\, \epsilon$-periodic in $s$. The solution $f^{\epsilon}$ of the advection equation \eqref{ADV_bis} is now constant along these curves, such that
$$
f^{\epsilon}(t,x,y) = f_{in}(X(0;t,x,y),Y(0;t,x,y))\,, \quad \forall(t,x,y) \in [0,T] \times \Omega\,, 
$$
which leads to the following analytic expression of the unique solution to \eqref{ADV_bis}
\be \label{exact_sol}
f^{\epsilon}(t,x,y) = f_{in} \Big(\cos \Big( \frac{t}{\epsilon} \Big)x - \sin \Big(\frac{t}{\epsilon}\Big)y,\sin \Big(\frac{t}{\epsilon}\Big)x + \cos \Big( \frac{t}{\epsilon} \Big)y \Big)\,.
\ee

\noindent Finally, the $\epsilon \to 0$ limit solution $f^0$ solves the problem \eqref{ADV_bis_0}, leading to

\be \label{limit_sol}
 f^0 = \langle f_{in} \rangle = \frac{1}{2\pi} \; \int_0^{2\pi} f_{in}(R\cos(s),R\sin(s))\,\mathrm{d}s\,,
\ee

\noindent where $R = \sqrt{x^2+y^2}$. Let us choose now as initial data a Gaussian peak not centered in the origin, i.e.
$$
f_{in}(x,y) = \exp  \Bigg(- \frac{(x-0.5)^2 + (y-0.5)^2}{2 \; \eta^2} \Bigg)\,, \; \; \eta = 0.05 \,, \quad L_x=L_y=1\,,
$$

\noindent and investigate how the numerical scheme is rendering its movement. 

In the following, we shall denote by $\Pi_h(f^{\epsilon})$ resp. $f^{\epsilon,\sigma}_h$ the exact solution \eqref{exact_sol} calculated on the grid mesh resp. the numerical solution of \eqref{NR} obtained with our (DAMM)-scheme, and $\Pi_h(f^{0})$ refers to the exact limit solution \eqref{limit_sol} calculated on the grid mesh. Let $\mathcal{Q} := (0,T) \times \Omega_S$. We introduce also the numerical errors
\be \label{truncation_errors}
\begin{array}{l}
%\ds \mathcal{L}^{\infty}_{\epsilon} := ||f^{\epsilon}_{ex}(T) - f^{\epsilon,\sigma}_{num}(T)||_{L^{\infty}(\mathcal{Q})}\,, \quad \mathcal{L}_0^{\infty} := ||f^{0}_{ex}(T) - f^{\epsilon,\sigma}_{num}(T)||_{L^{\infty}(\mathcal{Q})}\,, \\[3mm] 
\ds \mathcal{L}^{p}_{\eps, \mathcal{X}} := ||\Pi_h(f^{\epsilon}) - f^{\epsilon,\sigma}_h||_{L^{p}_h(\mathcal{X})}\,, \quad \;  \; \mathcal{L}^{p}_{0,\mathcal{X}} := ||\Pi_h(f^{0}) - f^{\epsilon,\sigma}_h||_{L^{p}_h(\mathcal{X})}\,,
%\ds \mathcal{L}^{2}_{\eps} := ||f^{\epsilon}_{ex}(T) - f^{\epsilon,\sigma}_{num}(T)||_{L^{2}(\mathcal{Q})}\,, \quad \;  \; \mathcal{L}^{2}_{0} := ||f^{0}_{ex}(T) - f^{\epsilon,\sigma}_{num}(T)||_{L^{2}(\mathcal{Q})}\,.  
\end{array}
\ee

\noindent where $p \in \{1,2,...,\infty\}$, $\mathcal{X}$ stands for $\mathcal{Q}$ or $\Omega_S$, and $L^p_h$ denotes the discrete $L^p$-norm.\\

Figure \ref{circle_non_limit} shows now the numerical (DAMM)-scheme solution and the corresponding exact solution, in the non-limit ($\eps=1$) regime. The solutions related to the limit regime ($\eps$=0) are plotted in Figure \ref{circle_limit}. One observes that the numerical scheme we propose in this paper approximates well the exact solutions in both extreme regimes. In the next paragraphs we will try to prove more rigorously this visible correspondence and justify the choice of the stabilization parameter $\sigma$.

\begin{figure}[H]
\centering
	
            \centering
            \includegraphics[width=0.95\textwidth]{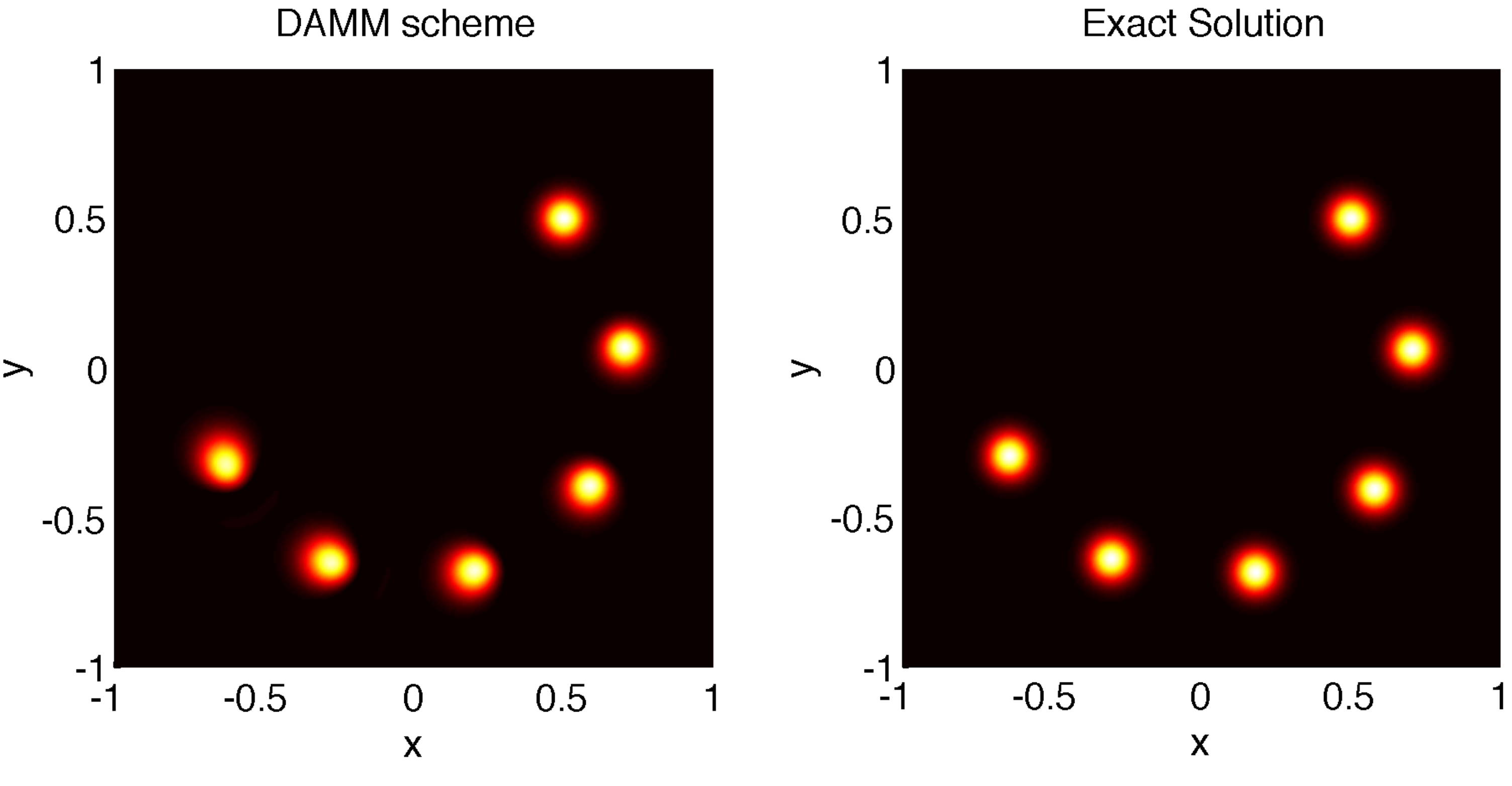}
            \caption{\small{(Non-limit case $\eps=1$). Superposition of the distribution function $f_h^{\eps,\sigma}$ at several time steps $t^n$ (left-panel) compared to the exact solution $f^{\eps}$ (right-panel), with $N_t=2000$, $T=3.5$, $N_x=N_y=200$, $\sigma = \Delta x^2$.                      } 					}
            \label{circle_non_limit}
    \end{figure}
    
\vspace{1cm}

\begin{figure}[H]
           \centering
           \includegraphics[width=0.95\textwidth]{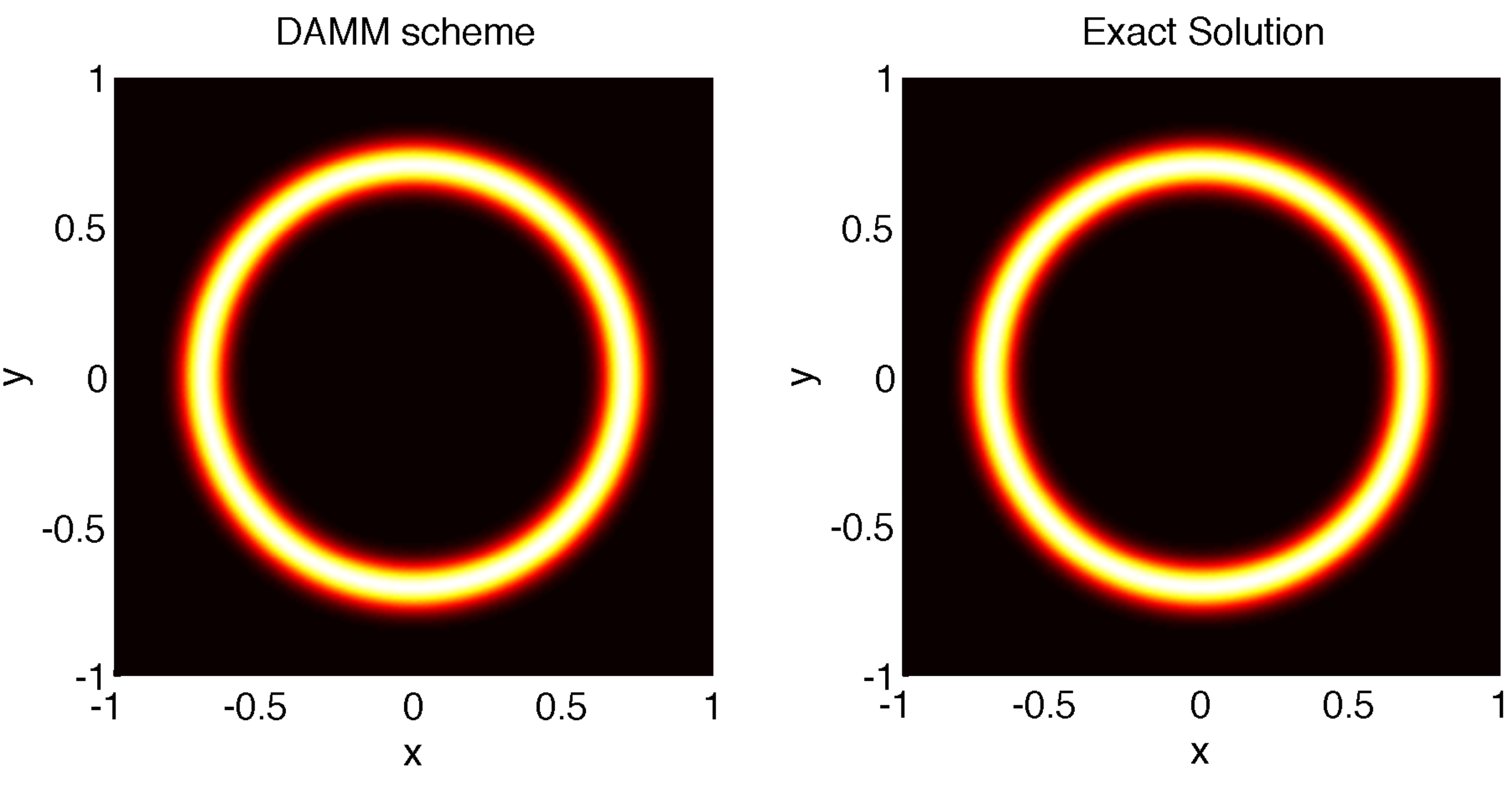}
            \caption{\small{(Limit case $\eps = 0$). Plot of $f^{0,\sigma}_h$ at final time $T=1$ with $N_t=200$ and $\sigma = \Delta x$ (left-panel), compared to the limit exact solution $f^{0}$ (right-panel). Mesh size : $N_x=N_y=200$.}}
            \label{circle_limit}
            
            \normalsize

%\caption{\small{Representation of both non limit and limit numerical solutions $f^{\eps,\sigma}_h$ (A)-left and $f^{0,\sigma}_h$ (B)-left, compared to the respective analytic solutions $f^{\eps}_{h}$ (A)-right and $f^{0}_{h}$ (B)-right. Mesh size:  $N_x=N_y=200$. }}

\end{figure}

\bigskip 

\bigskip

\bigskip 

\bigskip

\subsection{AP-property} Let us start by discussing the AP property of our scheme. As a reference scheme, we decided to take a fully implicit DIRK-scheme with Arakawa space discretization, solving \eqref{ADV_bis}. In Figure \ref{figure3} (A) we plot the condition number of the system matrix of the (DAMM)-scheme and of the implicit reference scheme as a function of $\epsilon$. One observes that the condition number of the (DAMM)-scheme is bounded uniformly in $\epsilon$, whereas the implicit scheme is ill-conditioned in the limit $\epsilon \to 0$. This underlines one important advantage of our (DAMM)-scheme when compared with standard schemes, namely the fact that the (DAMM)-scheme does not degenerate in the limit $\epsilon \to 0$. From the right panel (B), however, it is evident that the condition of the (DAMM)-scheme depends on the stabilization parameter $\sigma$. This reflects the fact that in the limit $\sigma \to 0$, the solution $(f^{\eps,0,n}, q^{\eps,0,n})$ of \eqref{ADV_AP_h} is not unique, and therefore the problem becomes ill-posed.  As mentioned later on, $\sigma$ has to be chosen not too small, such to have a reasonable condition number, and not too large, in order not to modify too much the problem. 

\begin{figure}[ht]
\centering
	\begin{subfigure}[b]{0,49\textwidth}
           \centering
           \includegraphics[width=\textwidth]{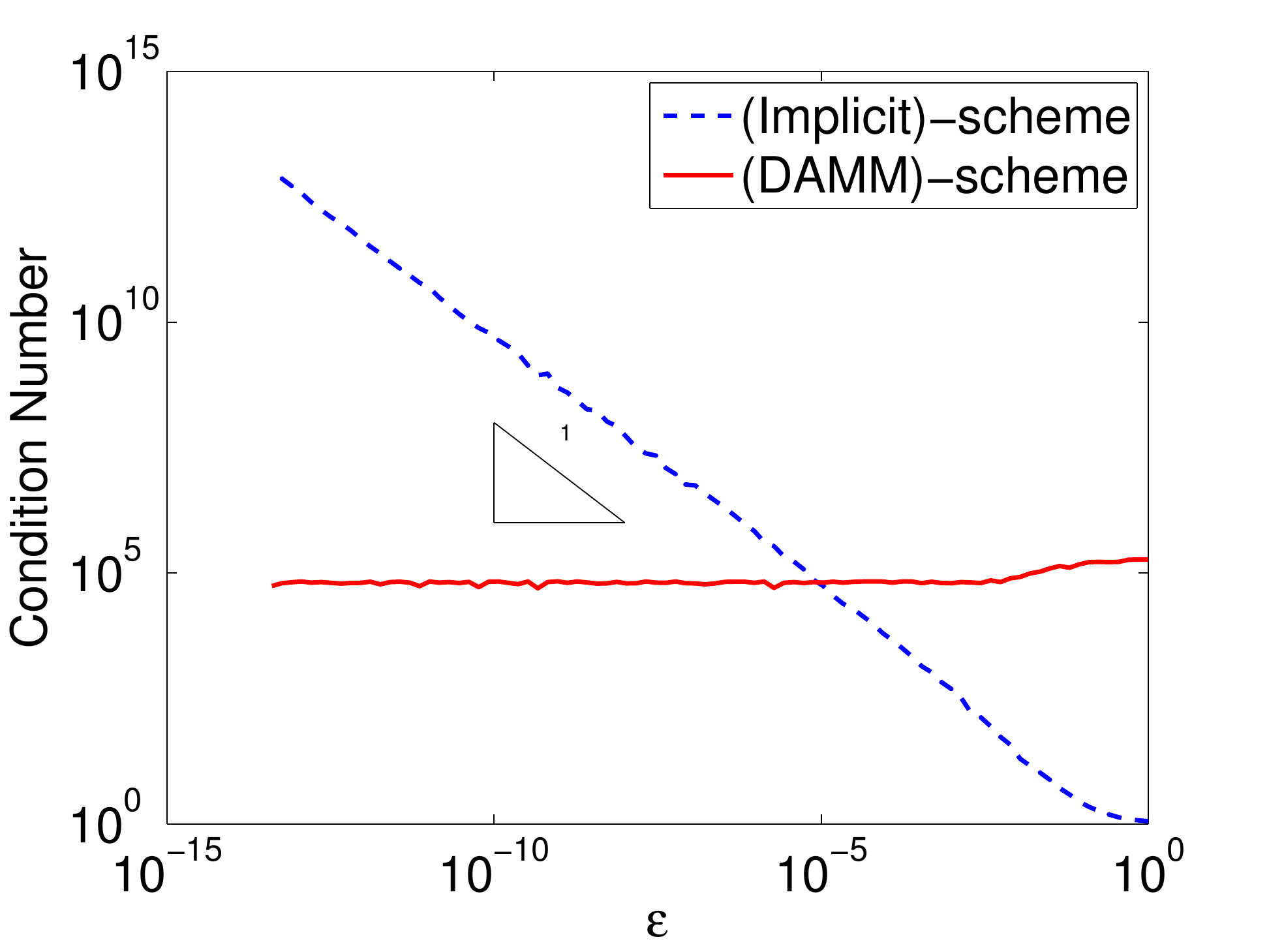}
            \caption{}
            \label{fig:a}
    \end{subfigure}
	\begin{subfigure}[b]{0.49\textwidth}
            \centering
            \includegraphics[width=\textwidth]{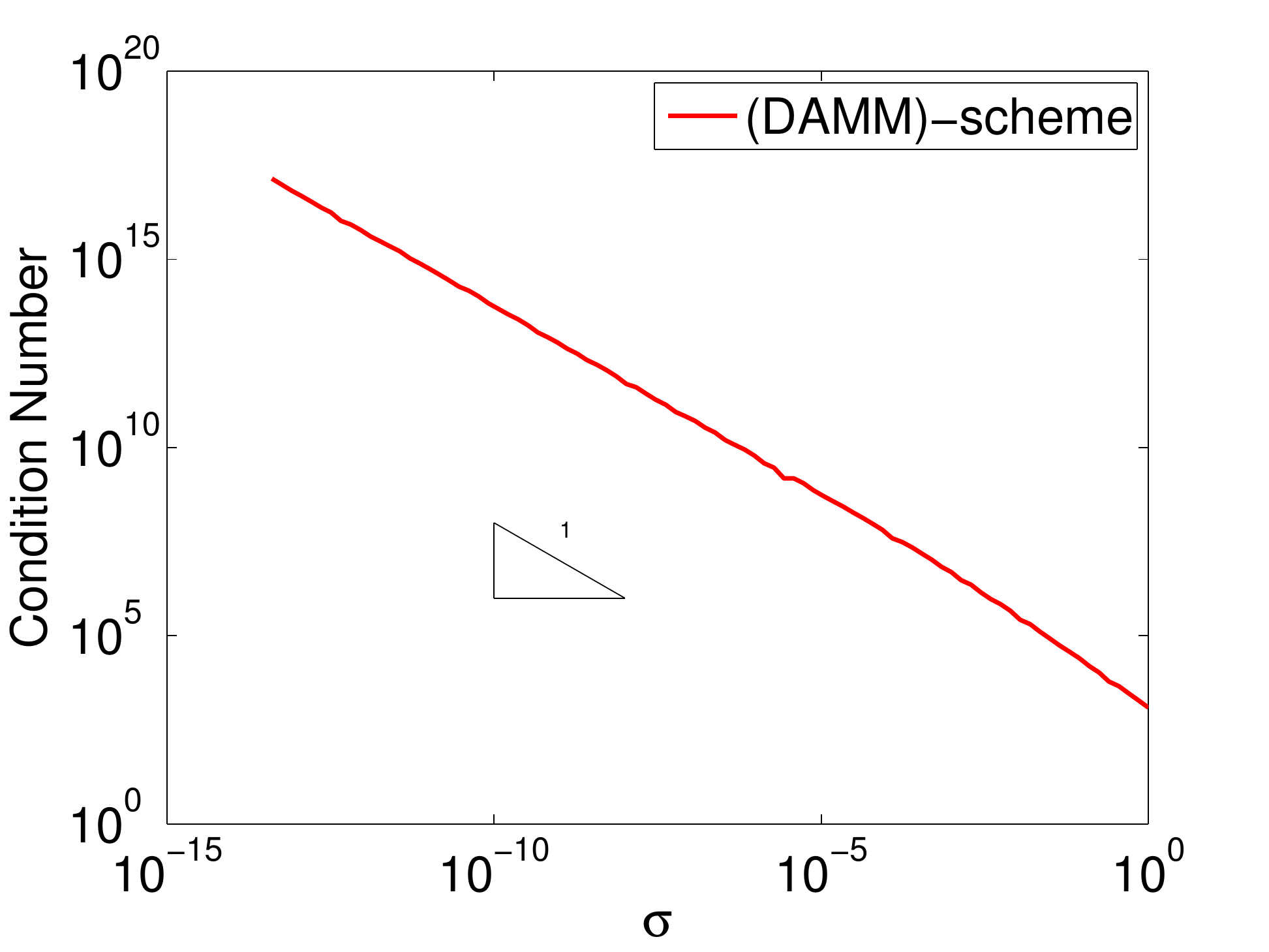}
            \caption{}
            \label{fig:b}
    \end{subfigure}
\caption{\small{Condition numbers as a function of $\epsilon$ with $\sigma=\Delta x^2$ for both (DAMM) and fully implicit schemes (A), and as a function of $\sigma$ with $\epsilon=1$ for the (DAMM)-scheme (B). Here $T=1$, $N_t = 200$, $N_x = N_y = 50$. }}
\label{figure3}
\end{figure}

\subsection{Convergence property} Next we study the convergence properties of the (DAMM)- scheme. Figure \ref{figure4} displays the convergence rates in $\Delta t$ and in $\Delta x = \Delta y$, obtained by comparison with the exact solutions \eqref{exact_sol} for $\epsilon=1$ and $\sigma=\Delta x^2$. In panel (A), one observes the expected second-order convergence rates with respect to time. In panel (B), the second-order convergence in space due to the Arakawa discretization of the Poisson brackets is evident. We observe from Table \ref{table1} that the convergence rate in space is even better for $\eps=0$. 

\begin{figure}[ht]
\centering
	\begin{subfigure}[b]{0,49\textwidth}
           \centering
           \includegraphics[width=\textwidth]{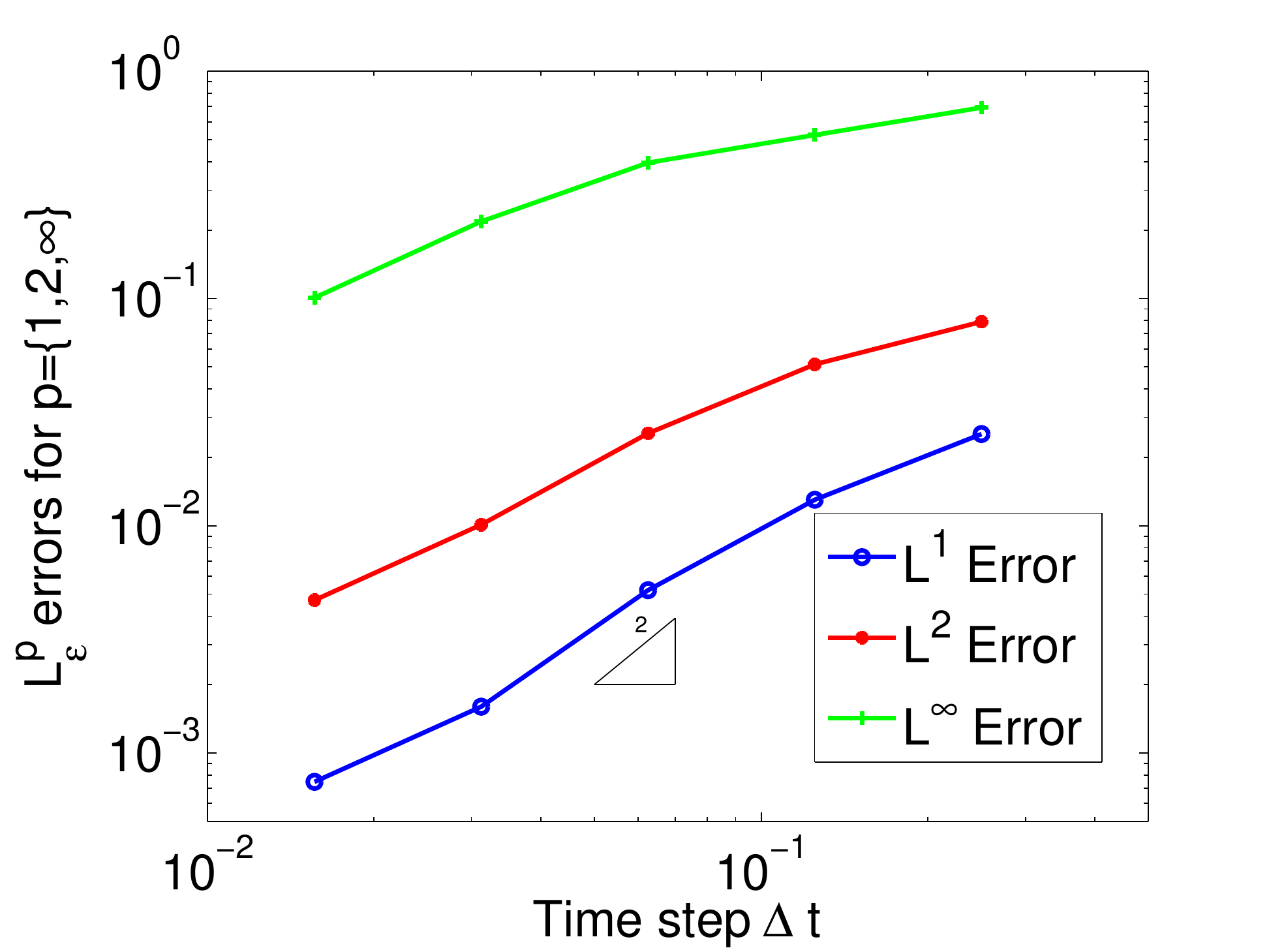}
            \caption{Errors in time. Fixed mesh size: $N_x=N_y=350$.}
            \label{fig:a}
    \end{subfigure}
	\begin{subfigure}[b]{0.49\textwidth}
            \centering
            \includegraphics[width=\textwidth]{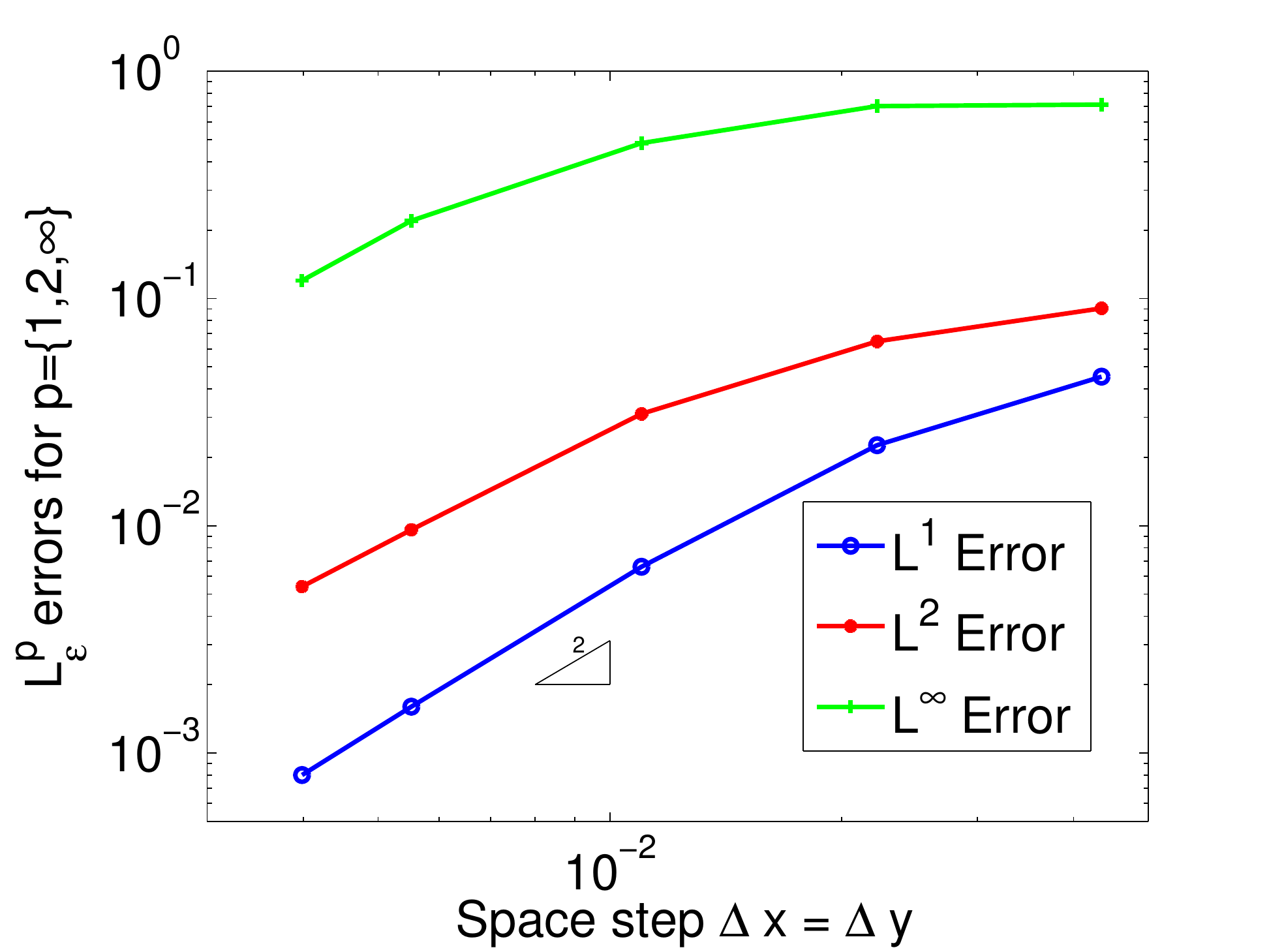}
            \caption{Errors in space. Fixed number of iterations: $N_t=700$.}
            \label{fig:b}
    \end{subfigure}
\caption{\small{Convergence studies for the (DAMM)-scheme: errors $\mathcal{L}^1_{\eps,\mathcal{Q}}$,  $\mathcal{L}^2_{\eps,\mathcal{Q}}$, and $\mathcal{L}^{\infty}_{\eps,\mathcal{Q}}$ as a function of $\Delta t$ (A) and $\Delta x = \Delta y$ (B), respectively. Parameters were $T=1$, $\eps=1$ and $\sigma = \Delta x^2$.  }}
\label{figure4}
\end{figure}

\subsection{Influence of the stabilization} Our next objective is to investigate the influence of the stabilization parameter $\sigma$ on the quality of the numerical solution in the different $\eps$-regimes. Let us start the discussion with the limit regime $\eps=0$. In Figure \ref{figure5} (A), we plot the $\mathcal{L}^1_{0,\Omega_S}$ error (wrt the exact sol.) over time $t$\,, for two different choices of $\sigma$, namely $\sigma = \Delta x$ and $\sigma = \Delta x^2$, and for several choices of $\Delta x$. One observes two phases: at first the numerical solution relaxes towards the (weak) analytic limit solution, until a plateau is reached and the error remains constant over time. This relaxation is faster for smaller values of $\sigma$, fact which can be explained by taking a look at the degenerate diffusion equation \eqref{Diff} we are effectively solving in the limit $\eps \to 0$. Smaller $\sigma$ means stronger diffusion along the field lines of $\mathbf{b}$, which means that the number of iterations $n_{eq}$ to reach the equilibrium plateau decreases with $\Delta x$, see Table \ref{table1}.  Observe also that the error in the equilibrium phase is the same for each $\sigma = \Delta x^{r}$ with $r \geq 1$, only the relaxation rate strongly depends on the choice of $r$. 

Let us mention briefly the computational time (CPU time) one needs for reaching the equilibrium plateau for the different cases studied in Figure \ref{figure5} (A). The problem is that the condition number of the system matrix is inversely proportional to $\sigma$, as already demonstrated in Figure \ref{figure3} (B). This bad conditioning would lead necessarily to an increase in CPU time for very small $\sigma$ which has to be evaluated. For example, regarding the case $\Delta x = 2/80$ (third curve in Figure \ref{figure5} (A) and third line in Table \ref{table1}), one obtains $t_{CPU}(\sigma =\Delta x) = 53 \; \textnormal{s}$ in contrast to $t_{CPU}(\sigma =\Delta x^2) = 104 \; \textnormal{s}$. Thus, even if Figure \ref{figure5} (A) suggests that a higher $r$ would be more suitable to attain quickly the equilibrium plateau, this previous study about the CPU time advices us to be more careful and choose $\sigma$ not too small.\\

In the regime $\epsilon=1$, displayed in Figure \ref{figure5} (B), the $\mathcal{L}_{\epsilon, \Omega_S}^1$ error increases linearly with time for all choices of $\Delta x$. There is a very weak $\sigma$ dependence in this regime as shown by the quasi superposition of the curves. The linear increase of the error is due to the dispersive character of the Arakawa discretization which leads to errors in the phase velocities. The not-dependence on $\sigma$ is due to the fact that the term $\eps \, \bfb \cdot \nabla \, q^{\eps}$ in the second equation of \eqref{ADV_AP} is larger for $\epsilon=1$ than the term $\sigma \, q^{\epsilon}$.\\

\begin{figure}[ht]
\centering
	\begin{subfigure}[b]{0,49\textwidth}
           \centering
           \includegraphics[width=\textwidth]{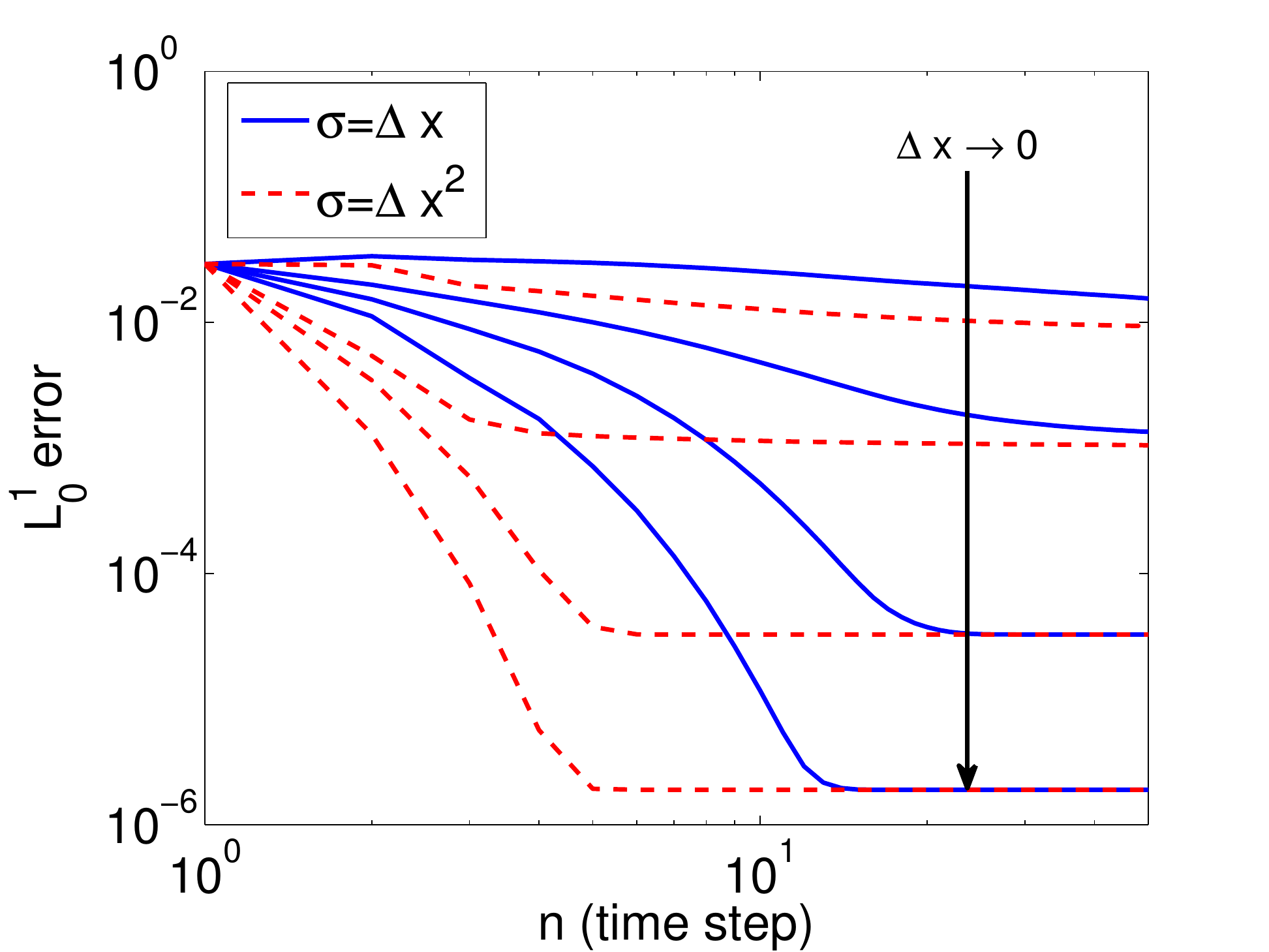}
             \caption{ $\eps = 0$ and two choices of $\sigma$.}
            \label{fig:a}
    \end{subfigure}
	\begin{subfigure}[b]{0.49\textwidth}
            \centering
            \includegraphics[width=\textwidth]{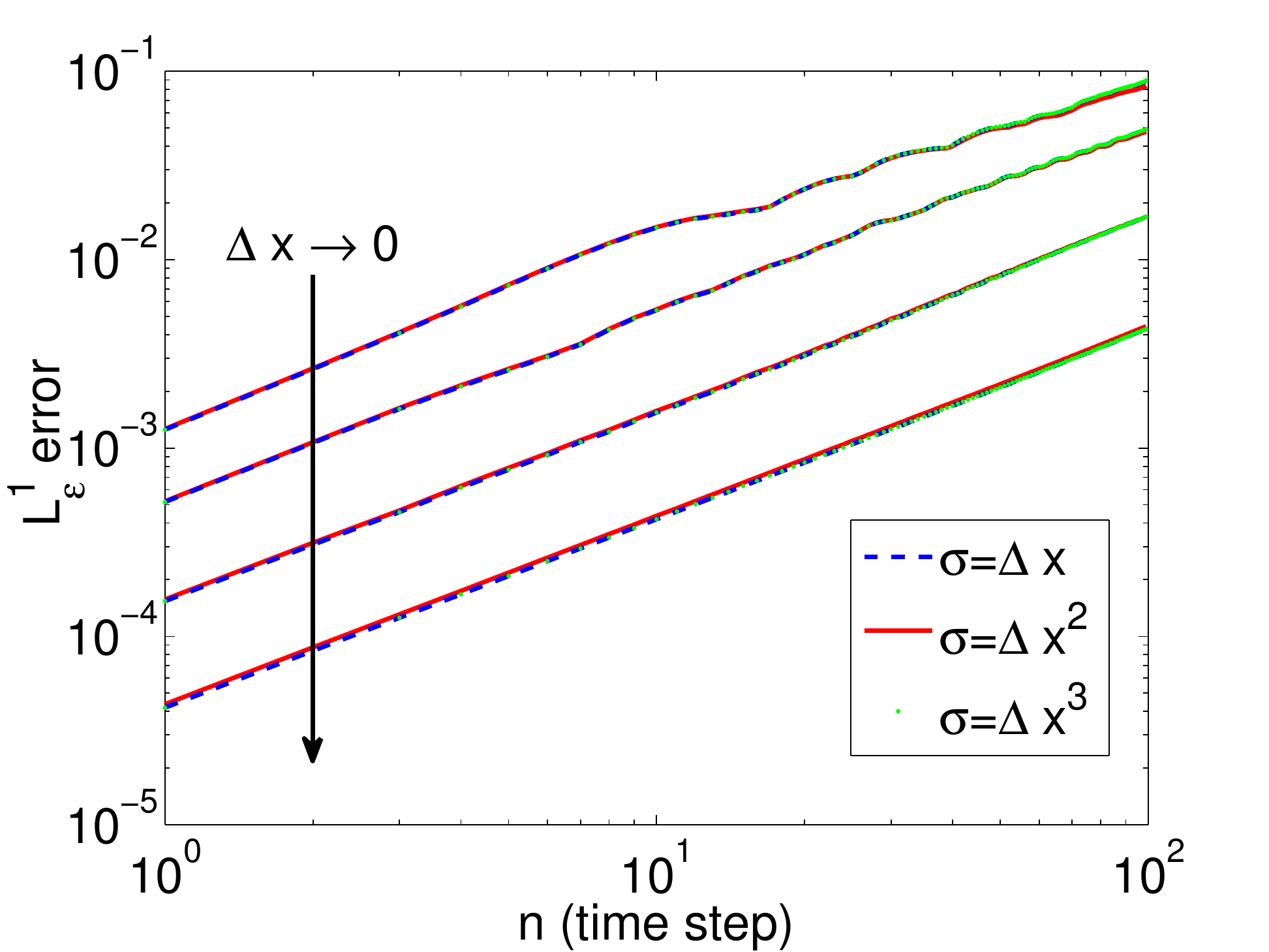}
              \caption{ $\eps = 1$  and three choices of $\sigma$.}
            \label{fig:b}
    \end{subfigure}
\caption{ $\mathcal{L}^1_{0,\Omega_S}$ (panel (A)) and $\mathcal{L}^1_{\eps,\Omega_S}$ (panel (B)) errors over time for the (DAMM)-scheme for two different $\epsilon$-regimes and several $\Delta x$.}
\label{figure5}
\end{figure}

\begin{table}[htbp]
\caption{Corresponding to Figure \ref{figure5} (A). $\mathcal{L}^1_{0,\Omega_S}$-error at final time $T=200\,\Delta t$, $\Delta t = 0.01$; number of iterations $n_{eq}$ and CPU-time $t_{CPU}$ for reaching the equilibrium plateau; condition nbr. CN of the linear system.}
\begin{center}
\begin{tabular}{l|l|c|c|c|c||c|c|c|c} %l:left c:center r:right |:table lines
\cmidrule[1pt]{3-10}                  % 1pt is the thickness 3-10 is column number
\multicolumn{2}{c|}{}&\multicolumn{4}{c||}{$\sigma=\Delta x$}&\multicolumn{4}{c}{$\sigma=\Delta x^2$} \\\cmidrule{3-10}
\multicolumn{2}{c|}{} & $\mathcal{L}^1_{0,\Omega_S}$ at time $T$ & $n_{eq}$ & $t_{CPU}$ & CN & $\mathcal{L}^1_{0,\Omega_S}$ at time $T$ & $n_{eq}$ & $t_{CPU}$ & CN \\\midrule\midrule
\multirow{4}{3mm}{\begin{sideways}\parbox{15mm}{\quad $\Delta x$}\end{sideways}}
& $2/20$ & $0.0106$ & $> 200$ & $3.1$s & $8.4e3$ & $0.0082$ & $175$ & $2.8s$ & $1.3e5$ \\
& $2/40$ & $0.0011$ & $169$ & $12$s & $2.6e4$ & $0.0010$ & $66$ & $4.9$s & $8.5e5$ \\
& $2/80$ & $3.2737e-5$ & $123$ & $53$s & $9.1e4$ & $3.2735e-5$ & $10$ & $104$s & $3.9e6$ \\
& $2/160$ & $1.8992e-6$ & $19$ & $133$s & $2.9e5$ & $1.8992e-6$ & $5$ & $1302$s & $3.5e7$ \\\midrule
\end{tabular}
\end{center}
\label{table1} 
\end{table}

 \subsection{Choice of the stabilization parameter $\sigma$} After having given some qualitative observations about the influence of the discretization parameter $\sigma$ in different $\epsilon$-regimes, let us present some ideas to optimize the choice of $\sigma$. In Figure \ref{figure7} (A), we plot the $\mathcal{L}^1_{\epsilon,\Omega_S}$ error at the final time $T$ for the non-limit regime $\eps=1$ as a function of $\sigma$ for several values of $\Delta x$.  In order to minimize the error $\mathcal{L}^1_{\epsilon,\Omega_S}$\,, the curves suggest to choose a $\sigma$-value depending on $\Delta x$. To investigate more precisely this dependence, we propose to take as "optimal" $\sigma$, for each fixed $\Delta x$, a value $\sigma_h^{\epsilon}$ such that
 
\be \label{sigma_critique}
% \sigma_c = \inf_{\sigma} \Bigg \{ \sigma \in [0,1] | \ln \Bigg(\frac{||f^{\epsilon}_{ex} - f^{\epsilon,\sigma}_{num}||_{L^{p}(\mathcal{X})}}{||f^{\epsilon}_{ex} - f^{\epsilon,\sigma_{min}}_{num}||_{L^{p}(\mathcal{X})}} \Bigg) < \eta \Bigg\} 
\sigma_h^{\epsilon} := \max \setc*{\sigma \in [\sigma_{min},1]}{\frac{||\Pi_h(f^{\epsilon}) - f^{\epsilon,\sigma}_h||_{L^{1}_h(\Omega_S)}-||\Pi_h(f^{\epsilon}) - f^{\epsilon,\sigma_{min}}_h||_{L^{1}_h(\Omega_S)}}{||\Pi_h(f^{\epsilon}) - f^{\epsilon,\sigma_{min}}_h||_{L^{1}_h(\Omega_S)}}<\eta}\,,
 \ee
 
 \noindent (where $\eta$ is an arbitrary precision) and evaluate how $\sigma^{\eps}_h$ varies with $\Delta x$. In Figure \ref{figure7} (B), we display $\ln(\sigma_h^{\epsilon})$ as a function of $\ln(\Delta x)$ for $\eta=0.01$ and $\sigma_{min}=7e-6$. The data approach a polynomial line of slope $p=1.917$, suggesting a relation between $\sigma_h^{\epsilon}$ and $\Delta x$ of the form $\sigma_h^{\epsilon}=C\, \Delta x^p$ (with $C>0$), very close to the relation chosen in the last sections. To end this paragraph, let us briefly analyze the influence of $\eta$, the precision criterion, appearing in \eqref{sigma_critique}. For that, Table \ref{table2} presents for several values of $\eta$, the slope of the polynomial fitting of order $1$ of the data $\ln(\sigma_h^{\epsilon})=f(\ln(\Delta x))$, as well as its correlation coefficient $r^2$. One notes that the slope is around $2$ in each case, validating the choice we have made before, for the stabilization parameter $\sigma$ in the large $\epsilon$-regime, namely $\sigma_h^{\eps} = (\Delta x)^2$.
 
 \begin{figure}[ht]
\centering
	\begin{subfigure}[b]{0,49\textwidth}
           \centering
           \includegraphics[width=\textwidth]{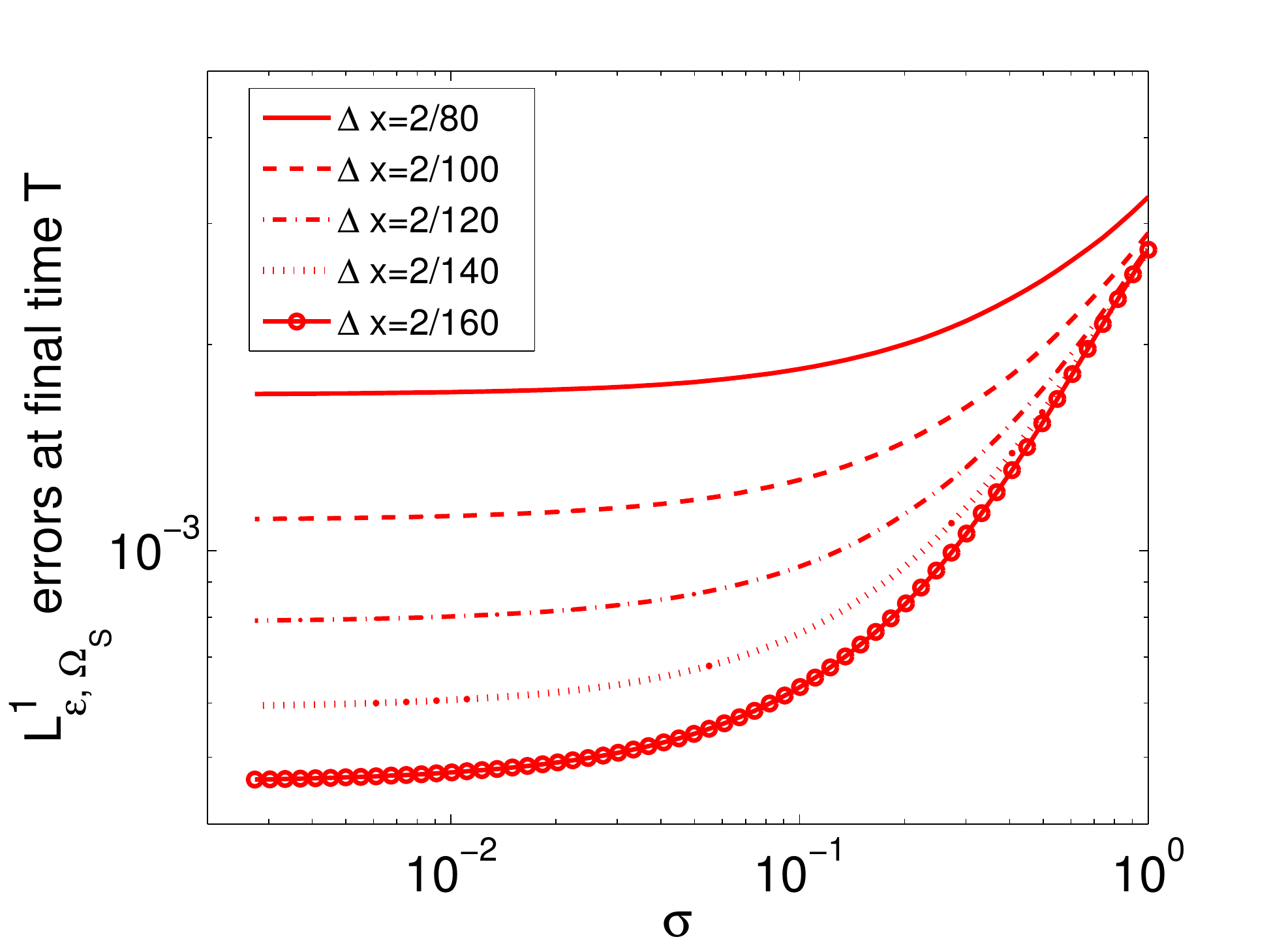}
            \caption{ $\mathcal{L}^1_{\eps,\Omega_S}$ error at final time $T$ as a function of $\sigma$.}
            \label{fig:a}
    \end{subfigure}
	\begin{subfigure}[b]{0.49\textwidth}
            \centering
            \includegraphics[width=\textwidth]{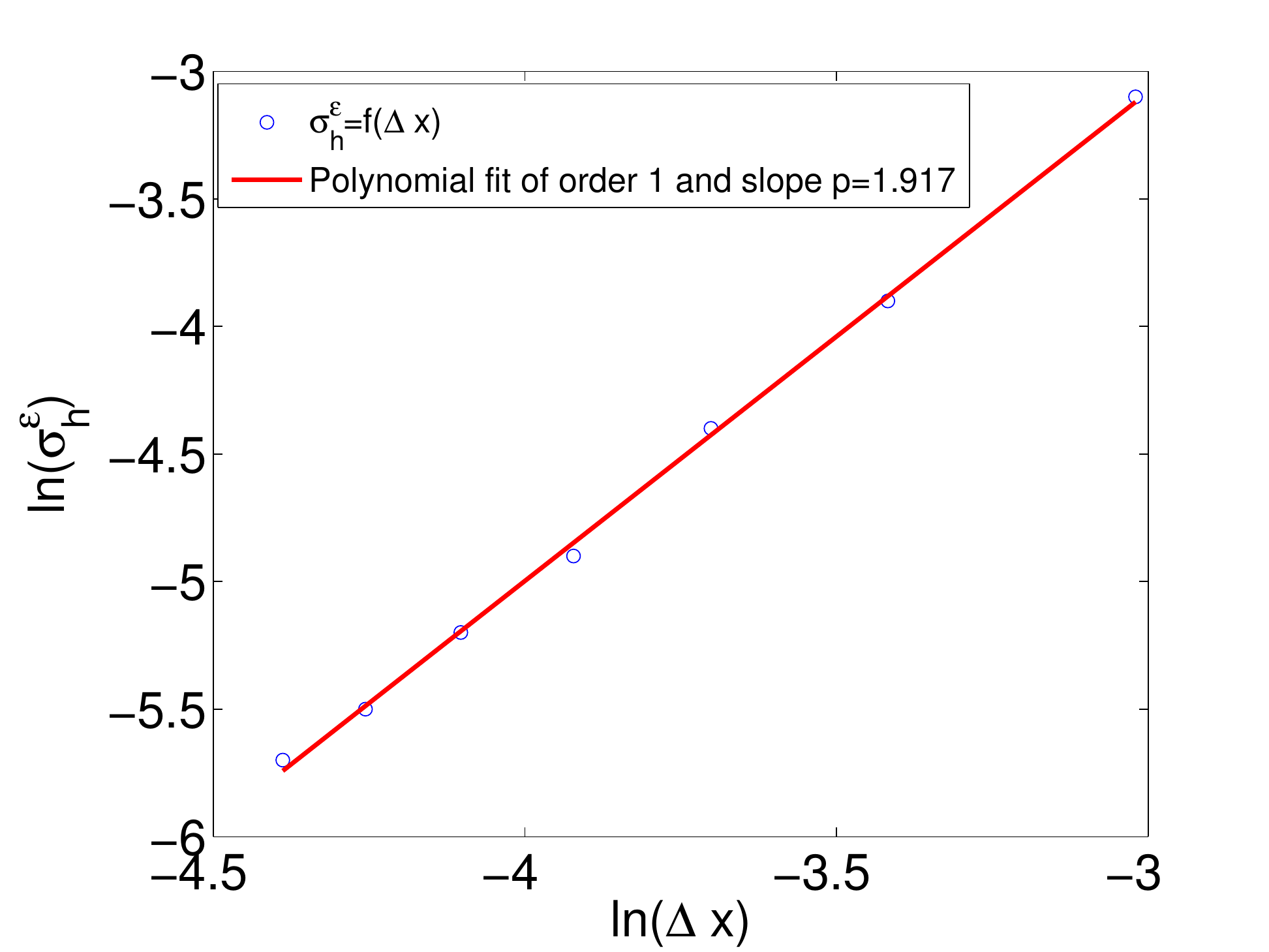}
            \caption{ Relation between $\sigma_h^{\eps}$ and $\Delta x$. Parameters were $\eta=0.01$, $\sigma_{min}=7e-6$.}
            \label{fig:b}
    \end{subfigure}
\caption{ (Non-limit case $\eps=1$). Study of the influence of $\sigma$ in the $\mathcal{L}^1_{\eps,\Omega_S}$  error (A) and of the relation between $\sigma_h^{\eps}$ and $\Delta x$ (B). Time discretization : $T=0.1$ and $\Delta t = 0.01$. }
\label{figure7}
\end{figure}

\begin{table} % Add the following just after the closing bracket on this line to specify a position for the table on the page: [h], [t], [b] or [p] - these mean: here, top, bottom and on a separate page, respectively
\caption{Polynomial fitting of order $1$ for the data $\ln (\sigma_h^{\epsilon})=f(\ln(\Delta x))$ for several values of the precision $\eta$. In each case, the slope $p$ of the line and the correlation coefficient $r^2$ are written.}
\begin{center} % Centers the table on the page, comment out to left-justify
\begin{tabular}{|c|c|c|}
 % The final bracket specifies the number of columns in the table along with left and right borders which are specified using vertical bars (|); each column can be left, right or center-justified using l, r or c. To specify a precise width, use p{width}, e.g. p{5cm}
% Top horizontal line
% Horizontal line spanning less than the full width of the table - you can add (r) or (l) just before the opening curly bracket to shorten the rule on the left or right side
\hline % In-table horizontal line
$\eta$ & $p$  & $r^2$\\
\hline % Content row 1
$1e-1$&$1.974$&$0.9975$\\
\hline
$5e-2$&$1.974$&$0.9975$\\
\hline
$1e-2$&$1.9175$&$0.9988$\\
\hline
$5e-3$&$1.8675$&$0.9988$\\
\hline
$1e-3$&$1.8597$&$0.9989$\\
\hline % Bottom horizontal line
\end{tabular}% A label for referencing this table elsewhere, references are used in text as \ref{label}
\end{center} 
\label{table2} 
\end{table}

We perform a similar analysis in the limit regime $\epsilon = 0$. For that, we display in Figure \ref{figure8} (A) the $\mathcal{L}^1_{0,\Omega_S}$ error at the final time $T$ as a function of $\sigma$\,, for several values of $\Delta x$. As in the non-limit $\epsilon$-regime, the curves suggest us to choose a $\sigma$ depending on $\Delta x$. Let us define the application $\mathcal{Z} : \sigma \mapsto ||\Pi_h(f^{0}) - f^{0,\sigma}_h||_{L^{1}_h(\Omega_S)}$. To study the dependence between $\sigma$ and $\Delta x$, we choose for each $\Delta x$ a $\sigma_h^0$ defined by:

\be
\sigma_h^0 := \arg \max_{\sigma \in [\sigma_{min} ,1]} \left|\frac{d\mathcal{Z}}{d\sigma}\right|(\sigma)\,,
\ee
\noindent and we plot in Figure \ref{figure8} (B) the evolution of $\ln(\sigma_h^0)$ as a function of $\ln(\Delta x)$. The corresponding data follow a linear relation, with a slope of $p=0.857$, meaning that $\sigma$ can be chosen as $\sigma=\Delta x$ in the limit regime $\eps=0$ with the aim to reduce the $\mathcal{L}^1_{0,\Omega_S}$ error and avoid a bad condition number.

\bigskip

To conclude this first numerical part, one can say that this simple test case permits to make a deep analysis of the (DAMM)-scheme. In particular, the AP behavior of the scheme was confirmed, the orders of convergence in both space and time were confirmed, and the influence of $\sigma$ as well as its delicate choice have been intensively investigated. This study was enabled by the existence of analytic solutions of the problem, rigorously compared to solutions obtained by the (DAMM)-scheme for several $\epsilon$-regimes. Thanks to this verification, the (DAMM)-scheme can be used to resolve more complicated models where no analytic solutions are at hand. This is the topic of the next part.
 \begin{figure}[H]
\centering
	\begin{subfigure}[b]{0,49\textwidth}
           \centering
           \includegraphics[width=\textwidth]{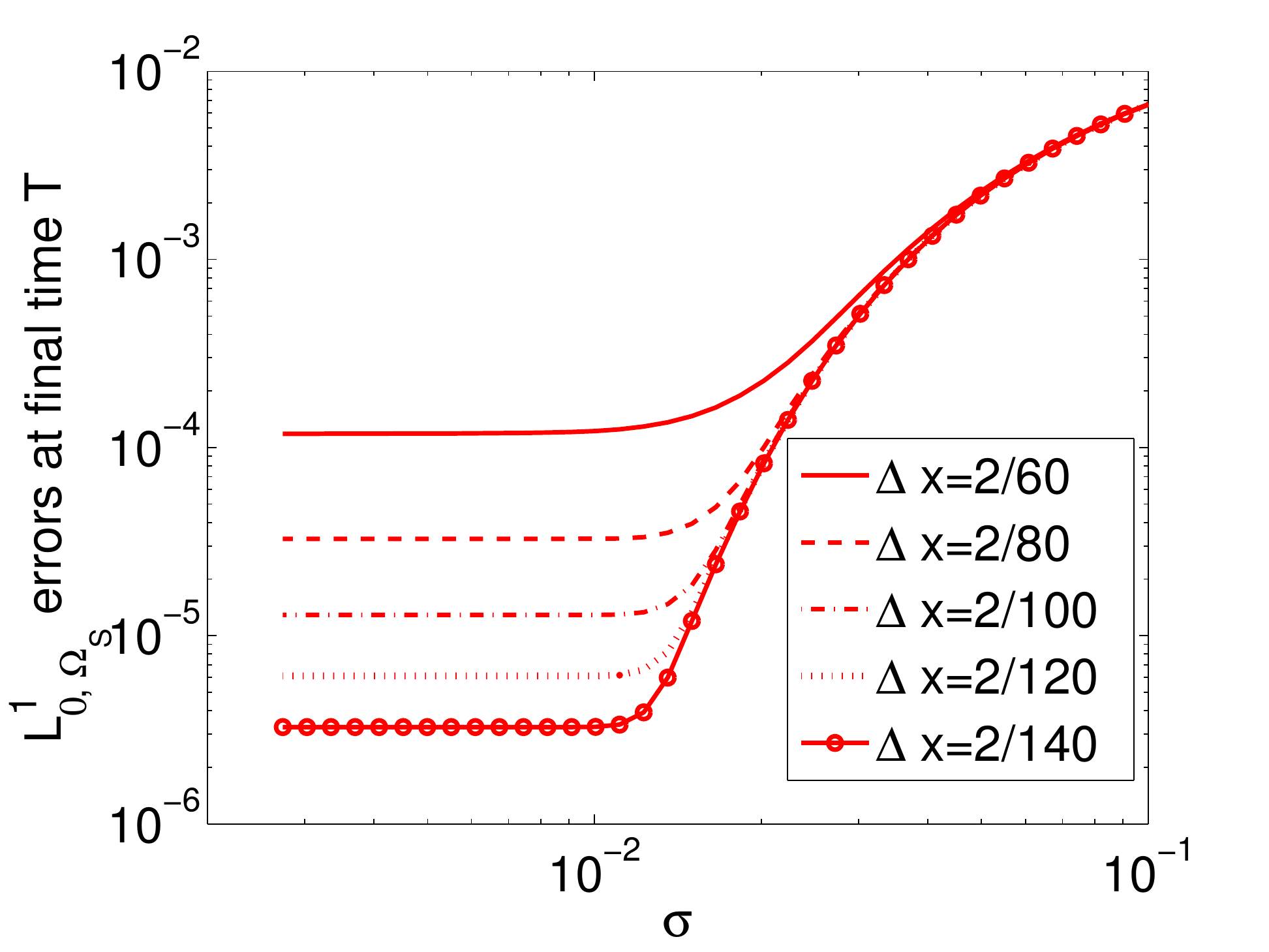}
            \caption{$\mathcal{L}^1_{0,\Omega_S}$ error at final time $T$ as a function of $\sigma$.}
            \label{fig:a}
    \end{subfigure}
	\begin{subfigure}[b]{0.49\textwidth}
            \centering
            \includegraphics[width=\textwidth]{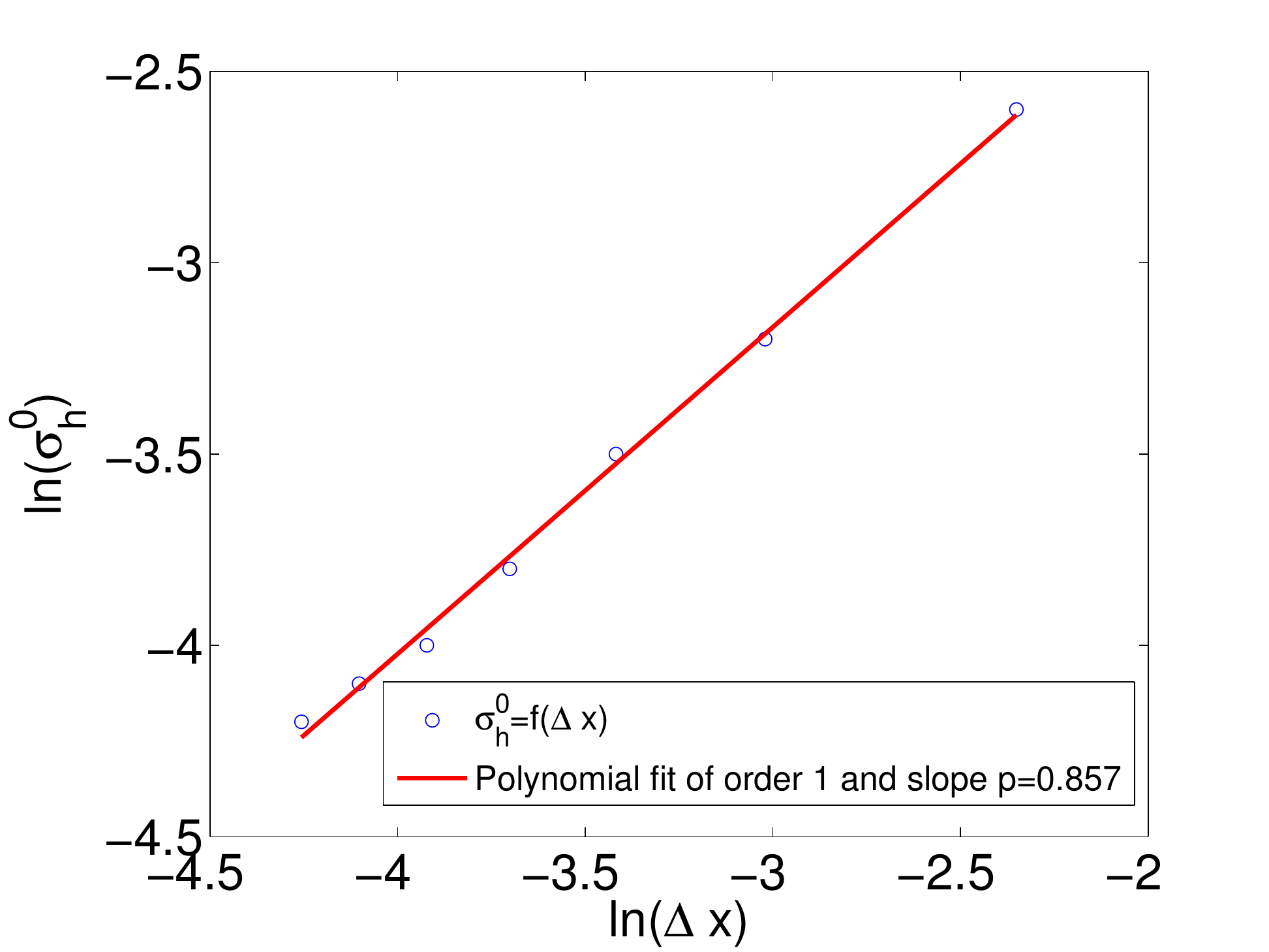}
            \caption{Relation between $\sigma_h^{0}$ and $\Delta x$. One chooses $\sigma_{\min}=7e-6$.}
            \label{fig:b}
    \end{subfigure}
\caption{(Limit case $\eps=0$). Study of the influence of $\sigma$ in the $\mathcal{L}^1_{0,\Omega_S}$error (A) and of the relation between $\sigma_h^{0}$ and $\Delta x$ (B). Time discretization : $T=0.1$ and $\Delta t = 0.01$.}
\label{figure8}
\end{figure}

%%%%%%%%%%%%%%%%%%%%%%%%%%%%%%%%%%%%%%%%%%%%%%

%%%%%%%%%%%%%%%%%%%%%%%%%%%%%%%%%%%%%%%%%%%
%%%%%%%%%%%%%%%%%%%%%%%%%%%%%%%%%%%%%%%%%%%
\section{Numerical simulations for the Vlasov-Poisson test case} \label{SEC6}
%%%%%%%%%%%%%%%%%%%%%%%%%%%%%%%%%%%%%%%%%%%
%%%%%%%%%%%%%%%%%%%%%%%%%%%%%%%%%%%%%%%%%%%

The aim of this section is dual: firstly to solve numerically the Vlasov-Poisson system \eqref{VP1D1V} using the (DAMM)-scheme and to simulate some particular physical phenomena (such as the Landau damping or the two-stream instability) ; and secondly to study the long-time asymptotics $\epsilon \to 0$ of the two-stream instability. Note that the literature on the Vlasov-Poisson system is very rich, some theoretical as well as numerical results can be found in the non-exhaustive list \cite{mbflr, crou, frenod1, golse}.

%%%%%%%%%%%%%%%%%%%%%%%%%%%%%%%%%%%%%
\subsection{The Vlasov-Poisson system and its numerical discretization}
%%%%%%%%%%%%%%%%%%%%%%%%%%%%%%%%%%%%%

\noindent In this chapter, we set $\mathcal{Q} := (0,T) \times [0,L_x] \times [-L_v,L_v]$. Using the Poisson bracket, the Vlasov-Poisson 1D1V system verified by $f^{\epsilon}:=f^{\epsilon}(t,x,v)$ reads

\begin{equation} \label{Vlasov_Poisson_bracket} 
(VP)^{\eps}\,\,\, \left\{
\begin{array}{ll||}
\ds \partial_t f^{\epsilon} + \frac{1}{\epsilon}\, \{f^{\epsilon}, \Psi^{\epsilon} \}=0\,, \\[3mm]
%\ds E^{\eps}(t,x) = - \partial_x \varphi^{\eps}(t,x)\,, \quad \quad \quad \forall(t,x,v) \in \mathcal{Q} \,, \\[3mm]
\ds -\partial_{xx} \varphi^{\eps}(t,x) = 1-n^{\eps}(t,x)\,,
\end{array}
\right.
\end{equation}

\noindent where $\Psi^{\epsilon}(t,x,v) := v^2/2 - \varphi^{\epsilon}(t,x)$ is the stream-function and $\ds n^{\eps}(t,x) := \int_{\mathbb{R}} f^{\eps}(t,x,v) dv$ denotes the electron density. Due to the fact that this problem is non-linear (unlike the previous one), its study is a more delicate task. 

\noindent Following the same reformulation as before, we can construct an Asymptotic-Preserving scheme for the Vlasov-Poisson system by introducing an auxiliary variable $q^{\epsilon}(t,x,v)$. The AP-reformulation of \eqref{Vlasov_Poisson_bracket} is then discretized with the help of the (DAMM)-scheme as before. The determination of the electric field $E^{\epsilon}(t,x)$ is guaranteed by the resolution of the discrete Poisson equation. The fully discretized (first order in time) reformulated Vlasov-Poisson system is summarized here for clarity. For each time step $n$, one is looking for $(f^{\epsilon,\sigma,n+1}_h, q^{\epsilon,\sigma,n+1}_h)$, by iterating in $l \in \mathbb{N}$ like
\begin{equation} \label{Vlasov_Poisson_bracket_AP_full} 
(RVP)^{\epsilon,\sigma,n,l}_{h}\,\,\, \left\{
\begin{array}{ll||}
\ds - \frac{\varphi^{\epsilon,\sigma,n+1,l}_{i+1}-2\, \varphi^{\epsilon,\sigma,n+1,l}_{i} + \varphi^{\epsilon,\sigma,n+1,l}_{i-1}}{ \Delta x^2} =1- \Delta v \,\sum_{j=1}^{N_v} f^{\epsilon,\sigma,n+1,l}_{i,j}\,,  \\[3mm]
\ds \Psi^{\epsilon,\sigma,n+1,l}_{i,j} = \frac{1}{2}\, v_j^2 - \varphi^{\epsilon,\sigma,n+1,l}_{i}\,,  \\[3mm]
\ds f^{\epsilon,\sigma,n+1,l+1}_{i,j} + \Delta t \; [q^{\epsilon,\sigma,n+1,l+1}_{h}, \Psi_{h}^{\epsilon,\sigma,n+1,l}]_{i,j}=f^{\epsilon,\sigma,n}_{i,j}\,, \\[3mm]
\ds [f^{\epsilon,\sigma,n+1,l+1}_{h}, \Psi_{h}^{\epsilon,\sigma,n+1,l}]_{i,j} = \epsilon  \, [q^{\epsilon,\sigma,n+1,l+1}_{h}, \Psi_{h}^{\epsilon,\sigma,n+1,l}]_{i,j} - \sigma  \; q^{\epsilon,\sigma,n+1,l+1}_{i,j}\,,
%\ds E^{\epsilon,\sigma,n+1,k}_i = -\frac{\varphi^{\epsilon,\sigma,n+1,k}_{i+1}-\varphi^{\epsilon,\sigma,n+1,k}_{i-1}}{2 \, \Delta x}\,. \\[3mm]
\end{array}
\right.
\end{equation}
and starting from 
$$
f^{\epsilon,\sigma,n+1,0}_{i,j} := f^{\epsilon,\sigma,n}_{i,j}\,.
$$
This iterative procedure has been done, because of the non-linearity of the previous system. In the following simulations, the stopping criterion for these iterations (at $l=l_f$) is 

$$
\frac{||f^{\epsilon,\sigma,n+1,l+1}_{i,j} - f^{\epsilon,\sigma,n+1,l}_{i,j}||_1}{||f^{\epsilon,\sigma,n+1,l}_{i,j}||_1} + \frac{||\varphi^{\epsilon,\sigma,n+1,l+1}_{i} - \varphi^{\epsilon,\sigma,n+1,l}_{i}||_1}{||\varphi^{\epsilon,\sigma,n+1,l}_{i}||_1} < 10^{-2}\,,
$$
and we finish by posing
$$
f^{\epsilon,\sigma,n+1}_{i,j} := f^{\epsilon,\sigma,n+1,l_f+1}_{i,j}\,.
$$
\begin{remark}
Note that we wrote the previous system \eqref{Vlasov_Poisson_bracket_AP_full} without the DIRK time discretization in order to simplify its writing, however the following simulations had been implemented with the (DAMM)-scheme, including a DIRK time discretization. 
\end{remark}

%%%%%%%%%%%%%%%%%%%%%%%%%%%%%%%%%%%%%%%%%%%%%%%%%%%
\subsection{Numerical simulations for weak Landau damping and $\eps=1$}%
%%%%%%%%%%%%%%%%%%%%%%%%%%%%%%%%%%%%%%%%%%%%%%%%%%%

In order to validate our numerical procedure, we are interested in the Landau damping, for which analytic results are at hand. The Landau damping represents the exponential decrease of the electric field energy as a function of time (see \cite{mouhot,landau} for more details). For these simulations, the following initial data (see for example \cite{crou,filbet}) is considered:
\be \label{initial_data_vp}
f_{in}(x,v) =\frac{1}{\sqrt{2\pi}}\,(1+\gamma \cos(kx))\,e^{-v^2/2}\,, 
\ee
\noindent where $\gamma$ refers to the amplitude and $k$ to the mode of the perturbation of the homogeneous equilibrium $\ds \mathcal{M}(v) = (2\pi)^{-1/2}\, e^{-v^2/2}$. In the following simulations, we take $L_v=10$, $L_x=2\pi/k$, $\eps=1$ and $\sigma = (\Delta x/L_x)^2$. 
In this section, we investigate the weak Landau damping, choosing a low amplitude of perturbation $\gamma$. According to \cite{oneil}, the weak Landau damping manifests for times $t<1/\sqrt{\gamma}$. Beyond this time, the non-linear effects begin to be significant. Thus, we resolve the Vlasov-Poisson system \eqref{Vlasov_Poisson_bracket_AP_full} with the above initial condition \eqref{initial_data_vp} for $\gamma=0.001$ and $k=0.5$\,. To simplify the notation, we shall denote in the following simply by $f^{\eps}$ our numerical solution obtained by the (DAMM)-scheme.

\begin{figure}[ht]
\centering
	\begin{subfigure}{.49\textwidth}
  	\centering
  	\includegraphics[width=\linewidth,trim={0cm 0cm 0cm 0cm},clip]{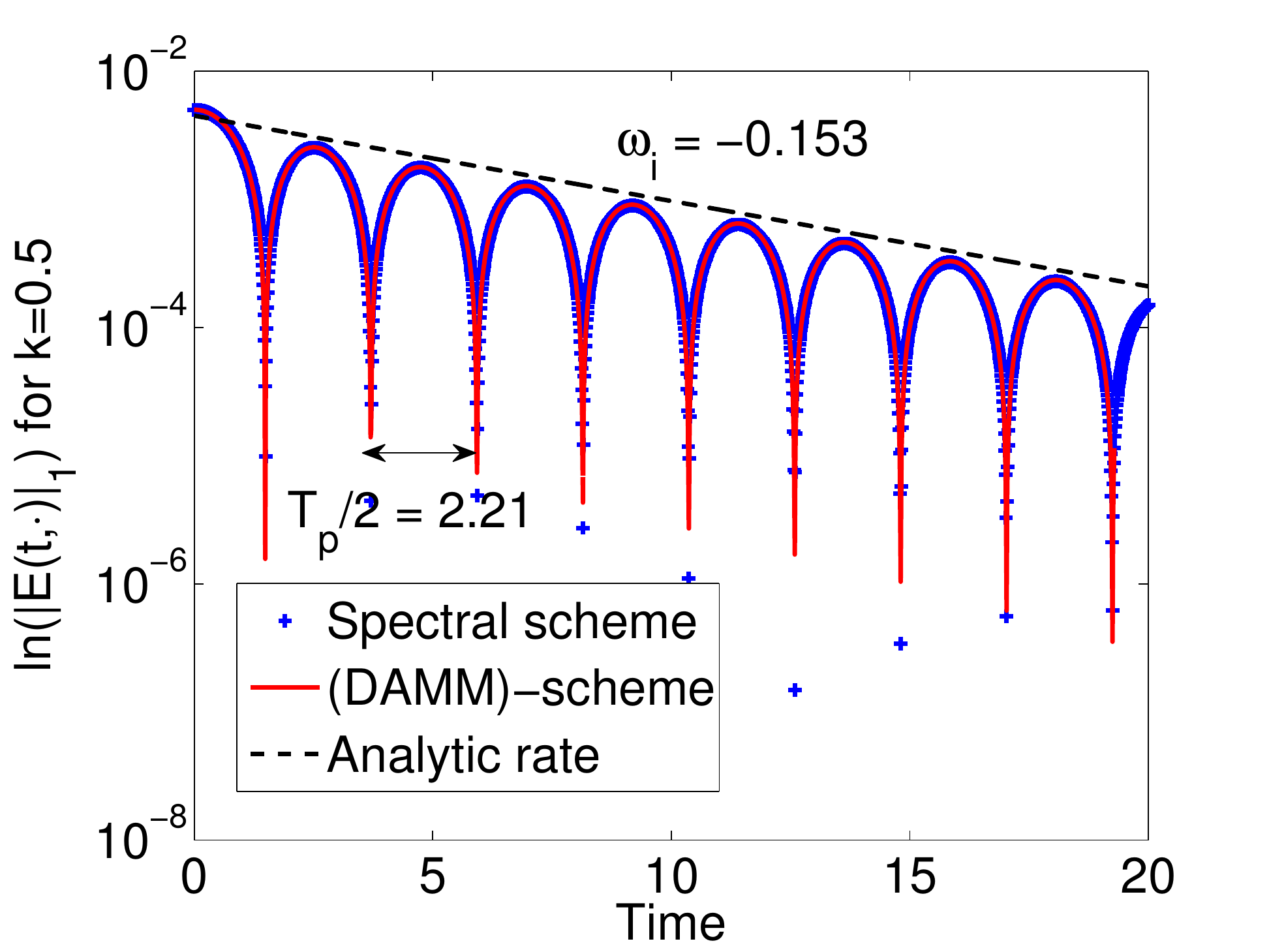}
	\caption{}
	\end{subfigure}
    \begin{subfigure}{.49\textwidth}
  	\centering
  	\includegraphics[width=\linewidth,trim={0cm 0cm 0cm 0cm},clip]{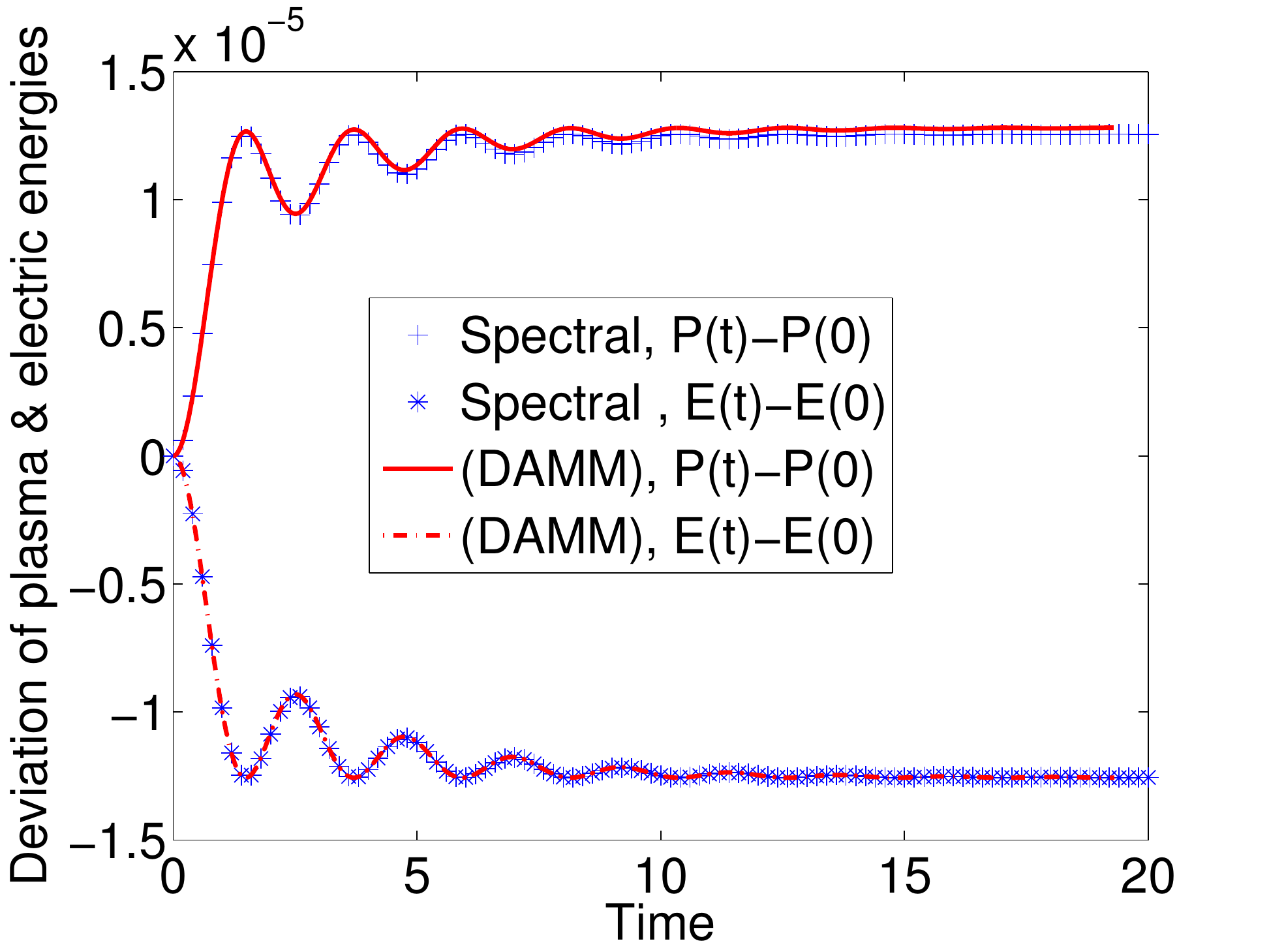}
	\caption{ }
	\end{subfigure} \\ 
	\caption{(Weak Landau damping for $\eps=1$). $L^1$-norm of the electric field (in log-scale) versus time (A) and deviations of plasma and electric energies versus time (B) for both (DAMM) and spectral schemes. $N_x=N_y=256$, $\Delta t =0.01$, $T=20$, $\sigma=(\Delta x /L_x)^2$.}
\label{linear_energy}
\end{figure}

In Figure \ref{linear_energy} (A) we represent the evolution in time of the $L^1$-norm of the electric field $||E^{\eps}(t,\cdot)||_1$ (in $log$-scale) obtained from the (DAMM)-scheme. So as to validate efficiently our (DAMM)-scheme, we plot in the same Figure \ref{linear_energy} (A) the corresponding evolution with a reference spectral scheme which resolves the system \eqref{Vlasov_Poisson_bracket}. The curves obtained from the two numerical schemes coincide perfectly. Moreover, we pay attention to the damping rate $\omega_i$ and the frequency of oscillations $\omega_p$, which depend on the perturbation mode $k$. Under certain approximations, several formulae of these latter can be found (see for example \cite{mak}).  One sees that both schemes approach the analytic values (for $k=0.5$), namely  $\omega_i= -0.153$ and $\omega_p = 2\,\pi/T_p = 1.415$. In Figure \ref{linear_energy} (B), we plot the deviation (from their initial value) of both electric and plasma energies. These latter are defined as $\ds \mathcal{E^{\eps}}(t) := 1/2 \, \int_{-L_x}^{L_x} \,|E^{\eps}(t, x)|^2\,dx$  and $ \ds \mathcal{P}^{\eps}(t) := 1/2 \, \int_{\mathbb{R}}\int_{-L_x}^{L_x} v^2\,f^{\eps}(t,x,v)\,dx\,dv$. The curves indicate clearly that the total energy $\mathcal{T}^{\eps} := \mathcal{E}^{\eps} + \mathcal{P}^{\eps}$ is conserved in compliance with the theory, for both (DAMM) and spectral schemes. Thus, the weak Landau damping is well simulated by the (DAMM)-scheme.

\bigskip

In Figure \ref{cut_linear_landau_damping} we displayed the distribution function in phase-space at time $t=0$ (panel (A)), $t=20$ (panel (C)) and $t=40$ (panel (E)). Note that one has plotted the perturbed part of the distribution function $f^\eps$, meaning $f^\eps-\mathcal{M}(v)$. In the panels (B), (D), and (F), we represented the cross-sections at $x=L_x/2$ of the previous plots, at the same times. These figures show us the continuous filamentation of $f^{\eps}$ over time. 

\begin{figure}[ht]
\centering
	\begin{subfigure}{.49\textwidth}
  	\centering
  	\includegraphics[width=\linewidth,trim={0cm 0cm 0cm 0cm},clip]{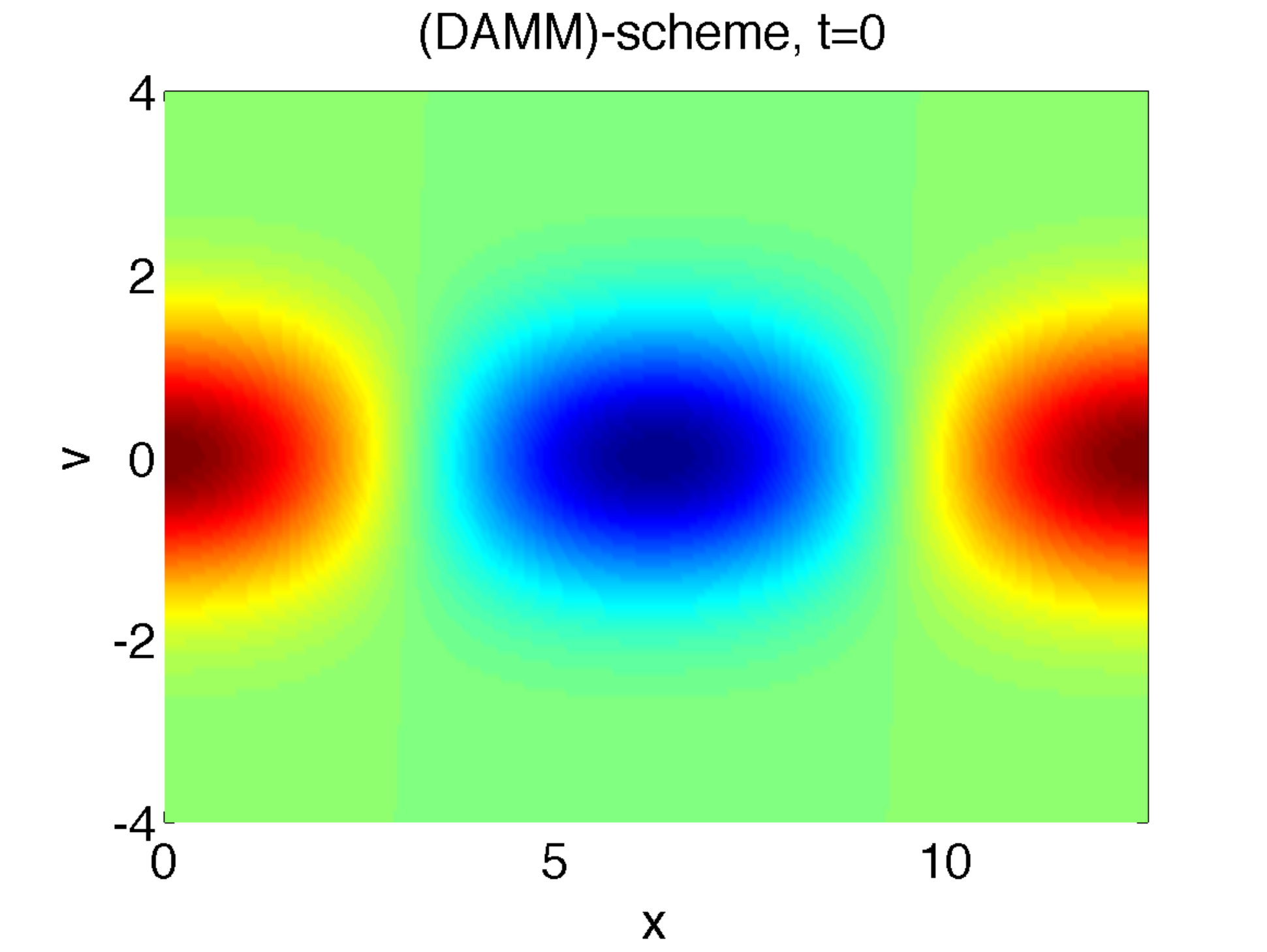}
	\caption{t=0.}
	\end{subfigure}
    \begin{subfigure}{.49\textwidth}
  	\centering
  	\includegraphics[width=\linewidth,trim={0cm 0cm 0cm 0cm},clip]{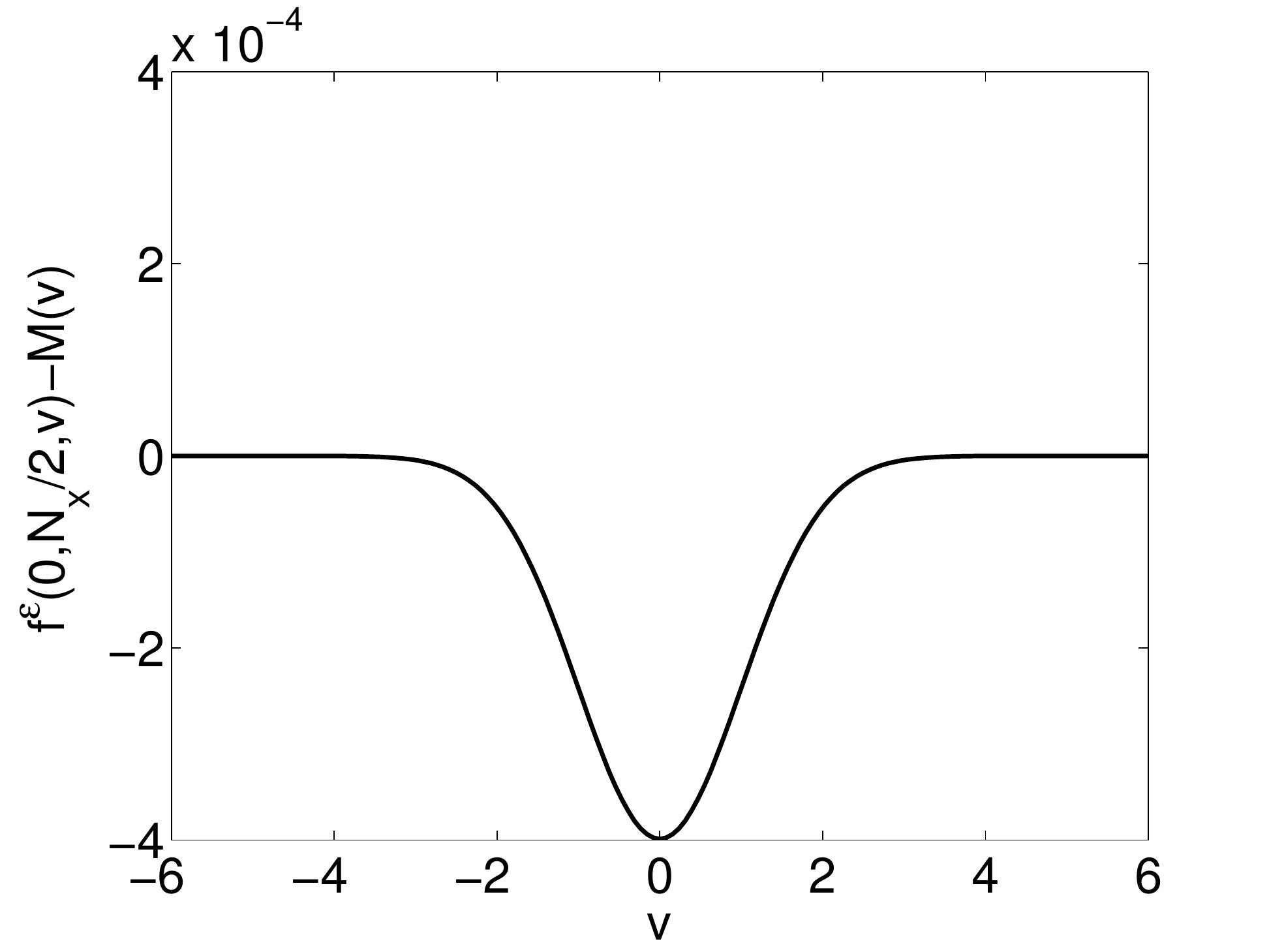}
	\caption{t=0.}
	\end{subfigure} \\ \vspace{0.5cm}
	\begin{subfigure}{.49\textwidth}
  	\centering
  	\includegraphics[width=\linewidth,trim={0cm 0cm 0cm 0cm},clip]{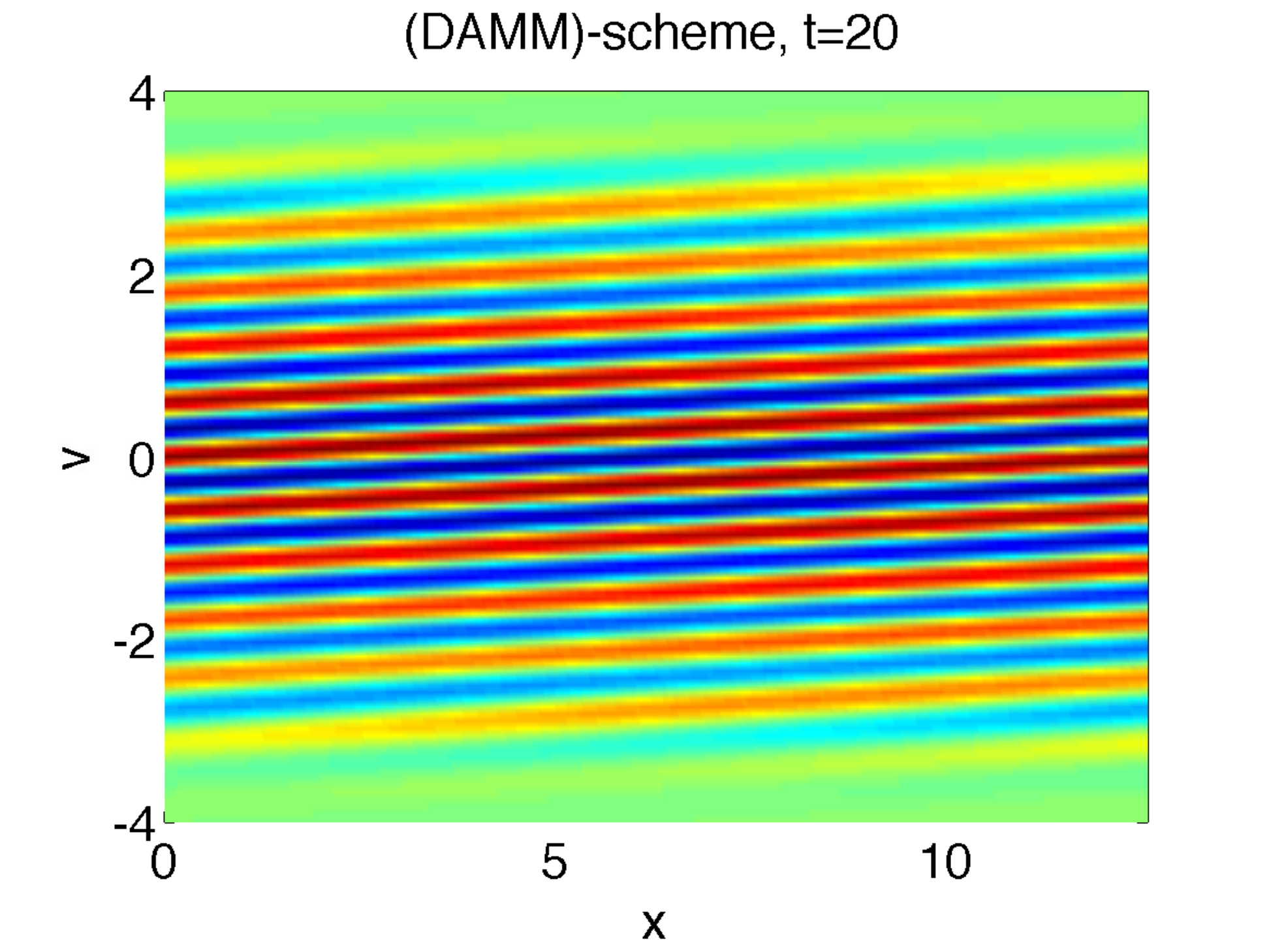}
	\caption{t=20.}
	\end{subfigure}
     \begin{subfigure}{.49\textwidth}
   	\centering
   	\includegraphics[width=\linewidth,trim={0cm 0cm 0cm 0cm},clip]{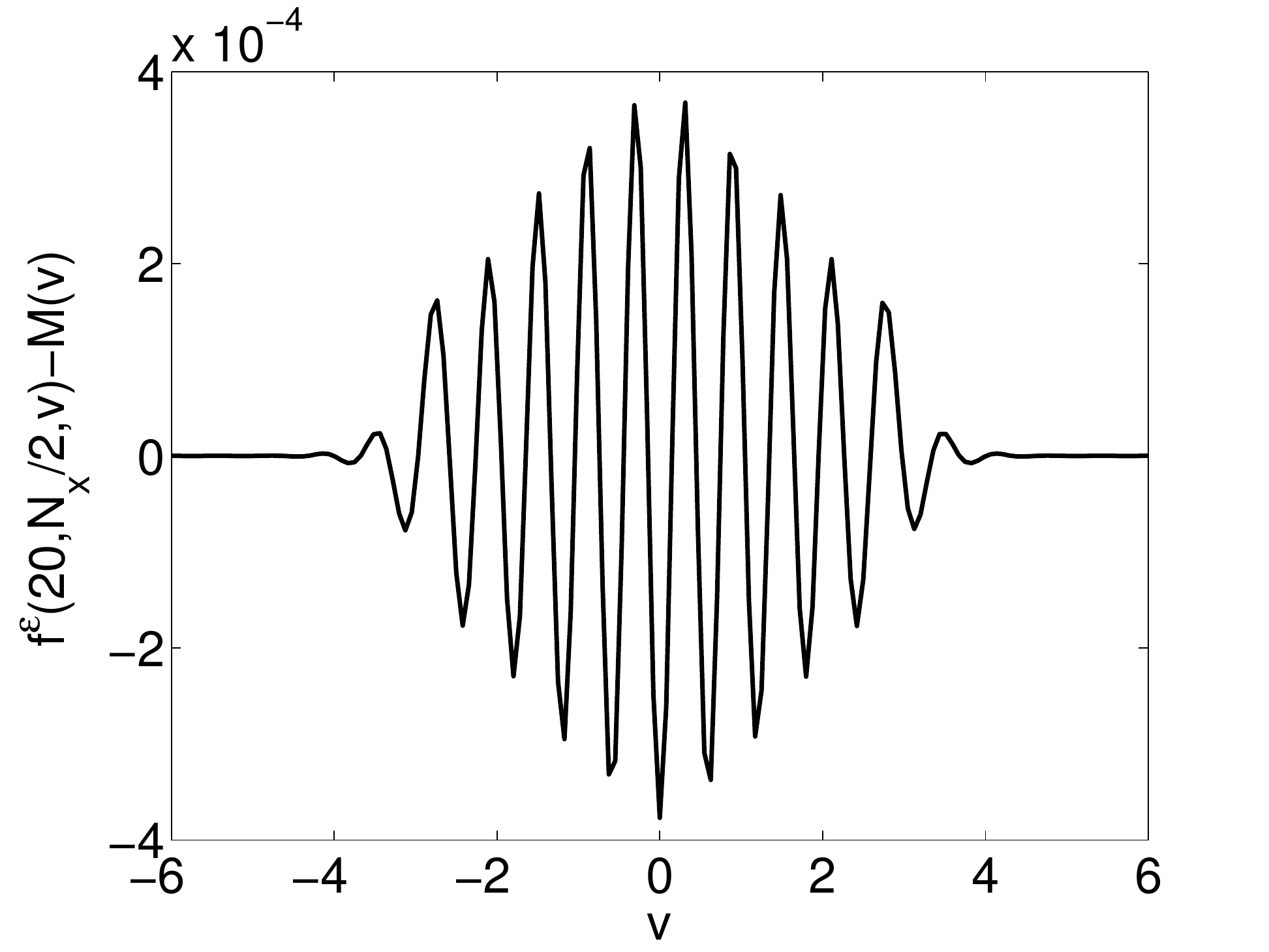}
 	\caption{t=20.}
 	\end{subfigure} \vspace{0.5cm}
 	\begin{subfigure}{.49\textwidth}   	
	\centering
   	\includegraphics[width=\linewidth,trim={0cm 0cm 0cm 0cm},clip]{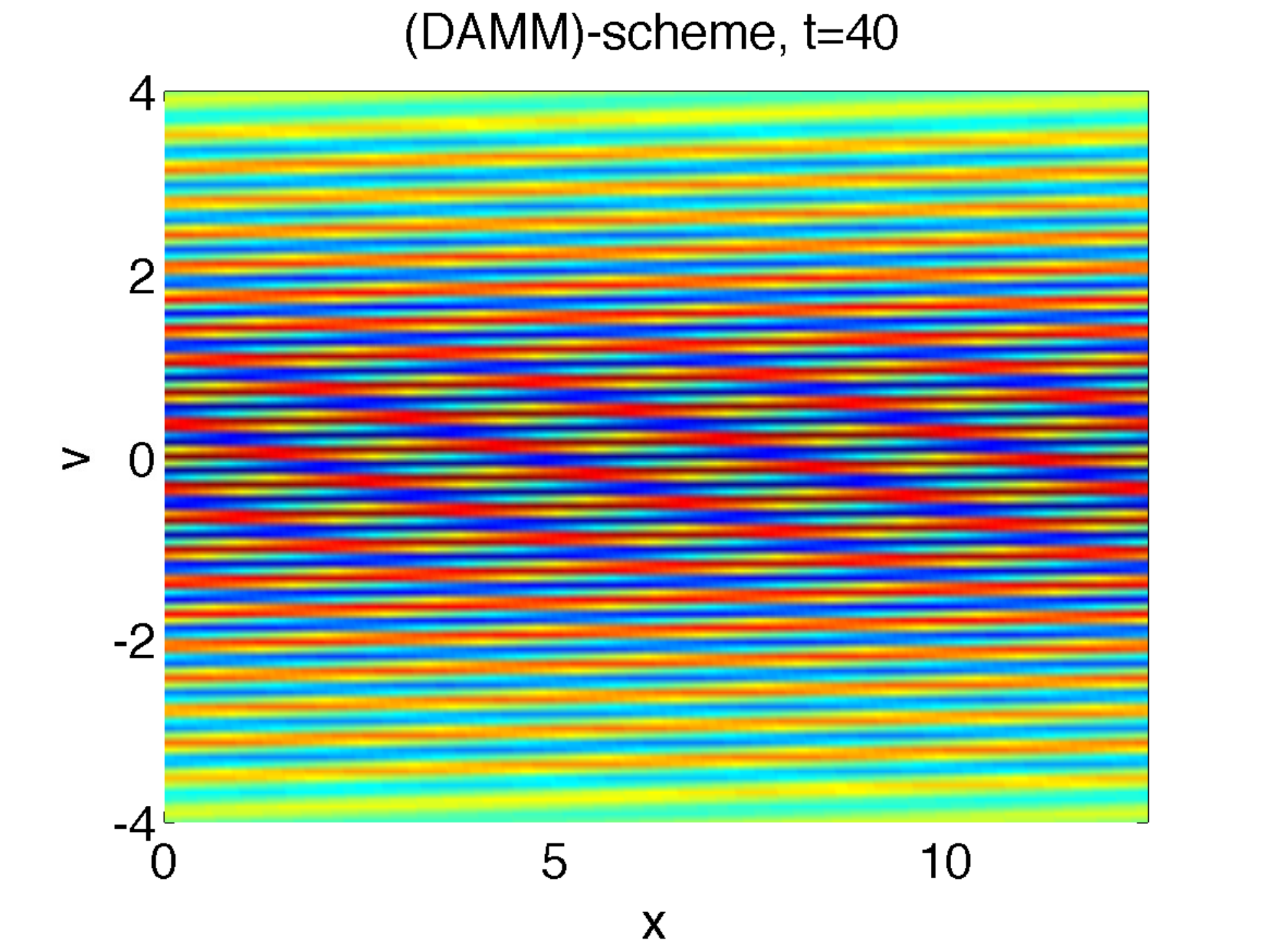}
 	\caption{t=40.}
 	\end{subfigure}
     \begin{subfigure}{.49\textwidth}
   	\centering
   	\includegraphics[width=\linewidth,trim={0cm 0cm 0cm 0cm},clip]{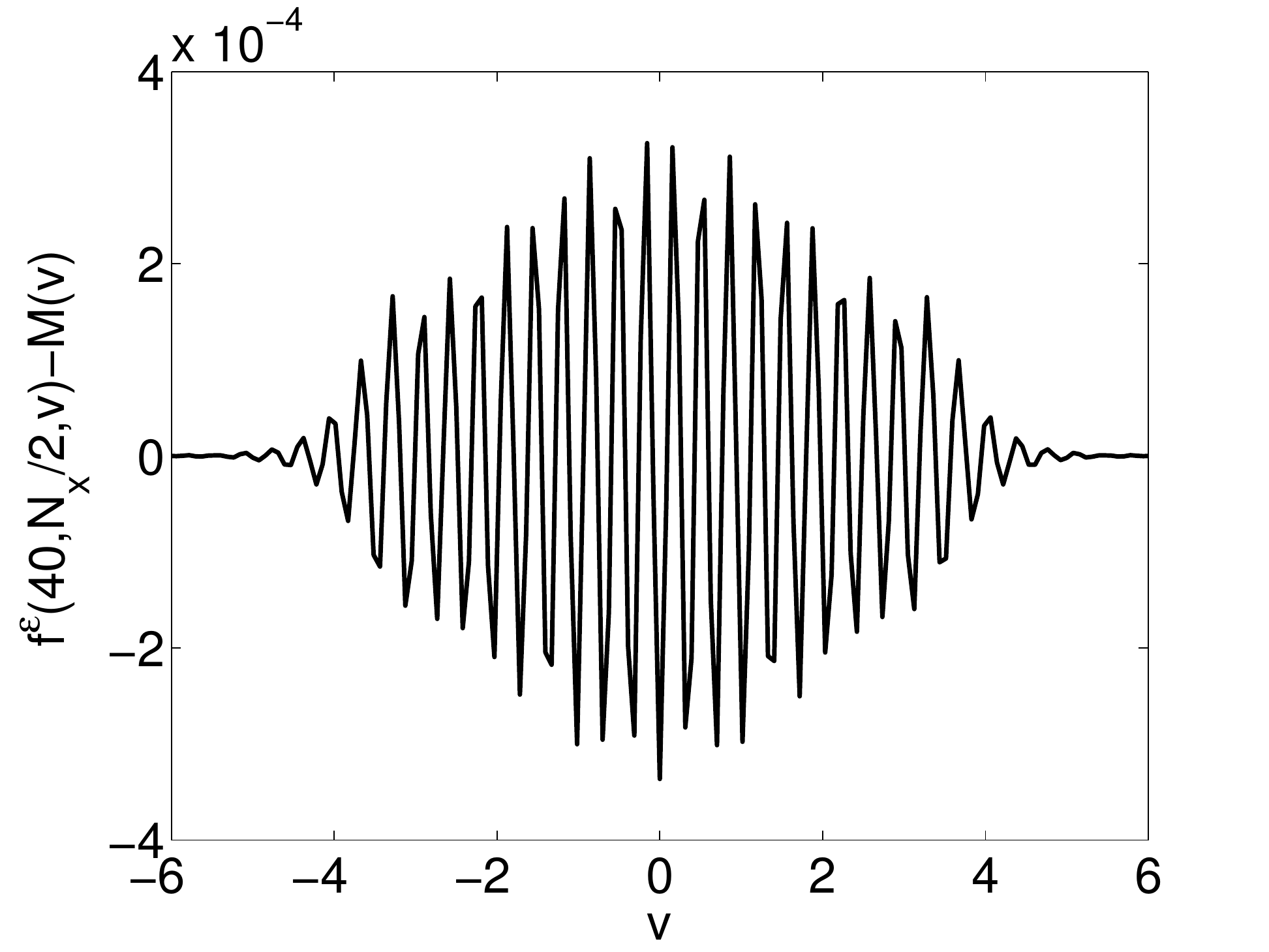}
 	\caption{t=40.}
 	\end{subfigure}
\caption{(Weak Landau damping for $\eps=1$) Zoom of the distribution function $f^{\eps}(t,x,v)-\mathcal{M}(v)$ at different times with $k=0.5$ and $\gamma=0.001$, obtained with the (DAMM)-scheme (A), (C) and (E) and cross-sections at $x=L_x/2$ of this latter (B), (D) and (F).  Mesh size : $N_x=N_y=256$. Other parameters were : $\Delta t = 0.01$, $T=40$, $\eps =1$ and $\sigma = (\Delta x/L_x)^2$.}
\label{cut_linear_landau_damping}
\end{figure}

%%%%%%%%%%%%%%%%%%%%%%%%%%%%%%%%%%%%%%%%%%%%%%%%%%%
\subsection{Numerical simulations for strong Landau damping and  $\eps=1$}
%%%%%%%%%%%%%%%%%%%%%%%%%%%%%%%%%%%%%%%%%%%%%%%%%%%

We shall perform now the numerical simulations for the non-linear Landau damping by taking a stronger perturbation as in the previous study. Nevertheless, we stay in the non-limit regime $\epsilon=1$. In Figure \ref{non_linear_damping}, we plot the distribution function $f^{\eps}(t,x,v)$ at different times, with $\gamma=k=0.3$ and the initial data \eqref{initial_data_vp}. Three levels can be pointed out. Up to $t=10$ (panel (A)), the linear effects dominate and the behavior of the electric energy is very close to the linear case. Then, starting from $t=20$ (panel (B)), the damping is stopped due to particle trapping, for finally leading to saturation at around $t=40$ (panel (D)). The phase-space trapping holes are clearly visible. In Figure \ref{non_linear_damping_cut}, we plot the space average of the distribution function at several times. The formation of several plateaus is clearly visible at time $t=10$ (panel (B)), indicating the trapping of particles in these areas. Over time, this trapping persists, although the numerical diffusion tends to damp these states. Indeed, due to the numerical dissipation, the filamentation is progressively eliminated when the filamentation scale become smaller than the velocity grid $\Delta v$. 

In Figure \ref{non_linear_energy} (A), we plot the evolution of the electric energy (in log-scale) as a function of time. Contrary to the weak Landau damping, the growth or decay rates of the oscillations are not known. Nevertheless, we can compare the (DAMM)-scheme to the reference spectral scheme. We observe a good correspondence between these two schemes. In order to carry on the investigations of the strong Landau damping, we look at the evolution of some particular quantities.  The Vlasov-Poisson system is well-known to conserve the total particle number (mass), the momentum, the total energy, the $L^p$-norms and the entropy. These quantities are given respectively by ($\Omega_x=(-L_x,L_x)$)

\begin{align}
M^{\eps}(t) &:= \int_{\mathbb{R} \times \Omega_x} f^{\eps}(t,x,v) \, dx \, dv \,, \\
\mathcal{M}_o^{\eps}(t)& := \int_{\mathbb{R} \times \Omega_x} v \, f^{\eps}(t,x,v) \, dx \, dv \,, \\
\mathcal{T}^{\eps}(t) &:= \mathcal{E}^{\eps} + \mathcal{P}^{\eps}\,, \\
C_p^{\eps}(t) &:= \Bigg(\int_{\mathbb{R} \times \Omega_x} |f^{\eps}(t,x,v)|^p \, dx \, dv \,\Bigg)^{1/p} \,, \\
S^{\eps}(t) &:= \int_{\mathbb{R} \times \Omega_x} - f^{\eps}(t,x,v) \, \ln( f^{\eps}(t,x,v)) \, dx \, dv \,.
\end{align}

Due to the presence of the stabilization parameter $\sigma$, the conserved quantities introduced previously are no more constant over time when computed via the (DAMM)-scheme. We investigated this in the panels of Figure \ref{non_linear_energy}. In particular in the panel (B) we see that the total energy is not conserved by the (DAMM)-scheme with $3 \, \%$ of deviation from its initial value. Analogous observations can be done for the mass (panel (C)), the entropy (panel (D)) and the $L^2$-norm (panel (F)). Nevertheless, the (DAMM)-scheme conserves the momentum (which is null for the initial condition \eqref{initial_data_vp}), unlike the spectral scheme. Despite the non-conservation of these quantities, their deviations from their initial value stay weak.\\ 

Having carefully considered the Landau damping through several numerical simulations performed by the (DAMM)-scheme, we are interested now in the study of the two-stream instability. Since the Landau damping does not attain an equilibrium (due to the continuous filamentation), it is not suitable for investigating the limit regime $\eps \to 0$.  As we will see, things are different in the case of the two-stream instability, which permit to investigate the limit $\eps \to 0$. 

\begin{figure}[ht]
\centering
	\begin{subfigure}{.42\textwidth}
  	\centering
  	\includegraphics[width=\linewidth,trim={0cm 0cm 0cm 0cm},clip]{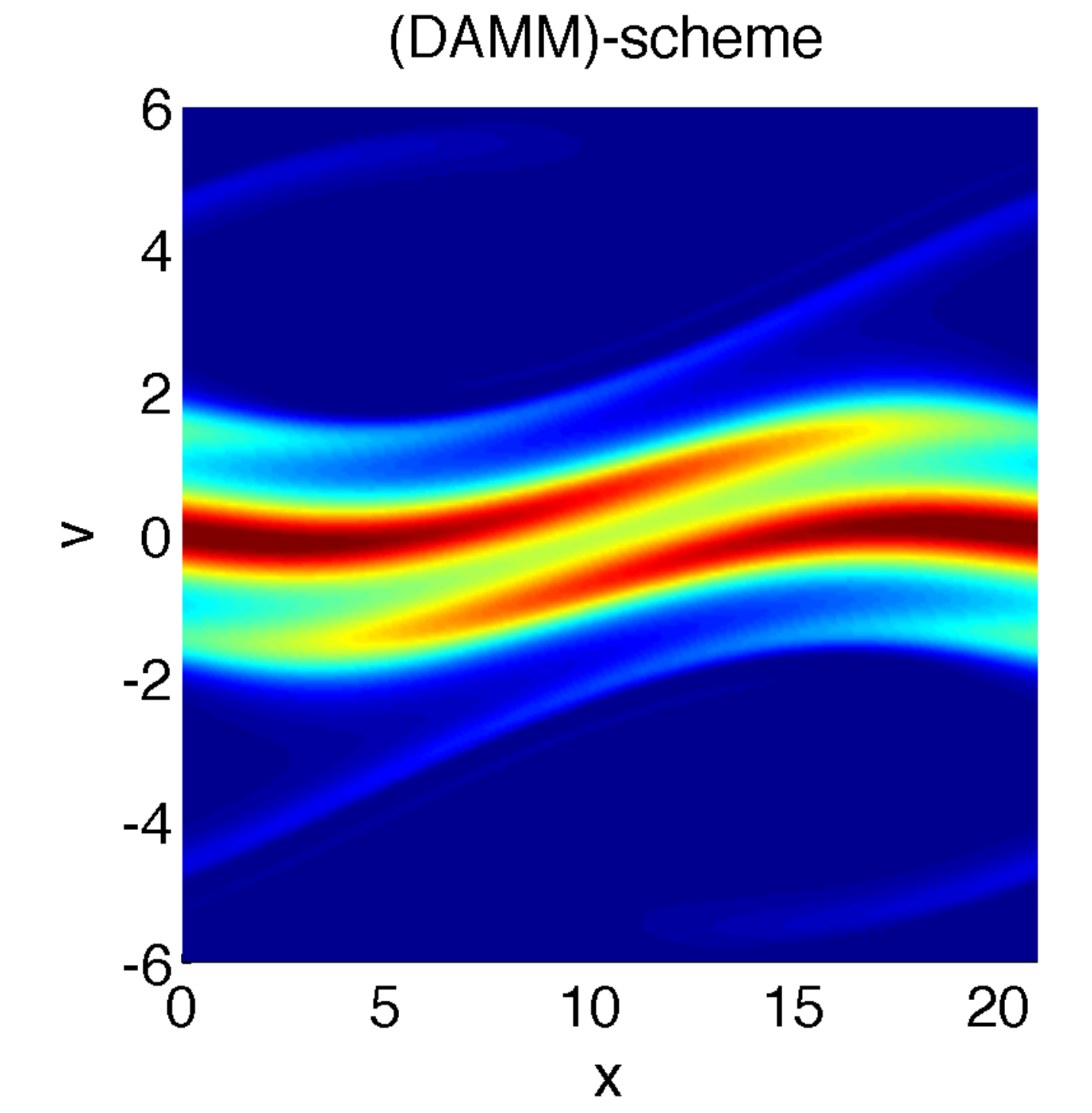}
	\caption{t=10.}
	\end{subfigure}
    \begin{subfigure}{.42\textwidth}
  	\centering
  	\includegraphics[width=\linewidth,trim={0cm 0cm 0cm 0cm},clip]{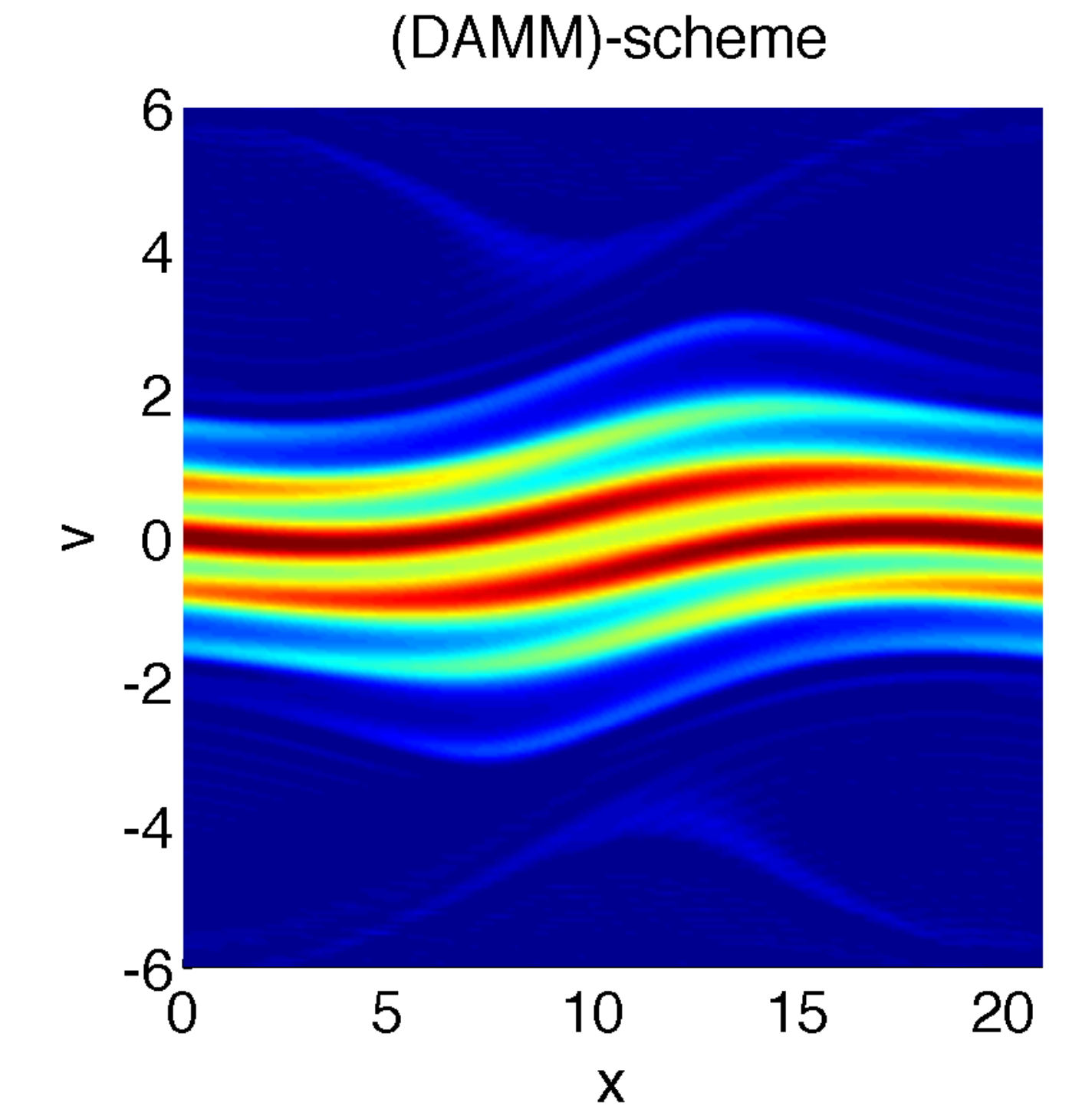}
	\caption{t=20.}
	\end{subfigure} \\ \vspace{0.5cm}
	\begin{subfigure}{.42\textwidth}
  	\centering
  	\includegraphics[width=\linewidth,trim={0cm 0cm 0cm 0cm},clip]{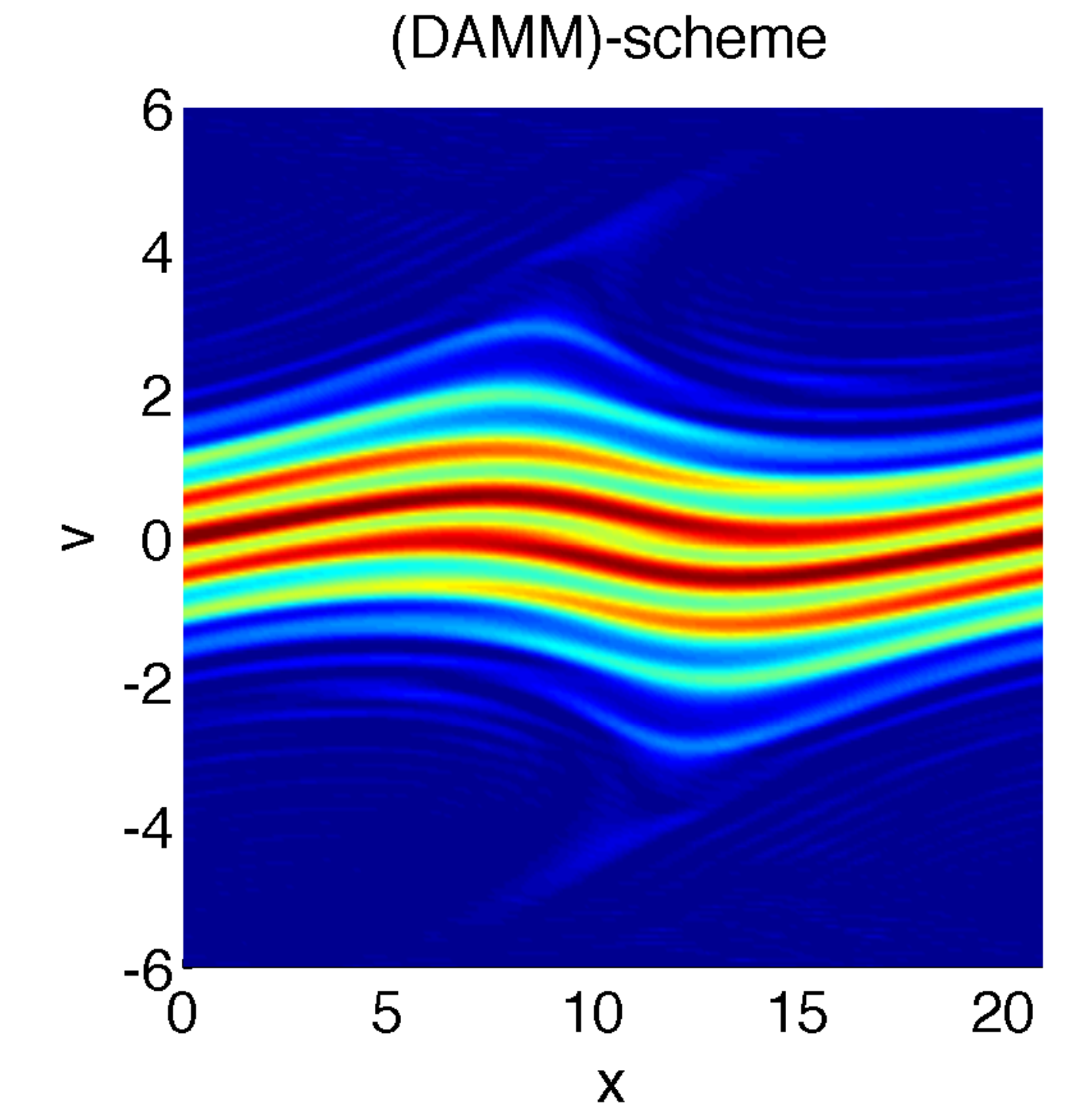}
	\caption{t=30.}
	\end{subfigure}
     \begin{subfigure}{.42\textwidth}
   	\centering
   	\includegraphics[width=\linewidth,trim={0cm 0cm 0cm 0cm},clip]{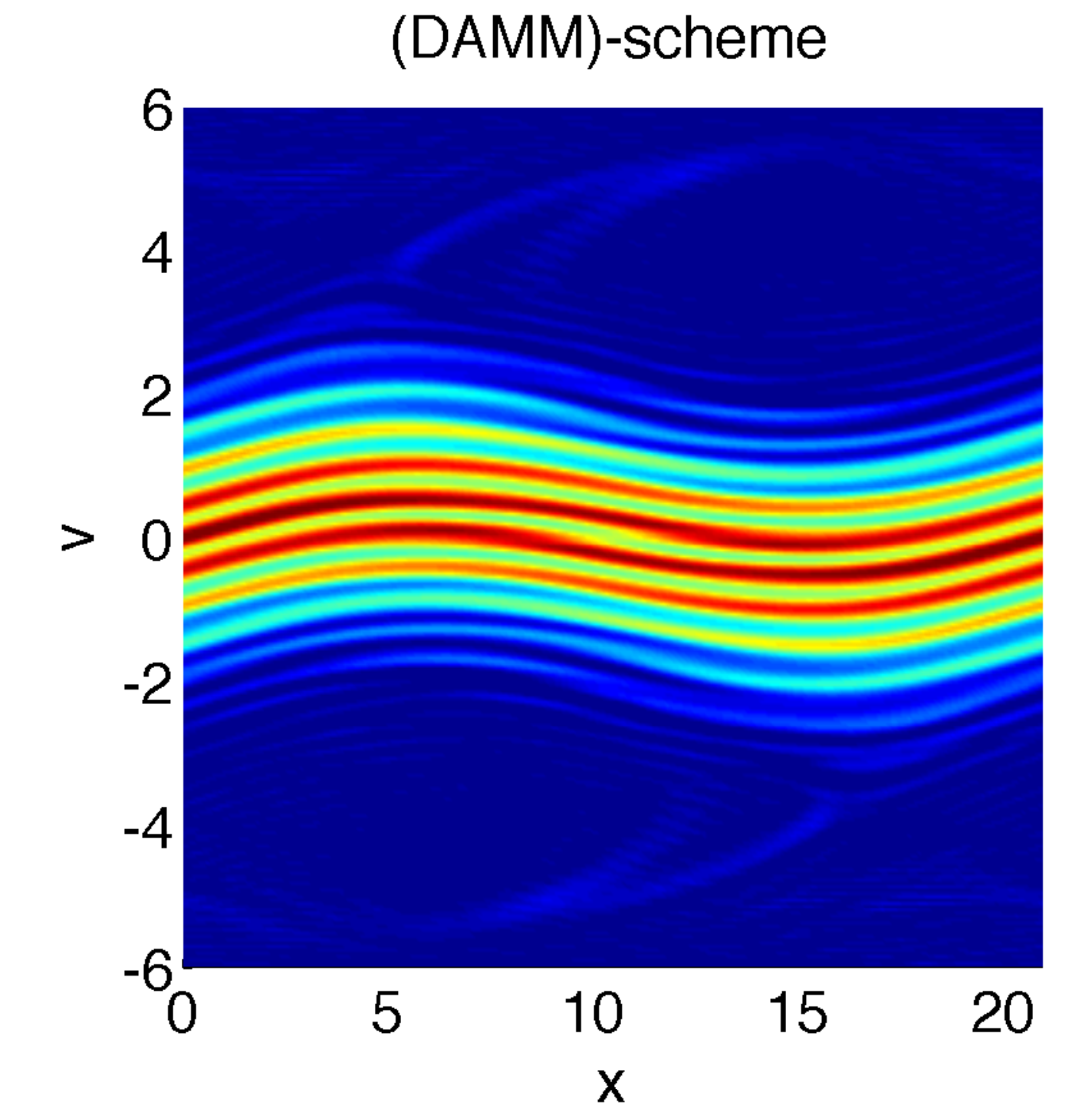}
 	\caption{t=40.}
 	\end{subfigure} \vspace{0.5cm}
 	\begin{subfigure}{.42\textwidth}   	
	\centering
   	\includegraphics[width=\linewidth,trim={0cm 0cm 0cm 0cm},clip]{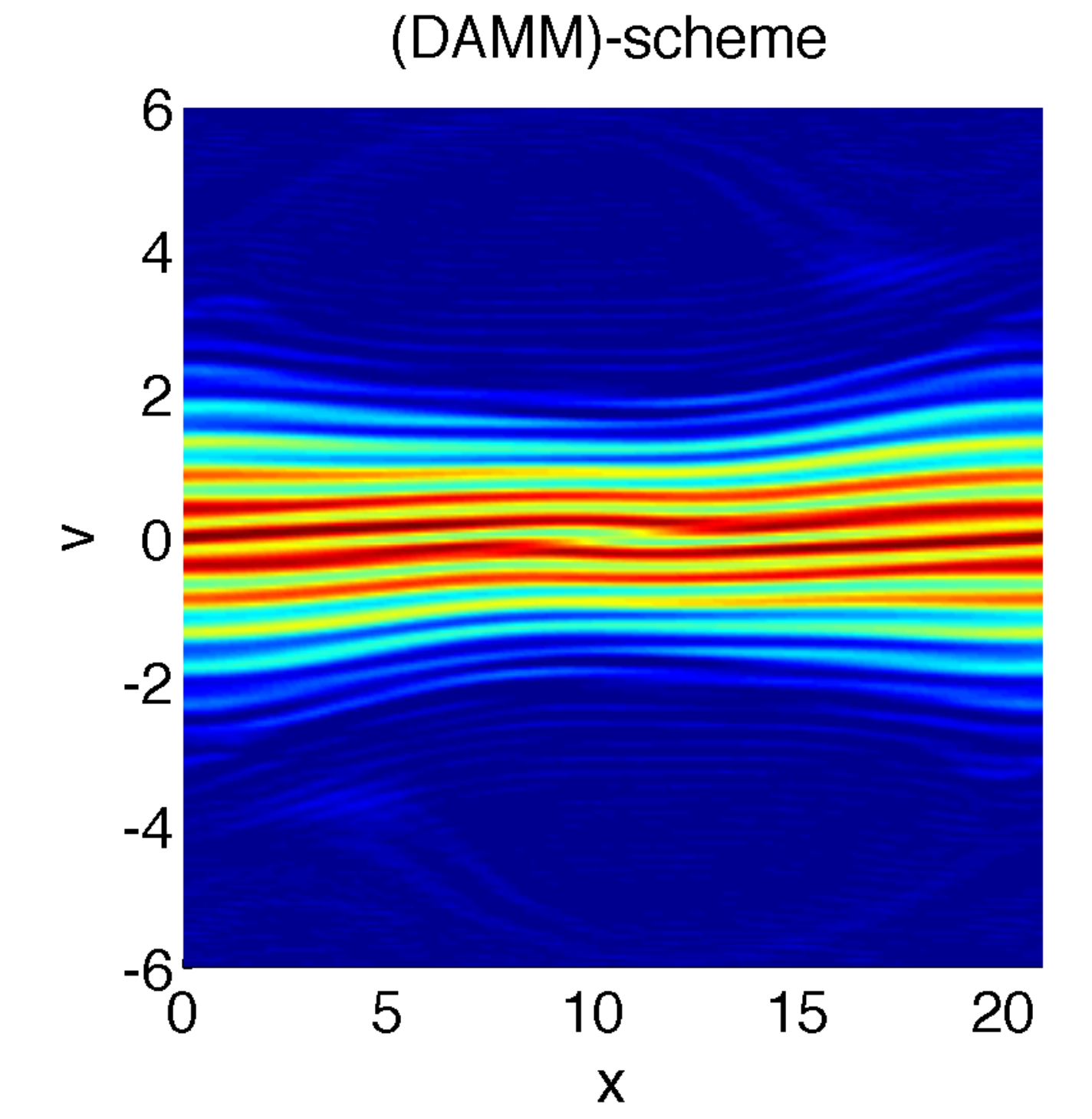}
 	\caption{t=50.}
 	\end{subfigure}
     \begin{subfigure}{.42\textwidth}
   	\centering
   	\includegraphics[width=\linewidth,trim={0cm 0cm 0cm 0cm},clip]{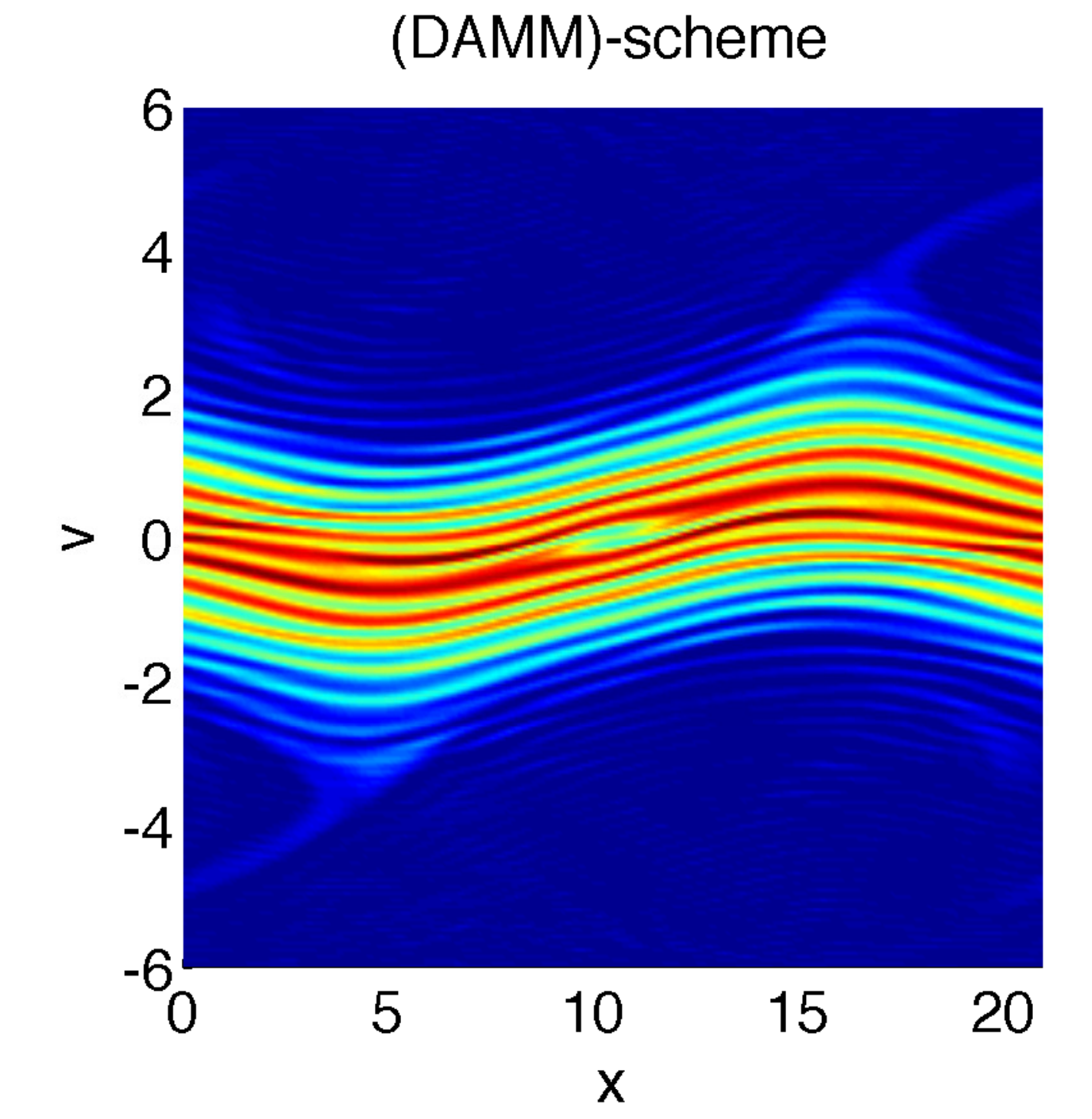}
 	\caption{t=60.}
 	\end{subfigure}
\caption{(Strong Landau damping for $\epsilon=1$) Zoom of the distribution function $f^{\eps}(t,x,v)$ at different times with $k=\gamma=0.3$, obtained with the (DAMM)-scheme. Mesh size : $N_x=N_y=256$.  Parameters were : $\Delta t = 0.01$, $T=60$, $\eps =1$ and $\sigma = (\Delta x/L_x)^2$.}
\label{non_linear_damping}
\end{figure}

\begin{figure}[ht]
\centering
	\begin{subfigure}{.49\textwidth}
  	\centering
  	\includegraphics[width=\linewidth,trim={0cm 0cm 0cm 0cm},clip]{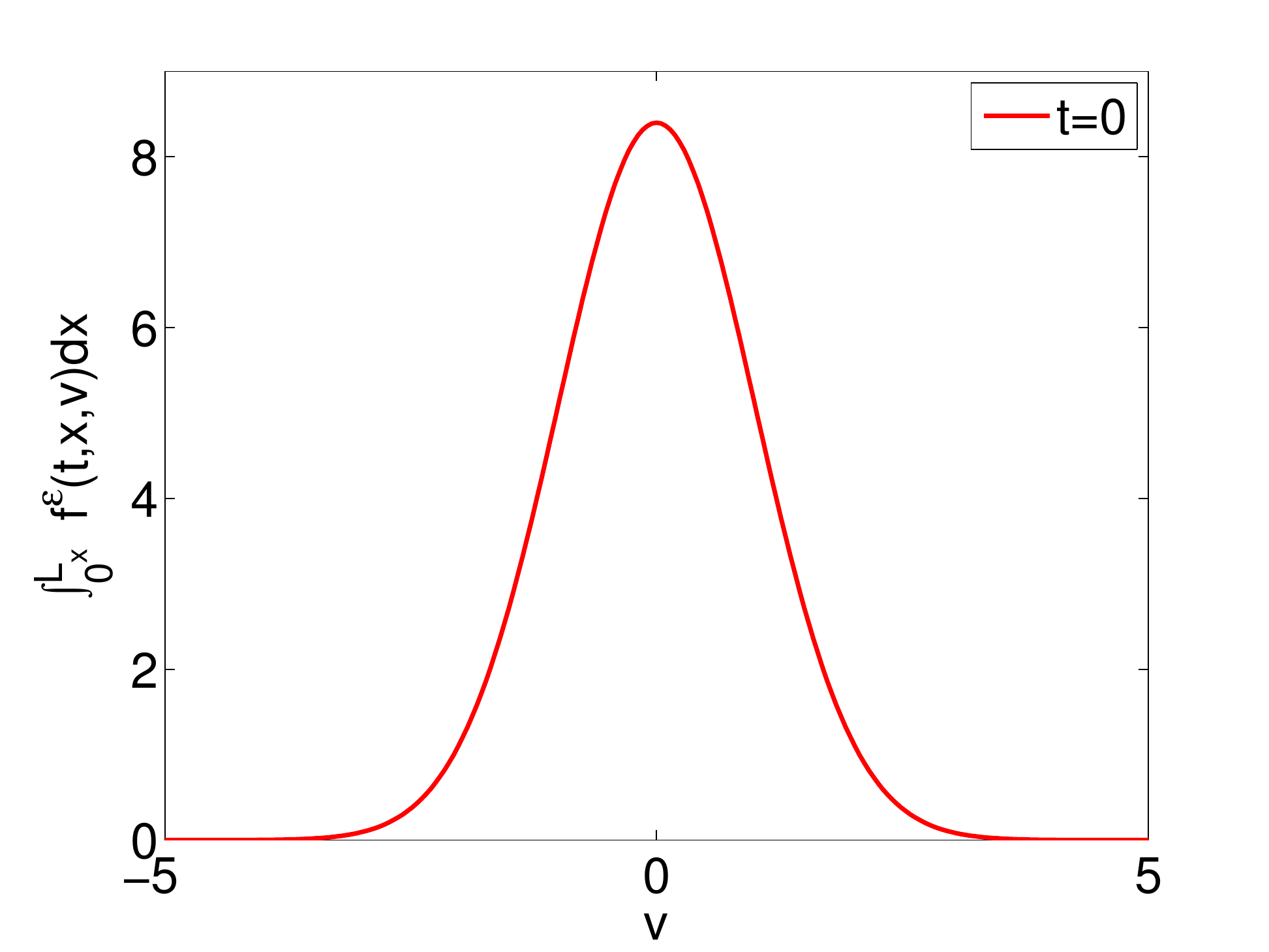}
	\caption{t=0.}
	\end{subfigure}
    \begin{subfigure}{.49\textwidth}
  	\centering
  	\includegraphics[width=\linewidth,trim={0cm 0cm 0cm 0cm},clip]{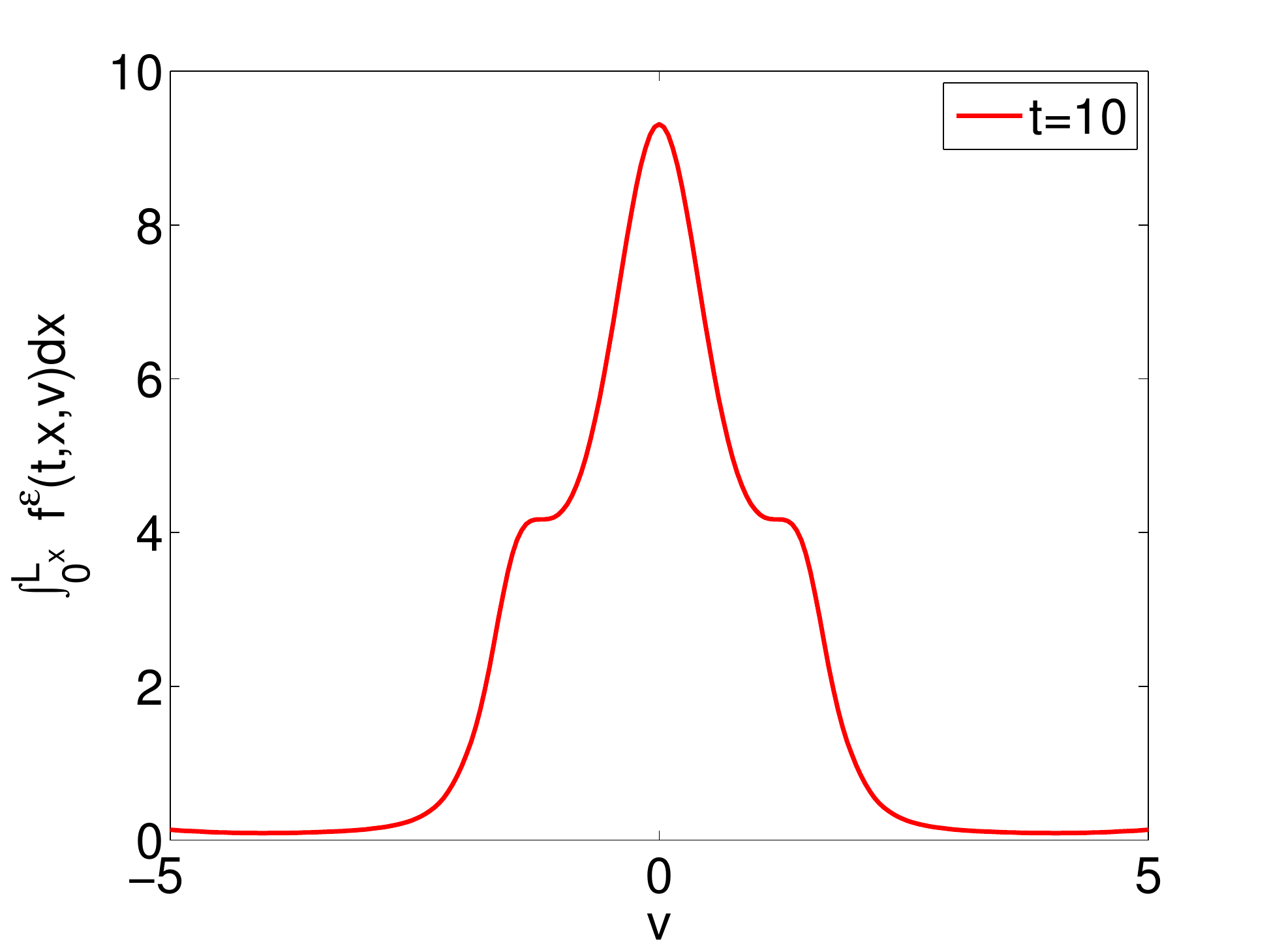}
	\caption{t=10.}
	\end{subfigure} \\ \vspace{0.5cm}
	\begin{subfigure}{.49\textwidth}
  	\centering
  	\includegraphics[width=\linewidth,trim={0cm 0cm 0cm 0cm},clip]{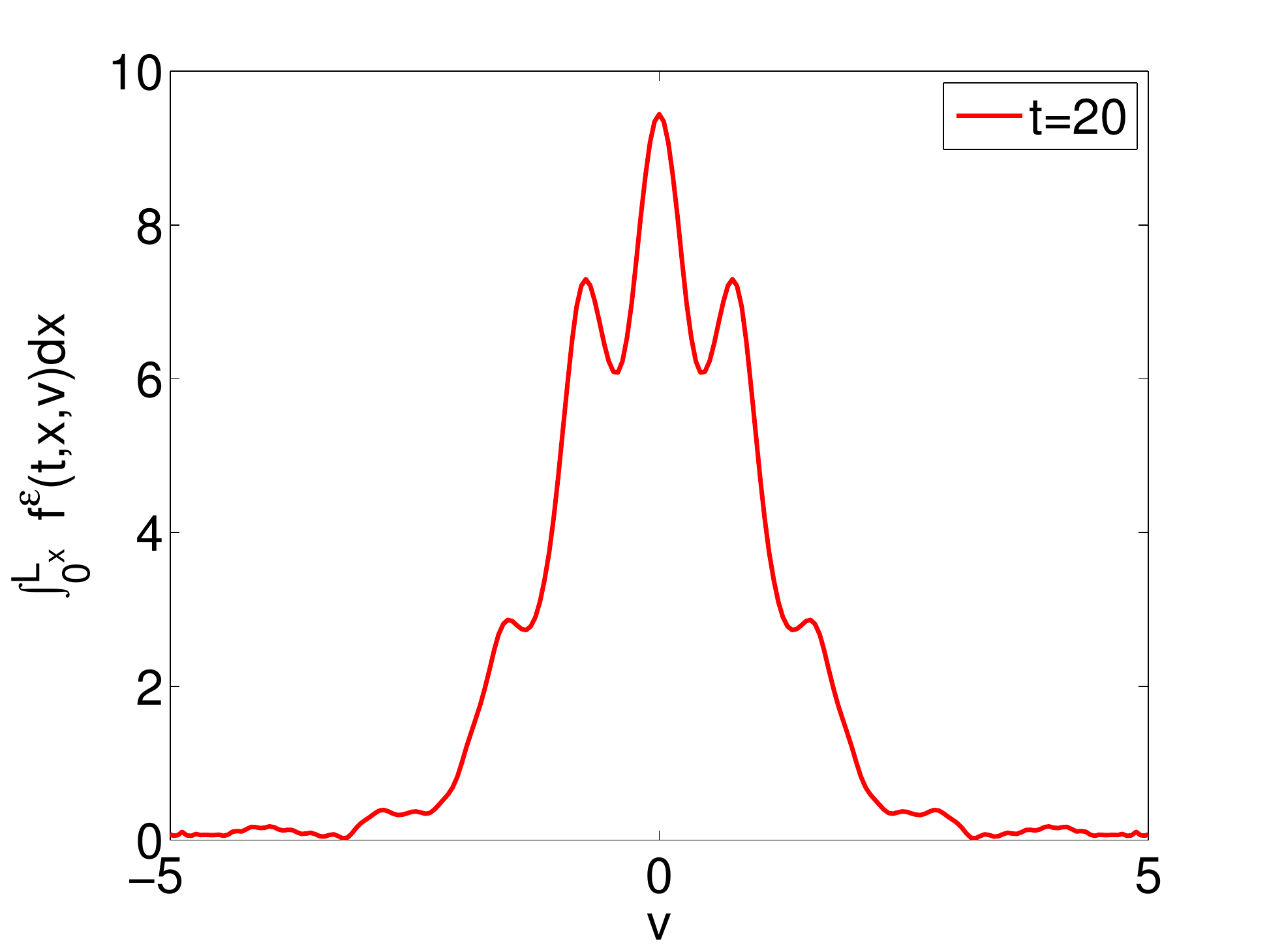}
	\caption{t=20.}
	\end{subfigure}
     \begin{subfigure}{.49\textwidth}
   	\centering
   	\includegraphics[width=\linewidth,trim={0cm 0cm 0cm 0cm},clip]{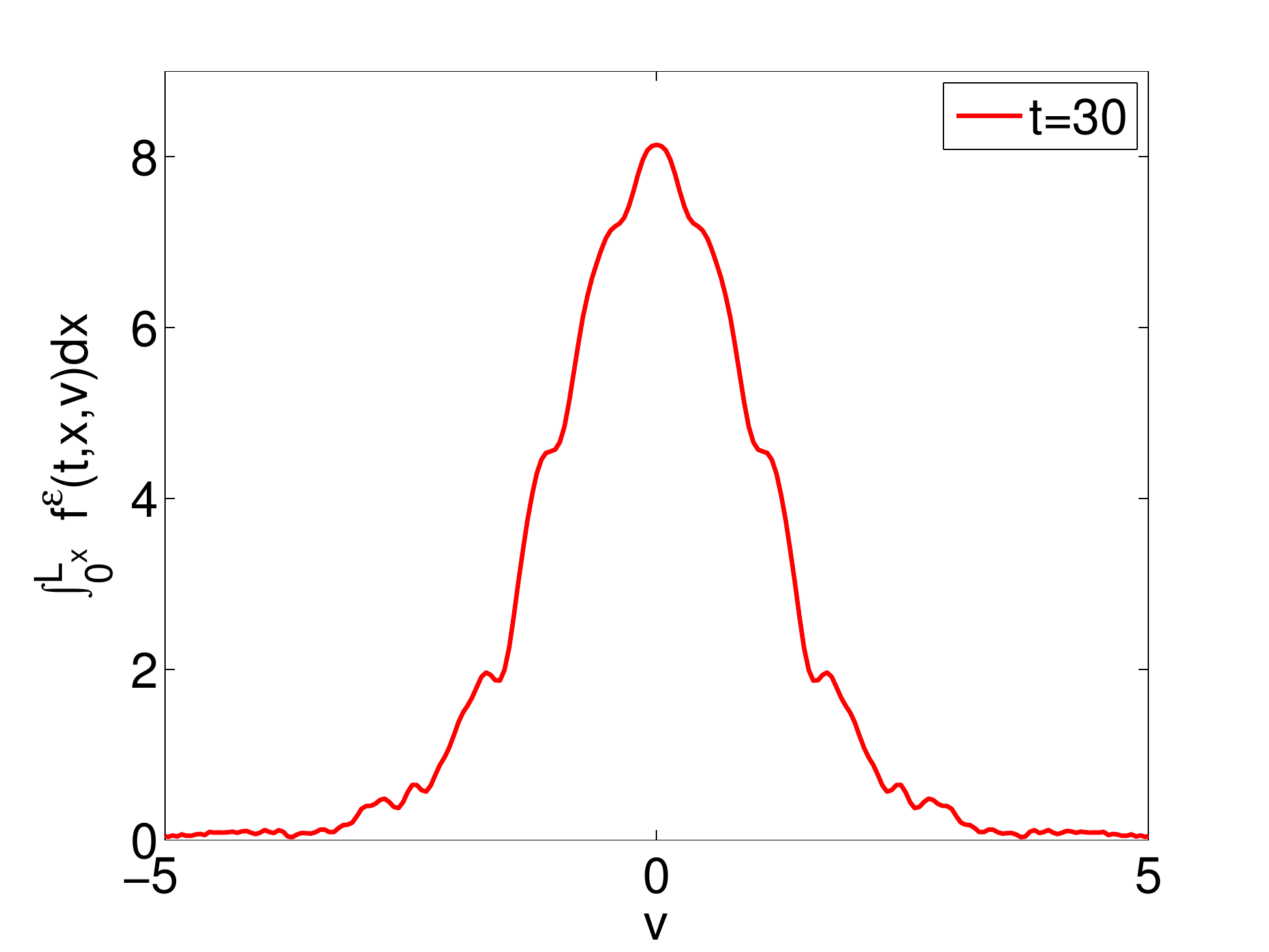}
 	\caption{t=30.}
 	\end{subfigure} \vspace{0.5cm}
 	\begin{subfigure}{.49\textwidth}   	
	\centering
   	\includegraphics[width=\linewidth,trim={0cm 0cm 0cm 0cm},clip]{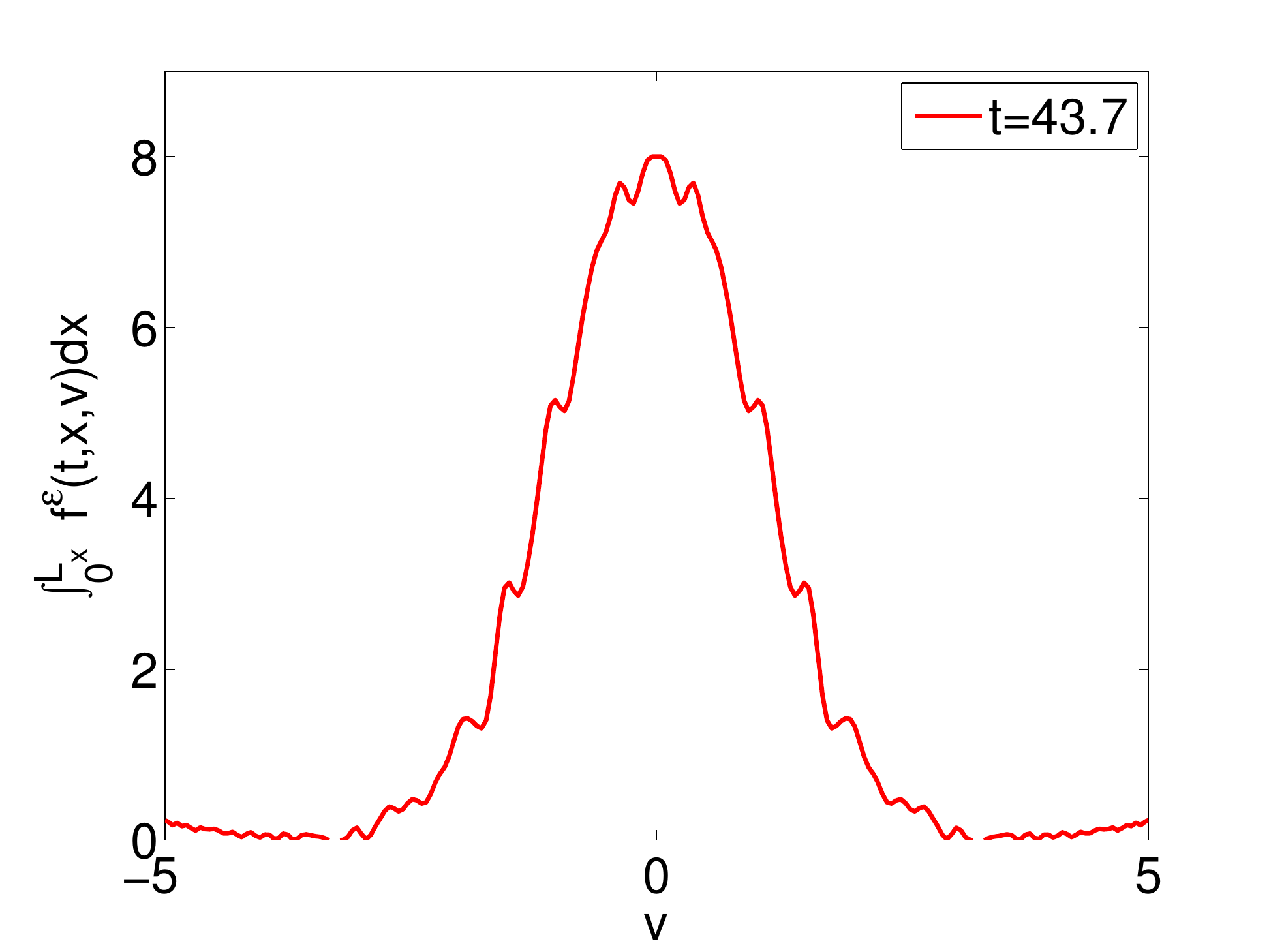}
 	\caption{t=43.7.}
 	\end{subfigure}
     \begin{subfigure}{.49\textwidth}
   	\centering
   	\includegraphics[width=\linewidth,trim={0cm 0cm 0cm 0cm},clip]{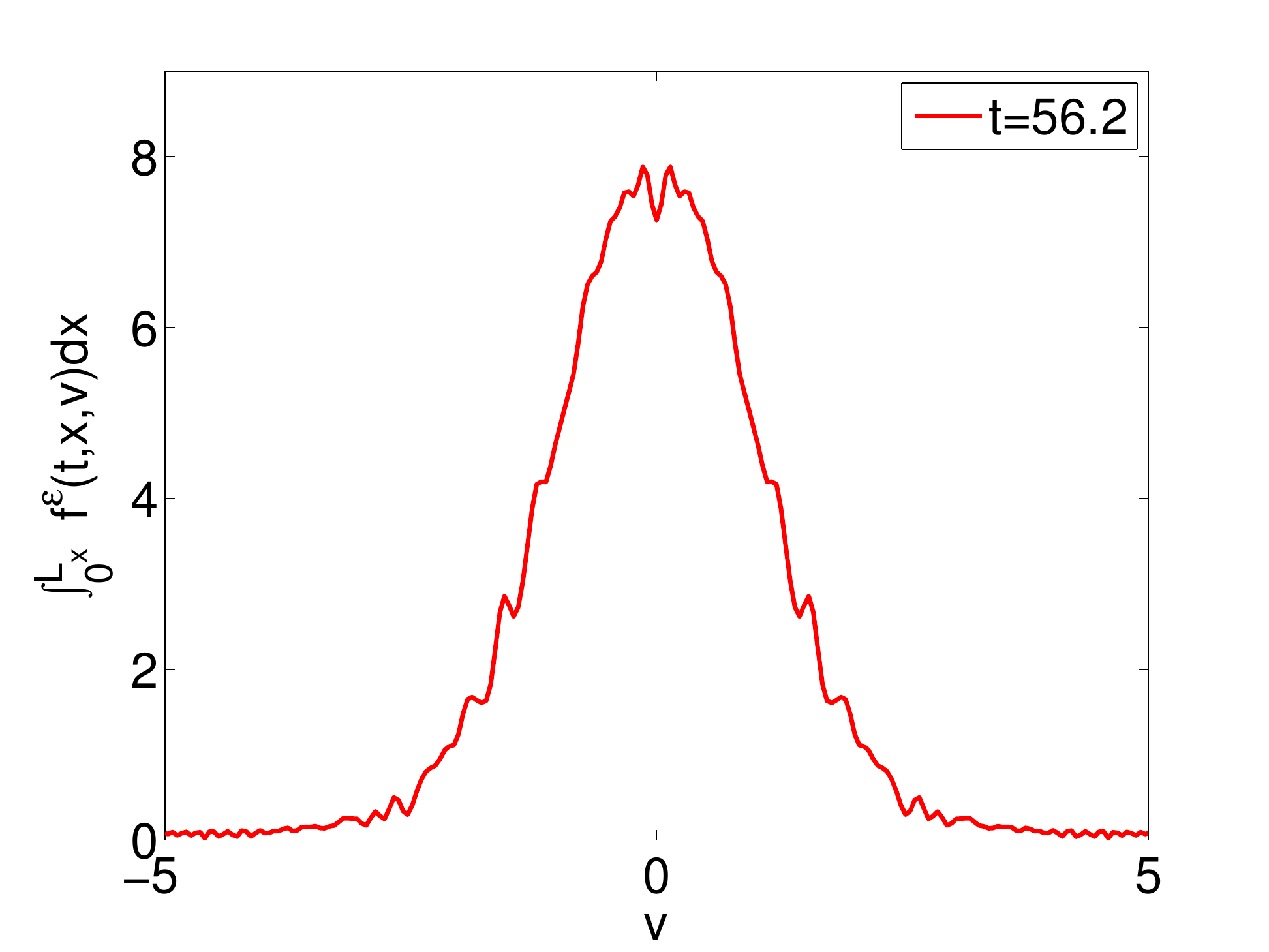}
 	\caption{t=56.2.}
 	\end{subfigure}
\caption{(Strong Landau damping for $\eps=1$) Spatial average of the distribution function $f^{\eps}(t,x,v)$ at different times with $k=\gamma=0.3$, obtained with the (DAMM)-scheme. Mesh size : $N_x=N_y=256$.  Parameters were : $\Delta t = 0.01$, $T=60$, $\eps =1$ and $\sigma = (\Delta x/L_x)^2$.}
\label{non_linear_damping_cut}
\end{figure}

\begin{figure}[ht]
\centering
	\begin{subfigure}{.48\textwidth}
  	\centering
  	\includegraphics[width=\linewidth,trim={0cm 0cm 0cm 0cm},clip]{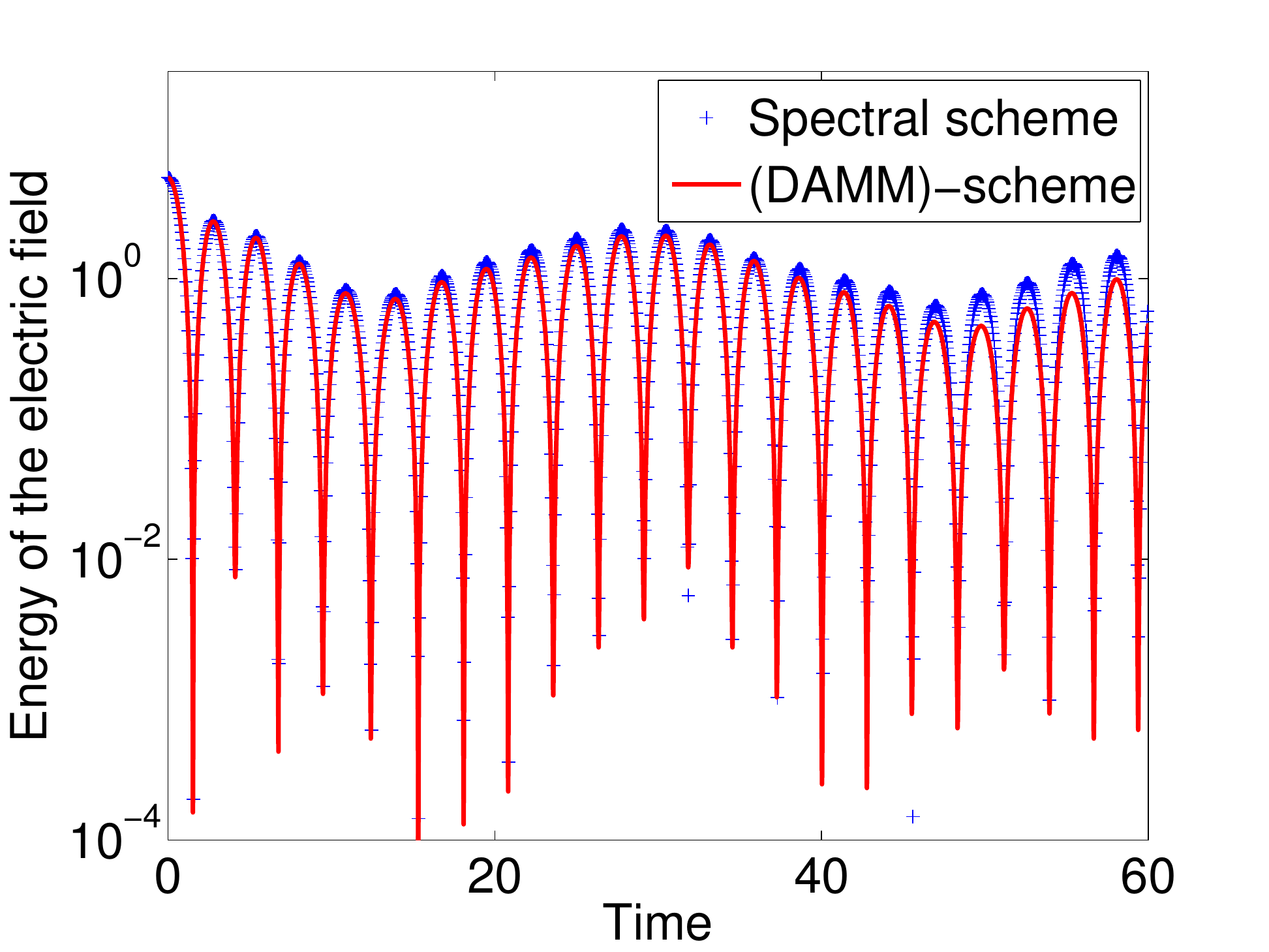}
	\caption{Strong Landau damping.}
	\end{subfigure}
    \begin{subfigure}{.48\textwidth}
  	\centering
  	\includegraphics[width=\linewidth,trim={0cm 0cm 0cm 0cm},clip]{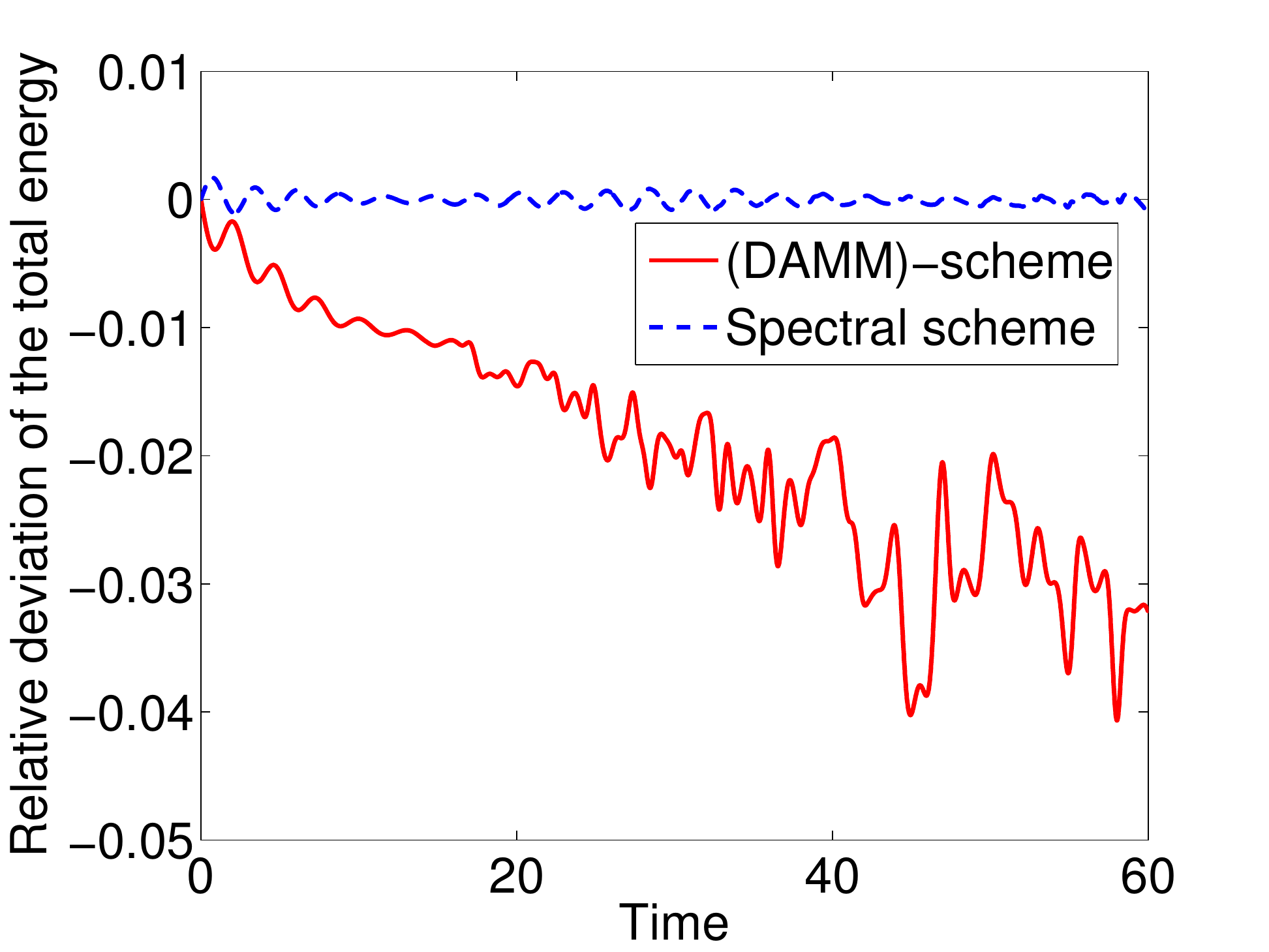}
	\caption{Total energy.}
	\end{subfigure} \\ 
	\begin{subfigure}{.48\textwidth}
  	\centering
  	\includegraphics[width=\linewidth,trim={0cm 0cm 0cm 0cm},clip]{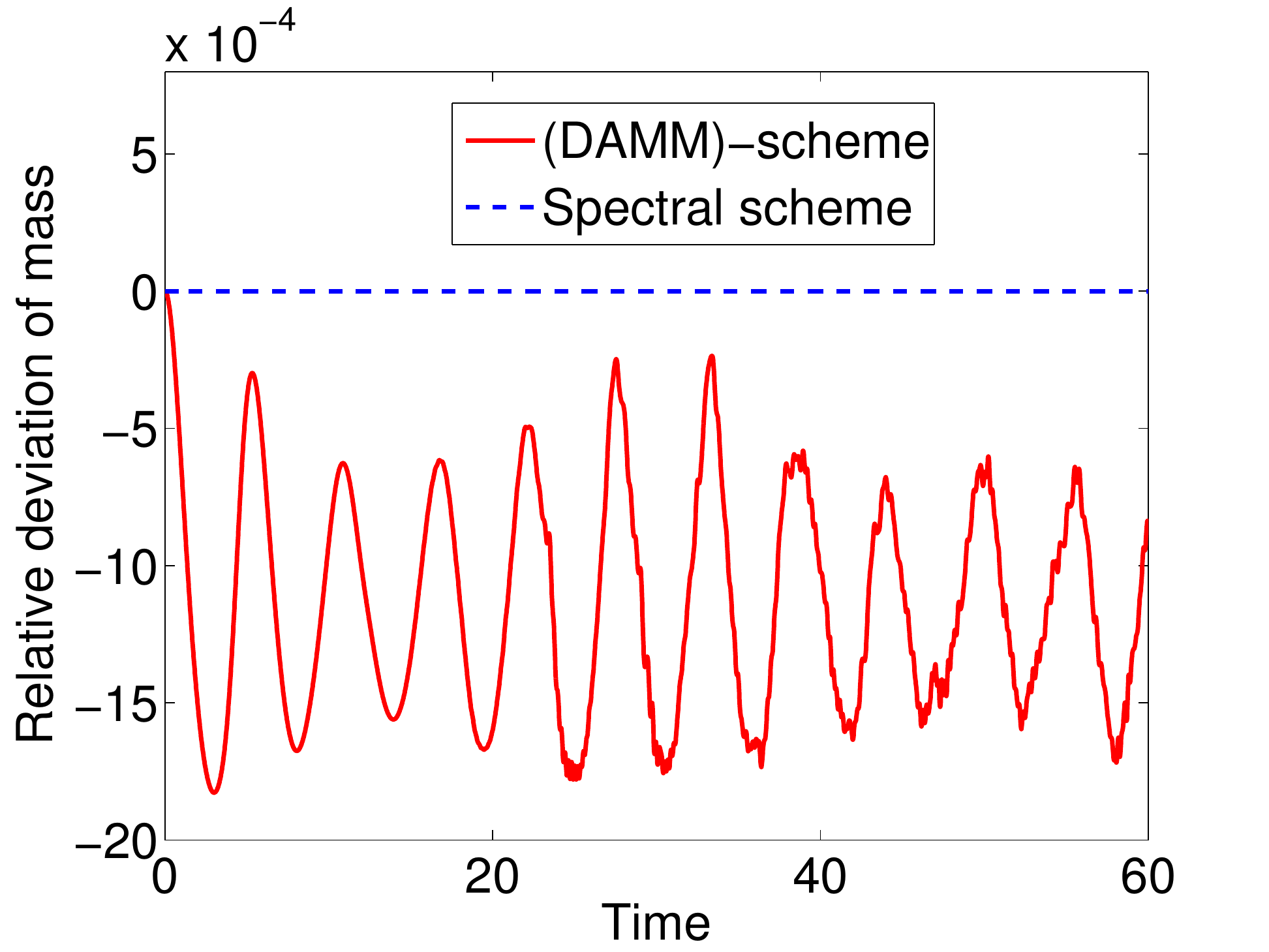}
	\caption{Mass.}
	\end{subfigure}
    \begin{subfigure}{.48\textwidth}
  	\centering
  	\includegraphics[width=\linewidth,trim={0cm 0cm 0cm 0cm},clip]{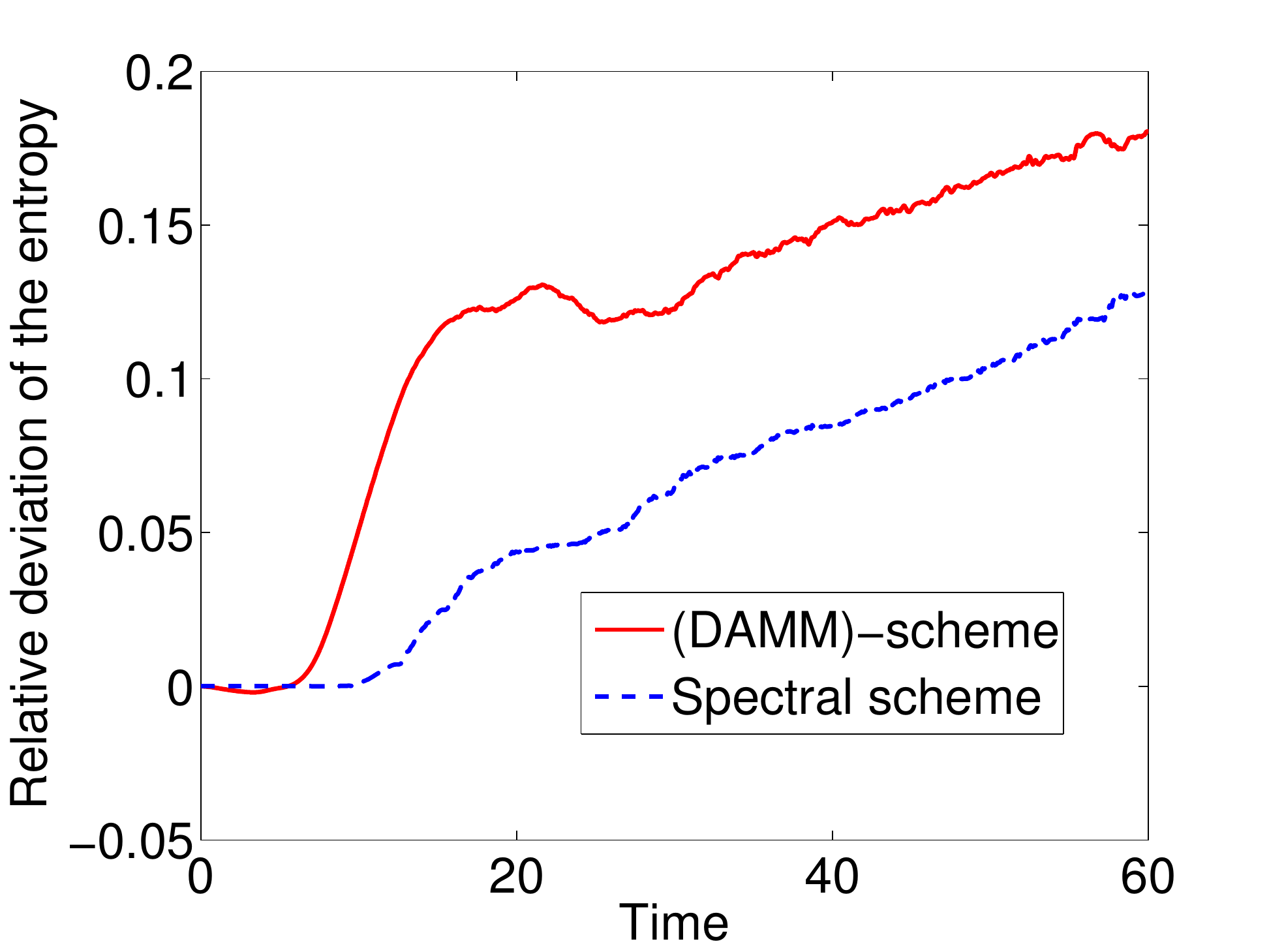}
	\caption{Entropy.}
	\end{subfigure} \\ 
	\begin{subfigure}{.48\textwidth}
  	\centering
  	\includegraphics[width=\linewidth,trim={0cm 0cm 0cm 0cm},clip]{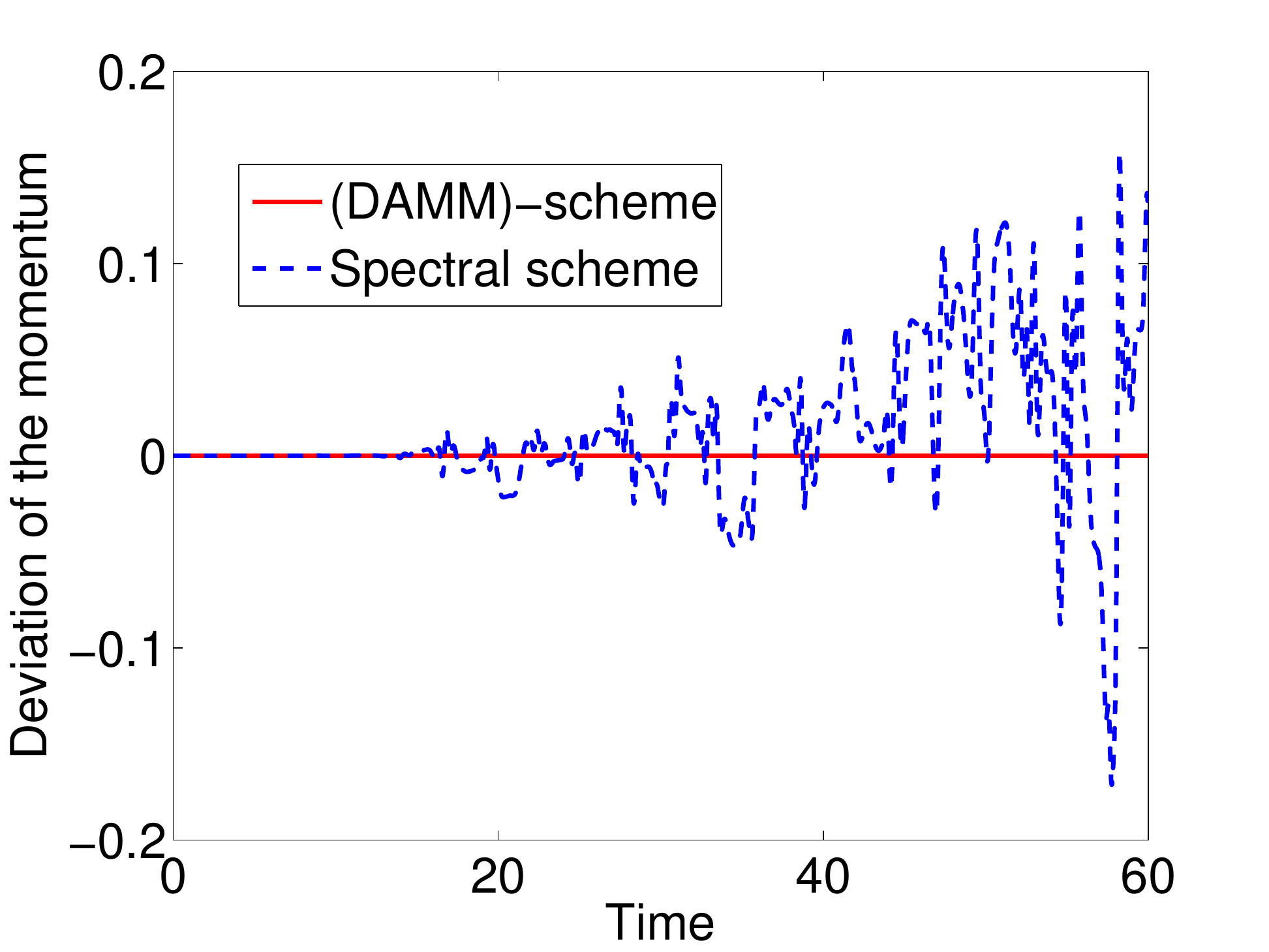}
	\caption{Momentum.}
	\end{subfigure}
    \begin{subfigure}{.48\textwidth}
  	\centering
  	\includegraphics[width=\linewidth,trim={0cm 0cm 0cm 0cm},clip]{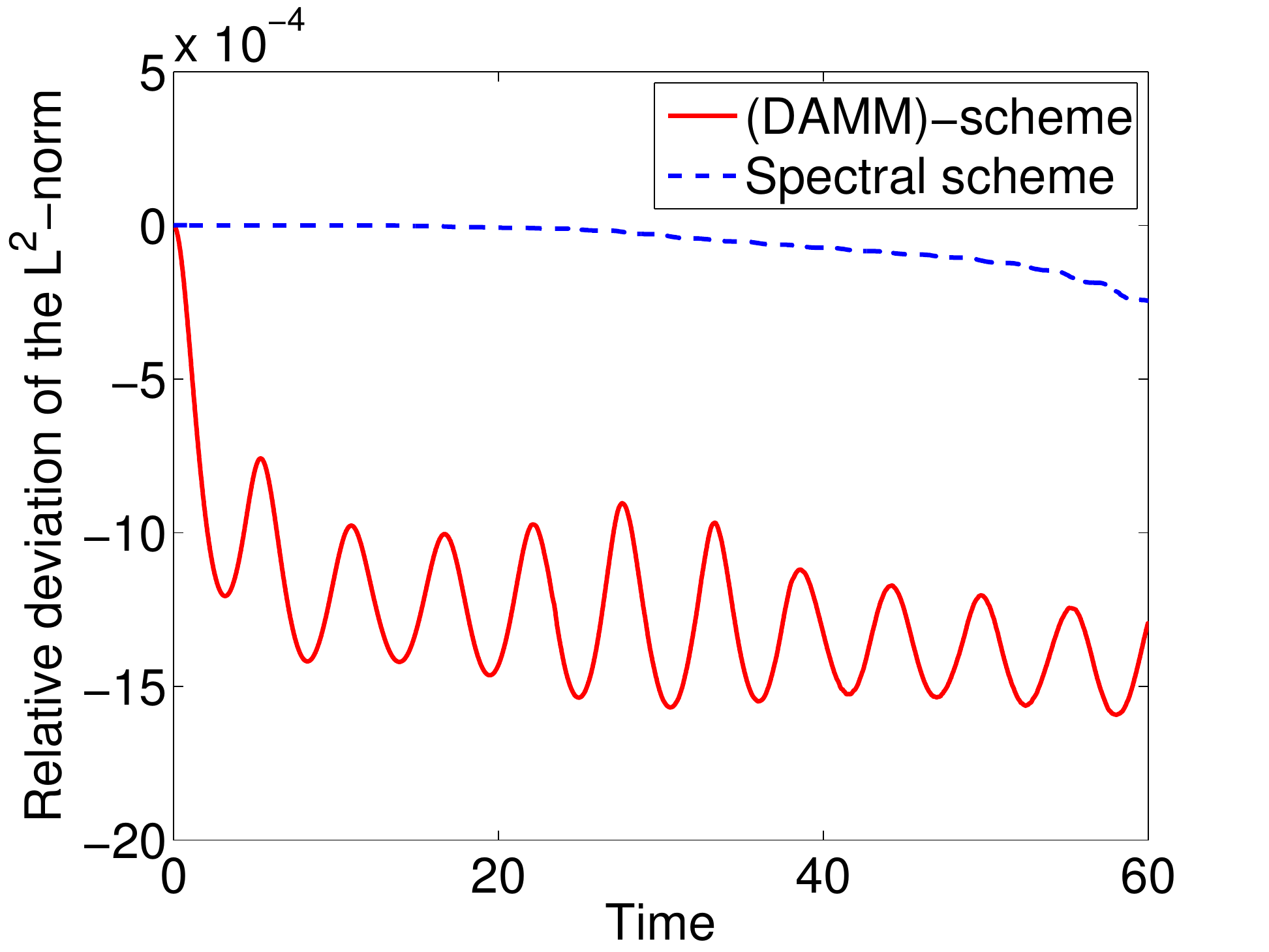}
	\caption{$L^2$-norm.}
	\end{subfigure} \\ 
\caption{(Strong Landau damping for $\eps=1$) Energy of the electric field versus time (A) and deviation over time for both (DAMM) and spectral schemes of several quantities (B), (C), (D), (E) and (F). Mesh size: $N_x=N_y=256$. $T=60$, $\Delta t=0.01$, $\eps=1$, and $\sigma = (\Delta x /L_x)^2$.}
\label{non_linear_energy}
\end{figure}

%%%%%%%%%%%%%%%%%%%%%%%%%%%%%%%%%%%%%%%%%%
\subsection{Numerical simulations of the two-stream instability, study of the limit case $\epsilon \to 0$.}

\subsubsection{Two-stream instability in the non-limit case $\eps=1$}

The two stream instability can be thought as the inverse of the Landau damping. It occurs when the velocity of the particles is slightly greater than the wave velocity $v_{\phi}$. The instability causes a transfer of energy, from the particles to electric field, unlike the Landau damping phenomenon where the exchange of energy occurs from the electric field to the particles. Thus, to simulate this instability, one imposes the following initial data:

\be \label{initial_data_ts}
f_{in}^1(x,v) =\frac{1}{\sqrt{2\pi}}\,(1+\gamma \cos(kx))\,\frac{1}{2}\,(e^{-(v-3)^2/2}+e^{-(v+3)^2/2})\,.
\ee

As a first step, we keep $\eps=1$. We choose $\gamma=0.001$, $k=0.2$, $L_v=10$, $L_x=2\,\pi/k$ and $\sigma = (\Delta x/L_x)^2$. In Figure \ref{two_streams_instability_damm}, we plot the distribution function $f^{\eps}(t,x,v)$, solution of \eqref{Vlasov_Poisson_bracket} at different times, with the previous initial condition \eqref{initial_data_ts}. The panels (A) and (C) refer to the (DAMM)-scheme whereas the panels (B) and (D) correspond to the reference spectral scheme. In both cases, the instability grows until the non-linear effects become significant. Over time, the non-linear effects cause a trapping phenomenon. To push ahead with the investigations, we plot in Figure \ref{TS_cons} the evolution over time  of $\ln(||E^{\eps}(t,\cdot)||_1)$ (A). The analytic value of the growth rate for the electric field, {\it i.e.} $\omega_i(k=0.2)=0.2548$, is very close to the numerical value observed and the curves obtained via the two numerical schemes (DAMM and spectral) coincide. As we made for the non-linear Landau damping, we examine in Figure \ref{TS_cons} the conservation of several quantities over time, for both (DAMM) and spectral schemes. As for the non-linear Landau damping, only the momentum (panel (E)) is well conserved by the (DAMM)-scheme, the total energy (panel (B)), the mass (panel (C)), the entropy (panel (D)) and the $L^2$-norm (panel (F)) indicate weak deviations from their initial value. Nevertheless, we will see in the next section the main advantages of the (DAMM)-scheme, when compared to standard schemes.

\begin{figure}[h]
\centering
	\begin{subfigure}{.49\textwidth}
  	\centering
  	\includegraphics[width=\linewidth,trim={0cm 0cm 0cm 0cm},clip]{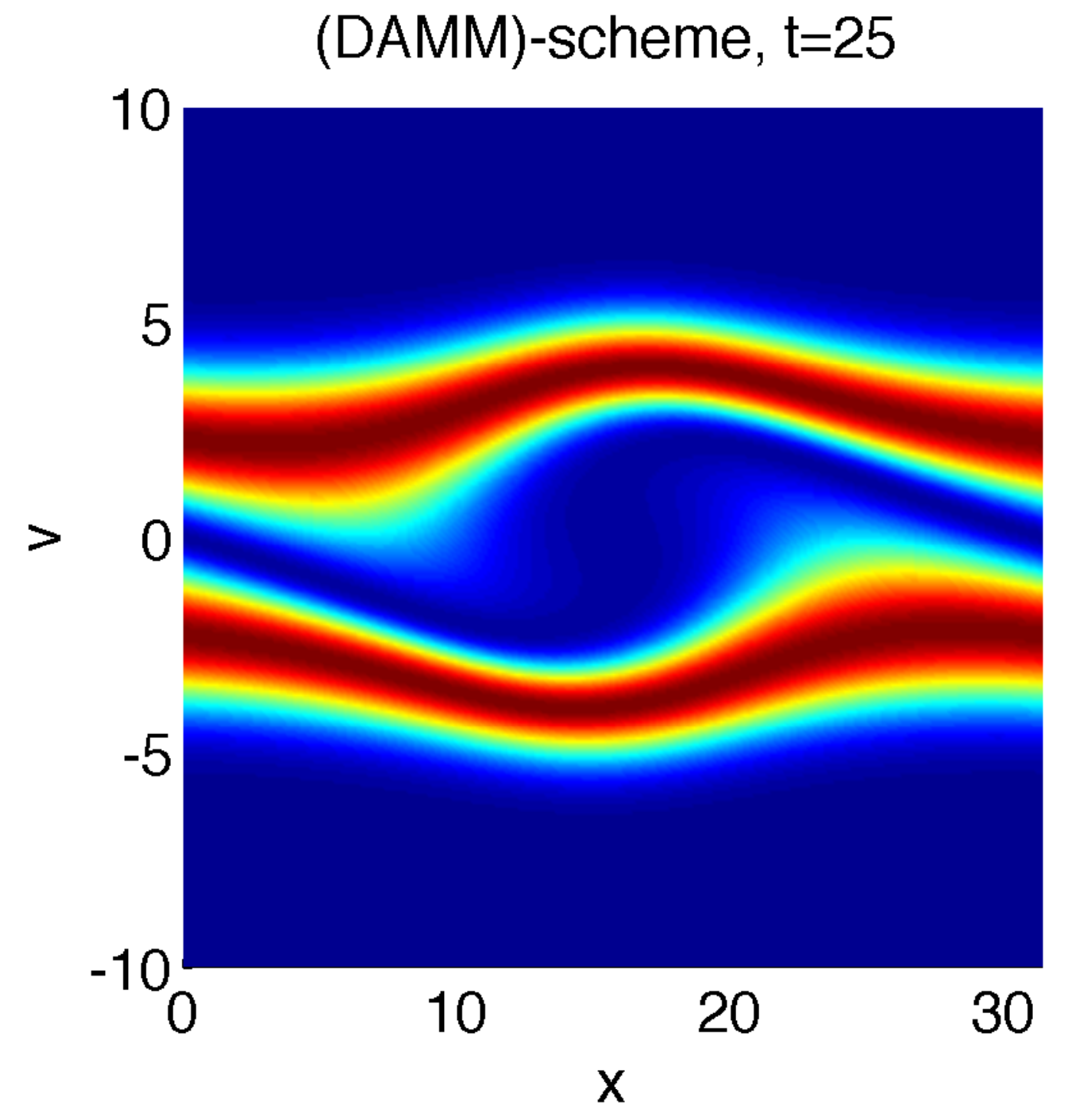}
	\caption{(DAMM)-scheme, $t=25$.}
	\end{subfigure}
    \begin{subfigure}{.49\textwidth}
  	\centering
  	\includegraphics[width=\linewidth,trim={0cm 0cm 0cm 0cm},clip]{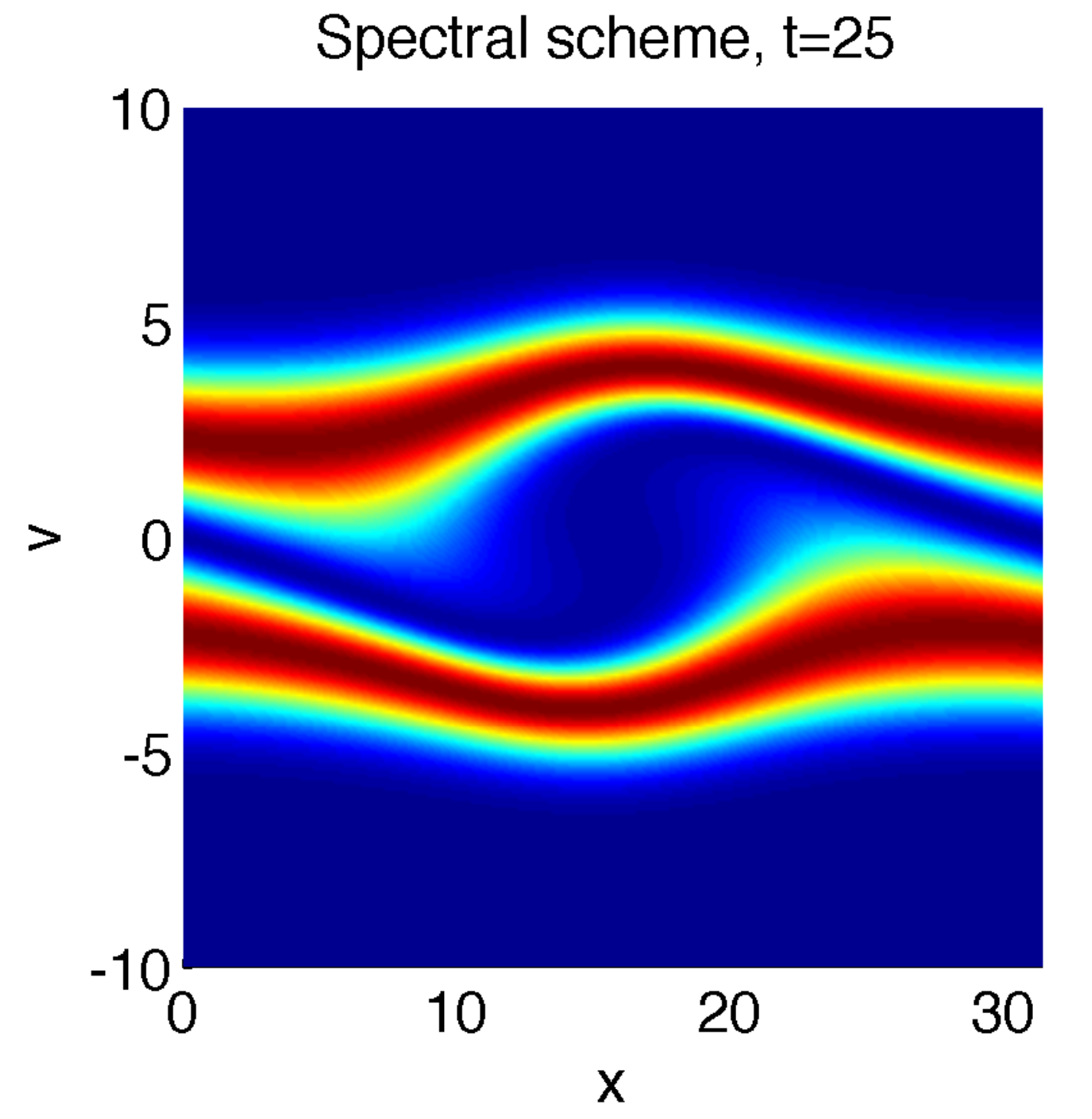}
	\caption{Spectral scheme, $t=25$.}
	\end{subfigure} \\ \vspace{0.5cm}
	\begin{subfigure}{.49\textwidth}
  	\centering
  	\includegraphics[width=\linewidth,trim={0cm 0cm 0cm 0cm},clip]{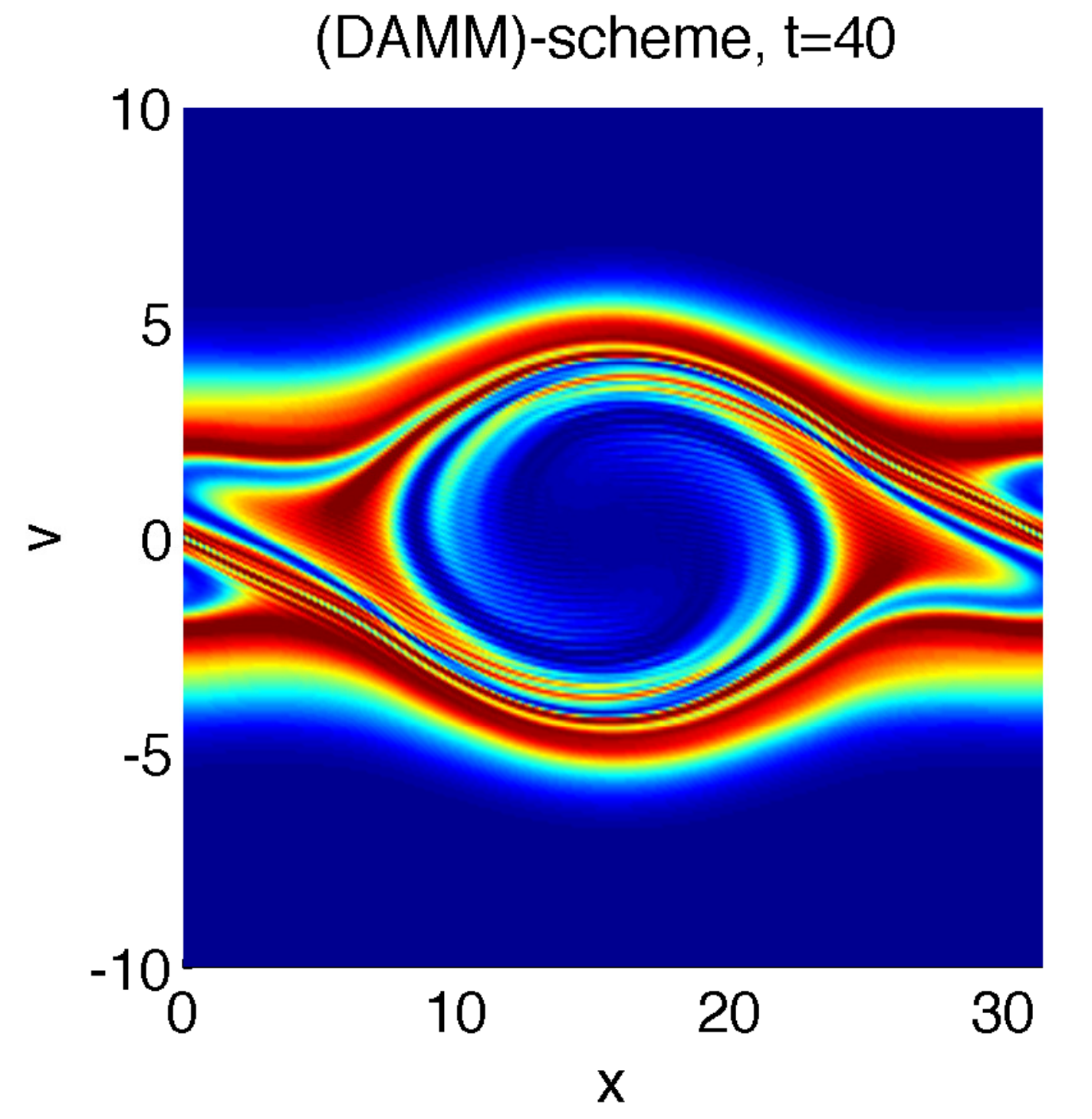}
	\caption{(DAMM)-scheme, $t=40$.}
	\end{subfigure}
     \begin{subfigure}{.49\textwidth}
   	\centering
   	\includegraphics[width=\linewidth,trim={0cm 0cm 0cm 0cm},clip]{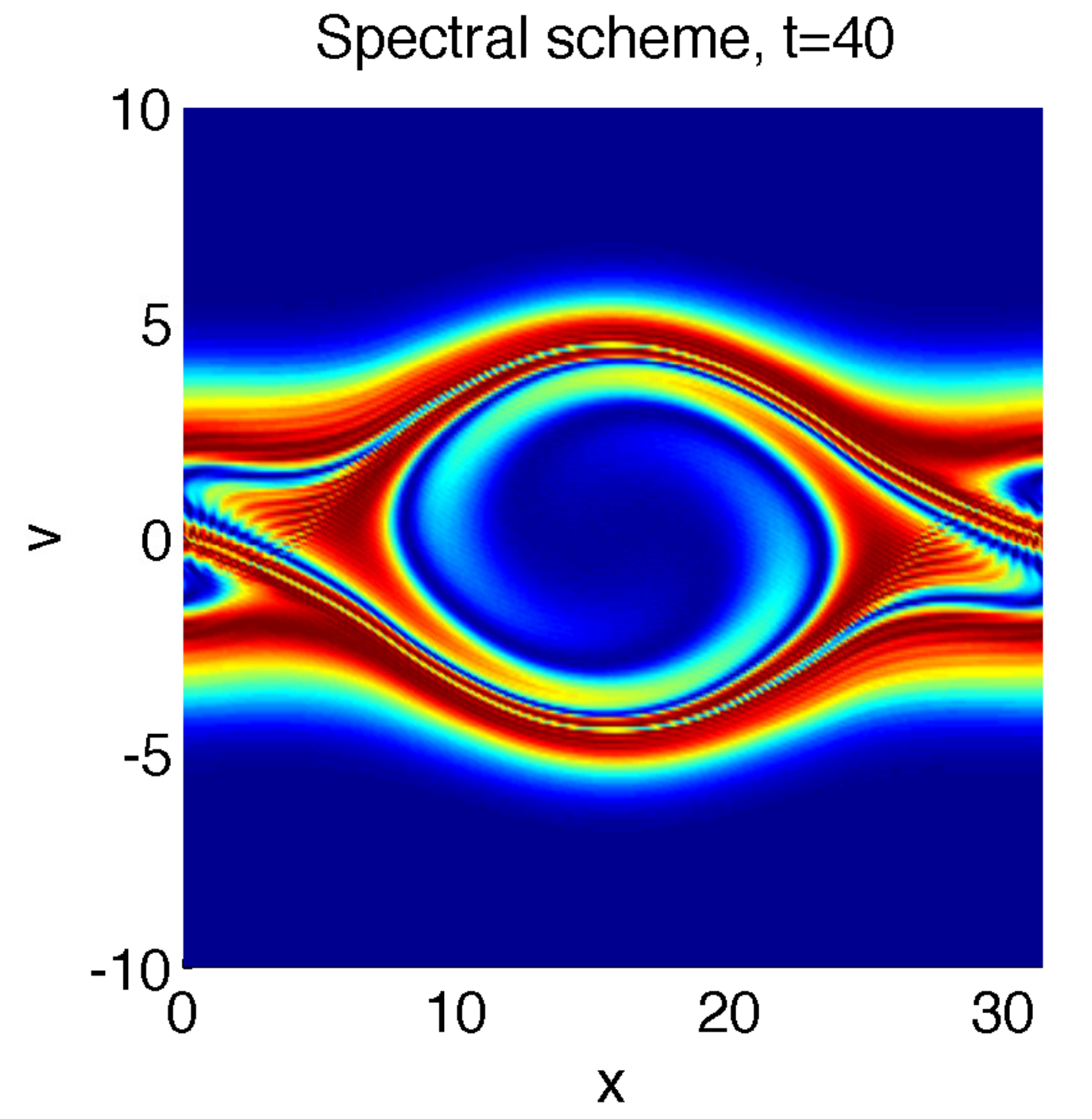}
 	\caption{Spectral scheme, $t=40$.}
 	\end{subfigure} \vspace{0.5cm}
\caption{(Two-stream instability for $\eps=1$ and $f^1_{in}$) Distribution function $f^{\eps}(t,x,v)$ at different times  with $k=0.2$ and $\gamma=0.001$ for the (DAMM)-scheme ((A) and (C)) and the spectral scheme ((B) and (D)). $T=50$, $N_x = 256$, $N_y=256$, $\Delta t = 0.1$, $\eps =1$ and $\sigma = (\Delta x/L_x)^2$.}
\label{two_streams_instability_damm}
\end{figure}

\begin{figure}[ht]
\centering
	\begin{subfigure}{.48\textwidth}
  	\centering
  	\includegraphics[width=\linewidth,trim={0cm 0cm 0cm 0cm},clip]{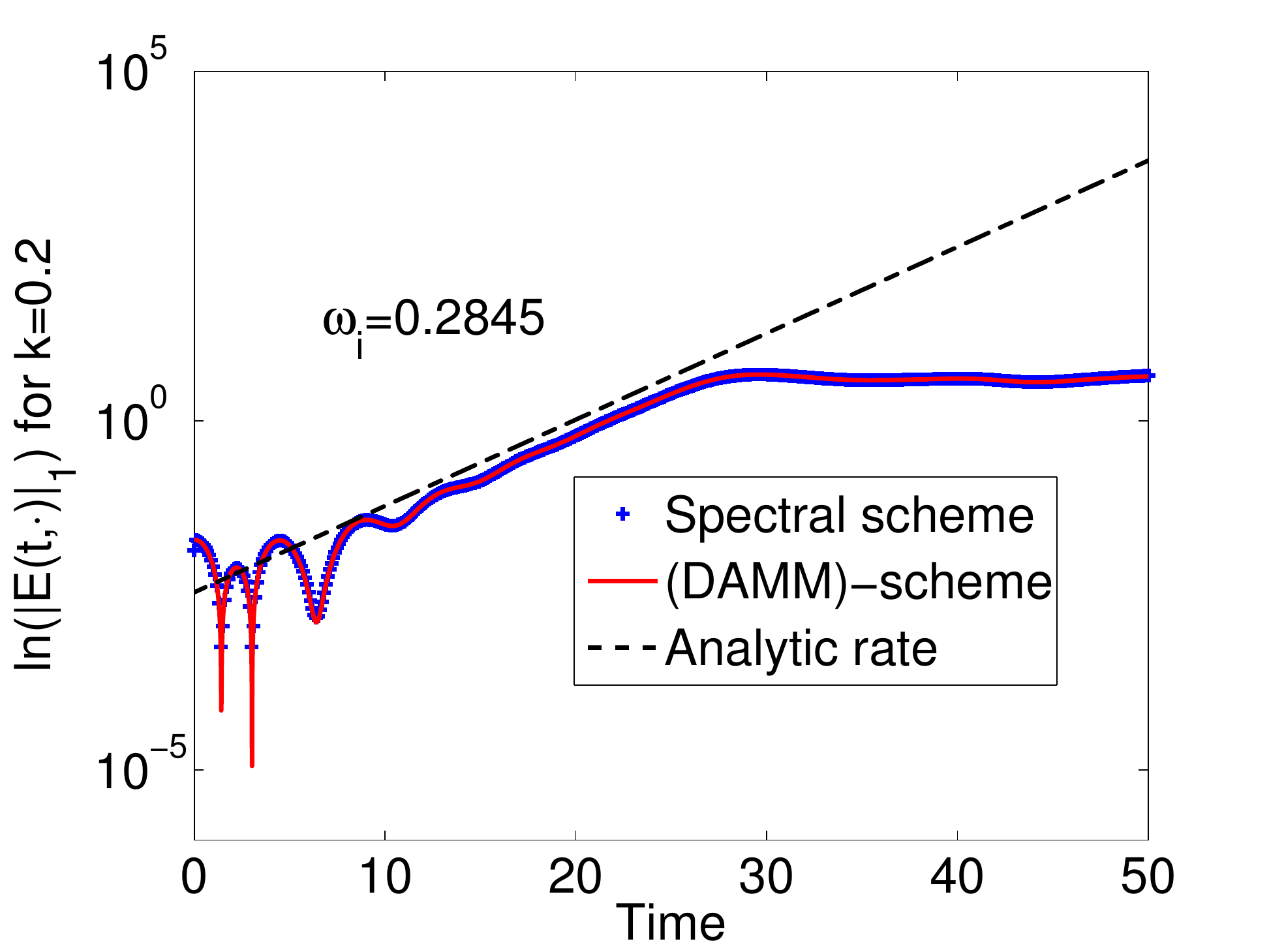}
	\caption{Electric Field (in log value).}
	\end{subfigure}
    \begin{subfigure}{.48\textwidth}
  	\centering
  	\includegraphics[width=\linewidth,trim={0cm 0cm 0cm 0cm},clip]{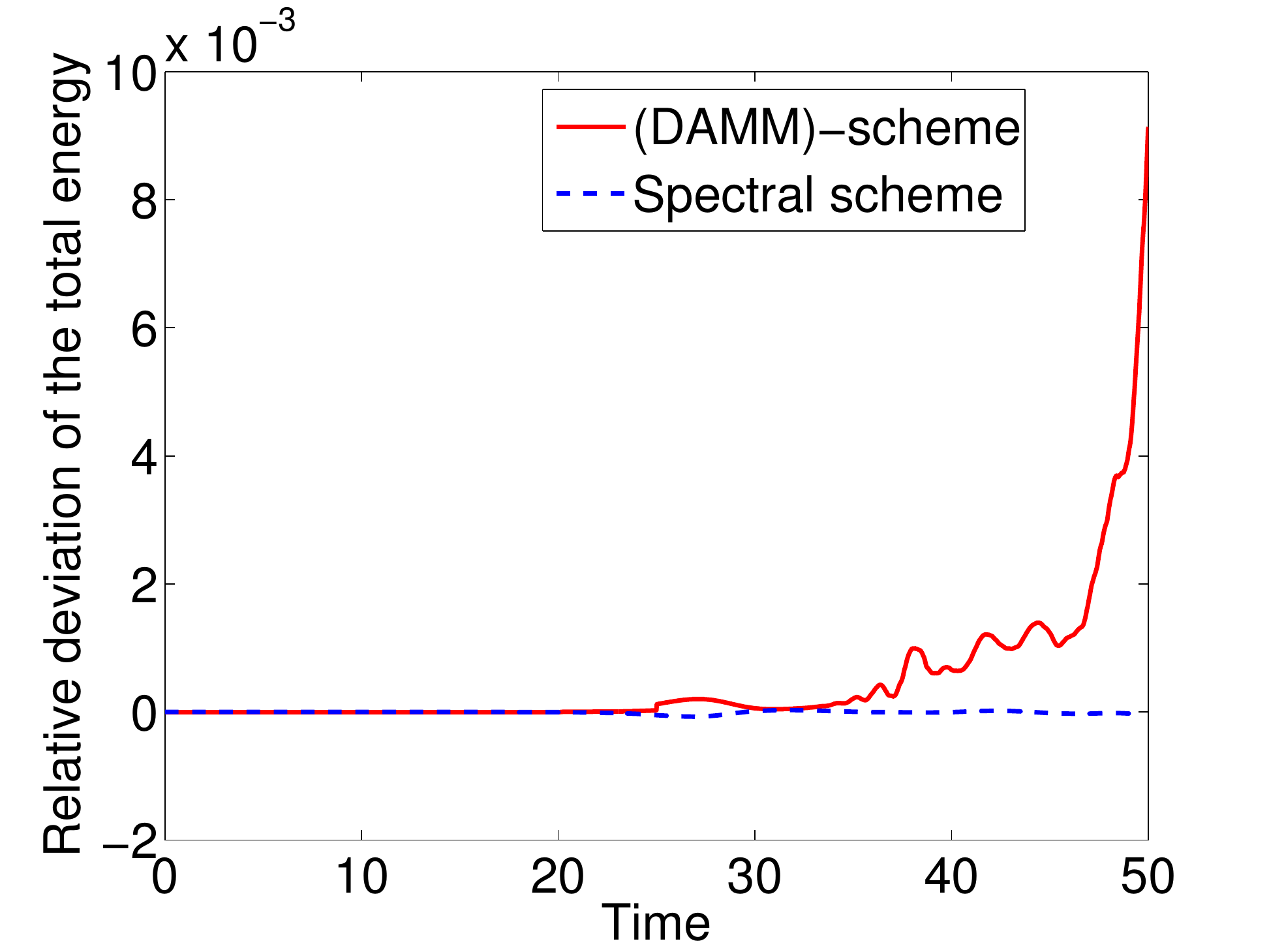}
	\caption{Total energy.}
	\end{subfigure} \\ 
	\begin{subfigure}{.48\textwidth}
  	\centering
  	\includegraphics[width=\linewidth,trim={0cm 0cm 0cm 0cm},clip]{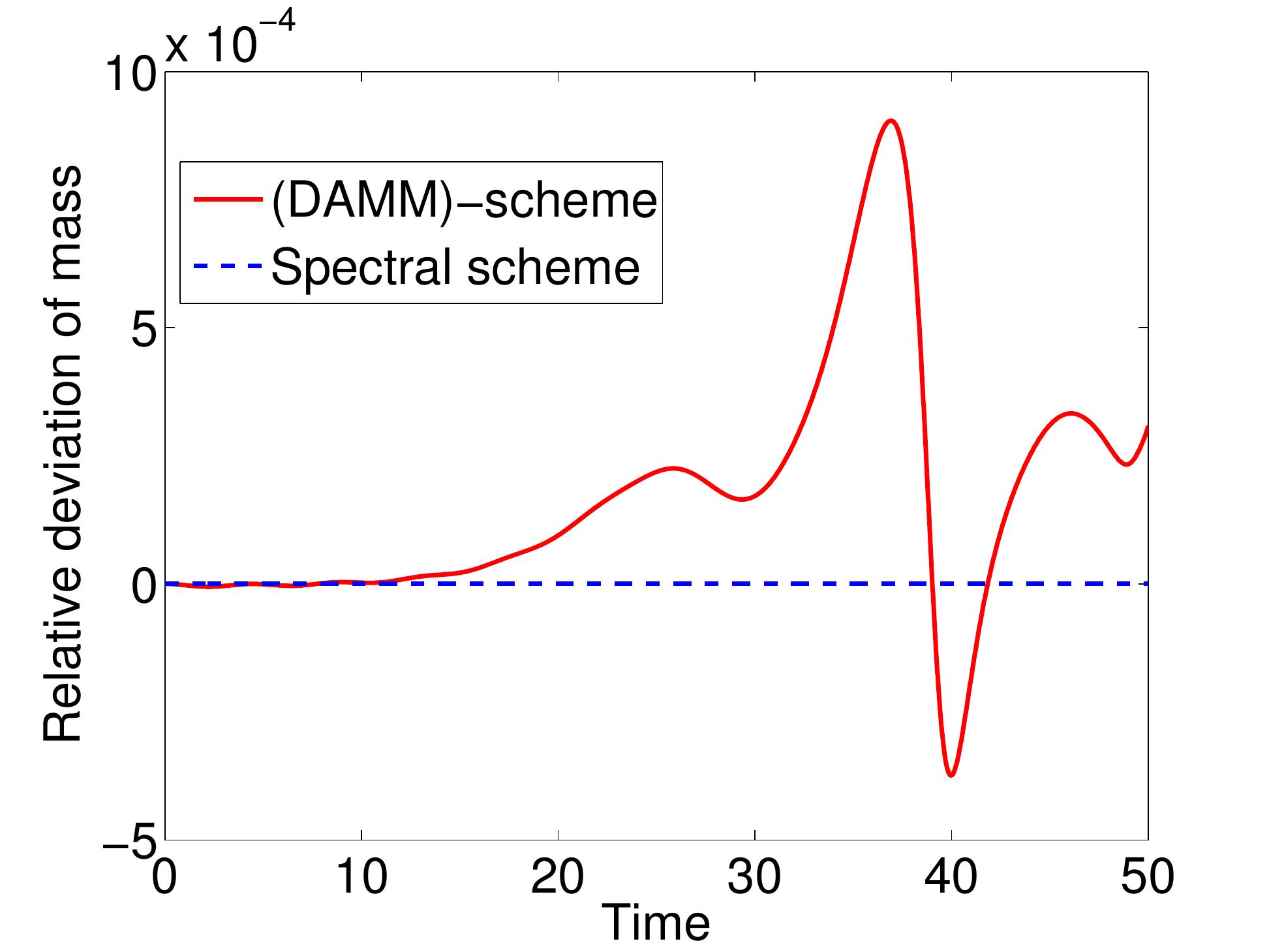}
	\caption{Mass.}
	\end{subfigure}
    \begin{subfigure}{.48\textwidth}
  	\centering
  	\includegraphics[width=\linewidth,trim={0cm 0cm 0cm 0cm},clip]{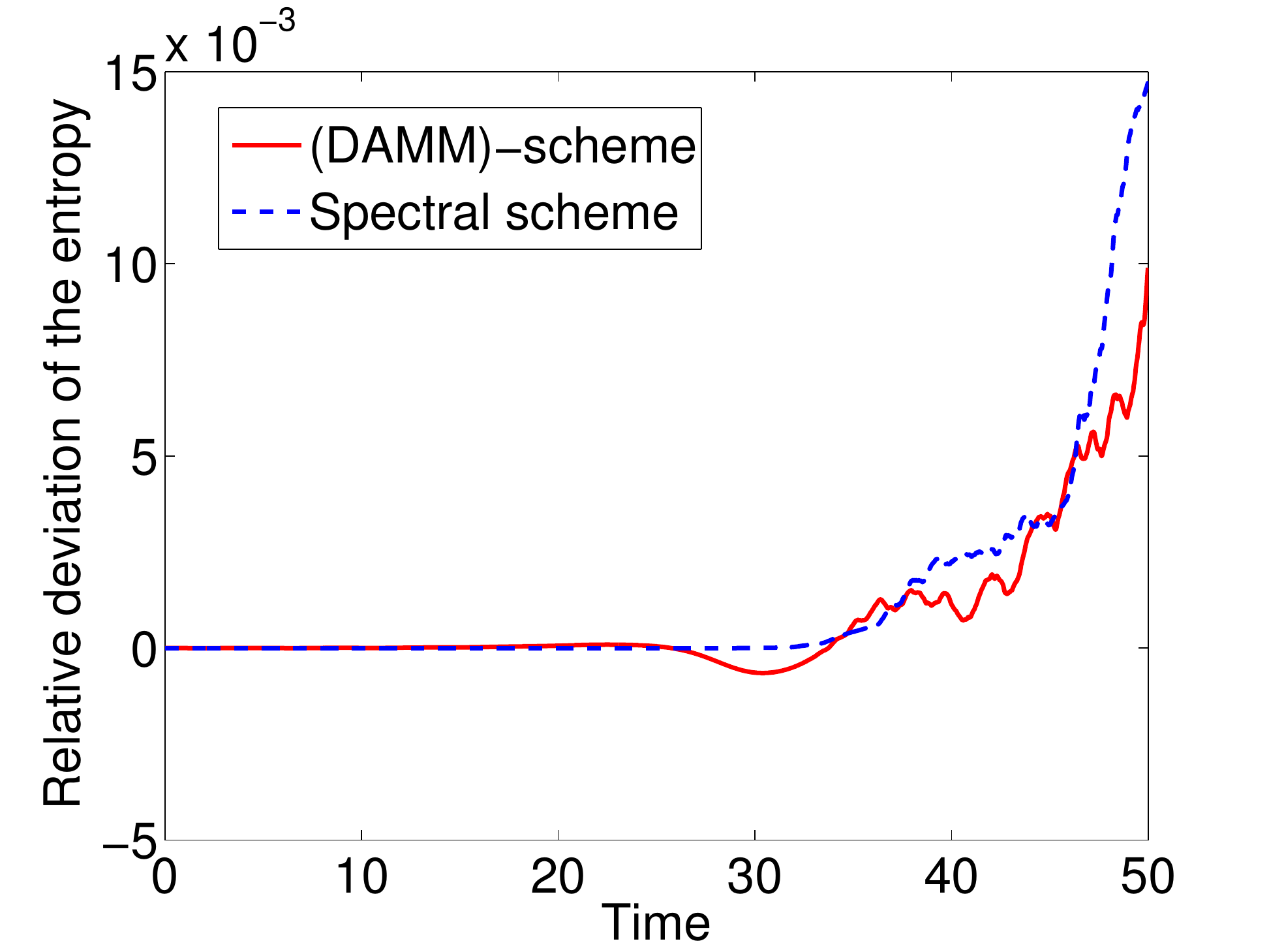}
	\caption{Entropy.}
	\end{subfigure} \\ 
	\begin{subfigure}{.48\textwidth}
  	\centering
  	\includegraphics[width=\linewidth,trim={0cm 0cm 0cm 0cm},clip]{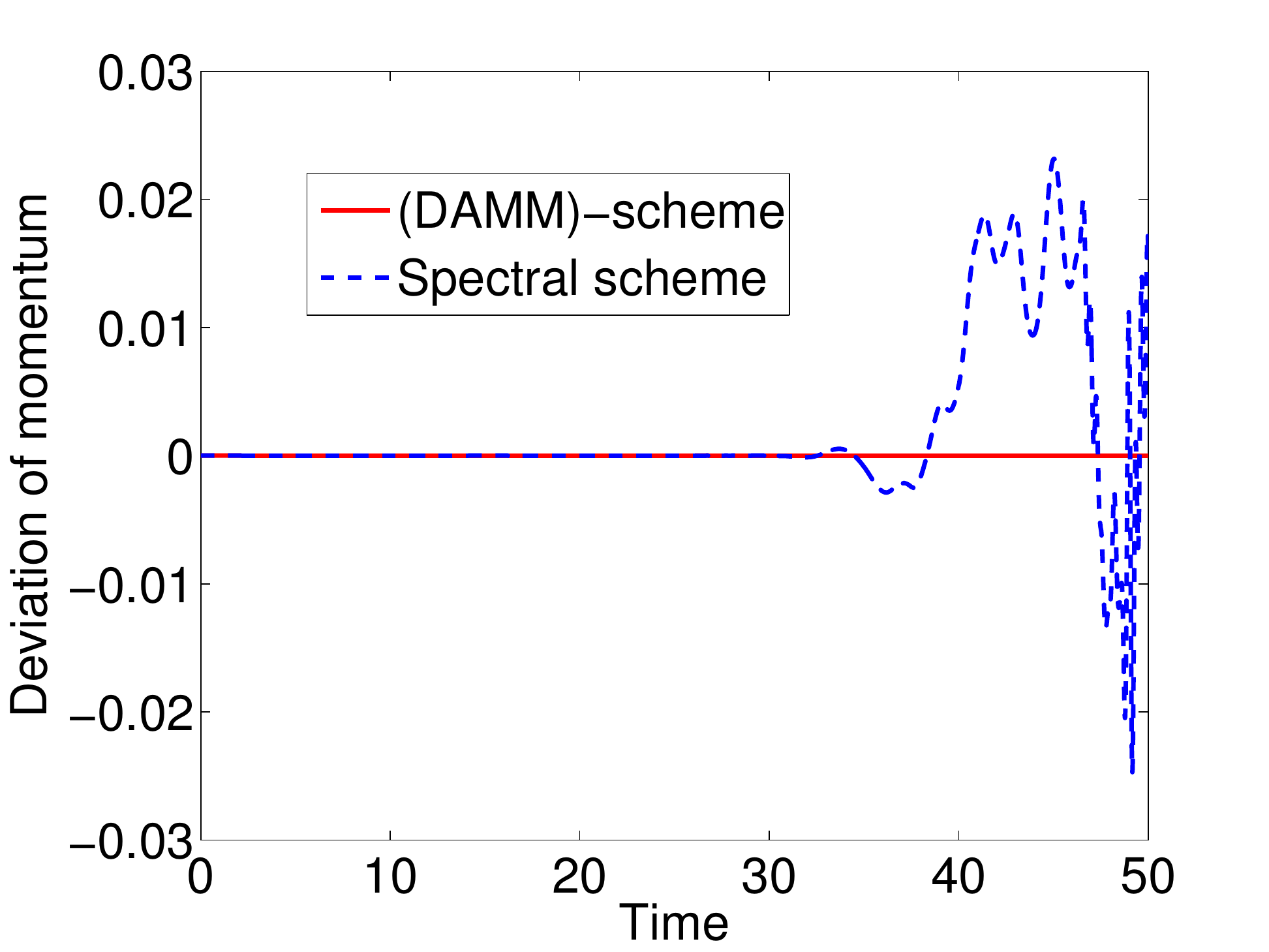}
	\caption{Momentum.}
	\end{subfigure}
    \begin{subfigure}{.48\textwidth}
  	\centering
  	\includegraphics[width=\linewidth,trim={0cm 0cm 0cm 0cm},clip]{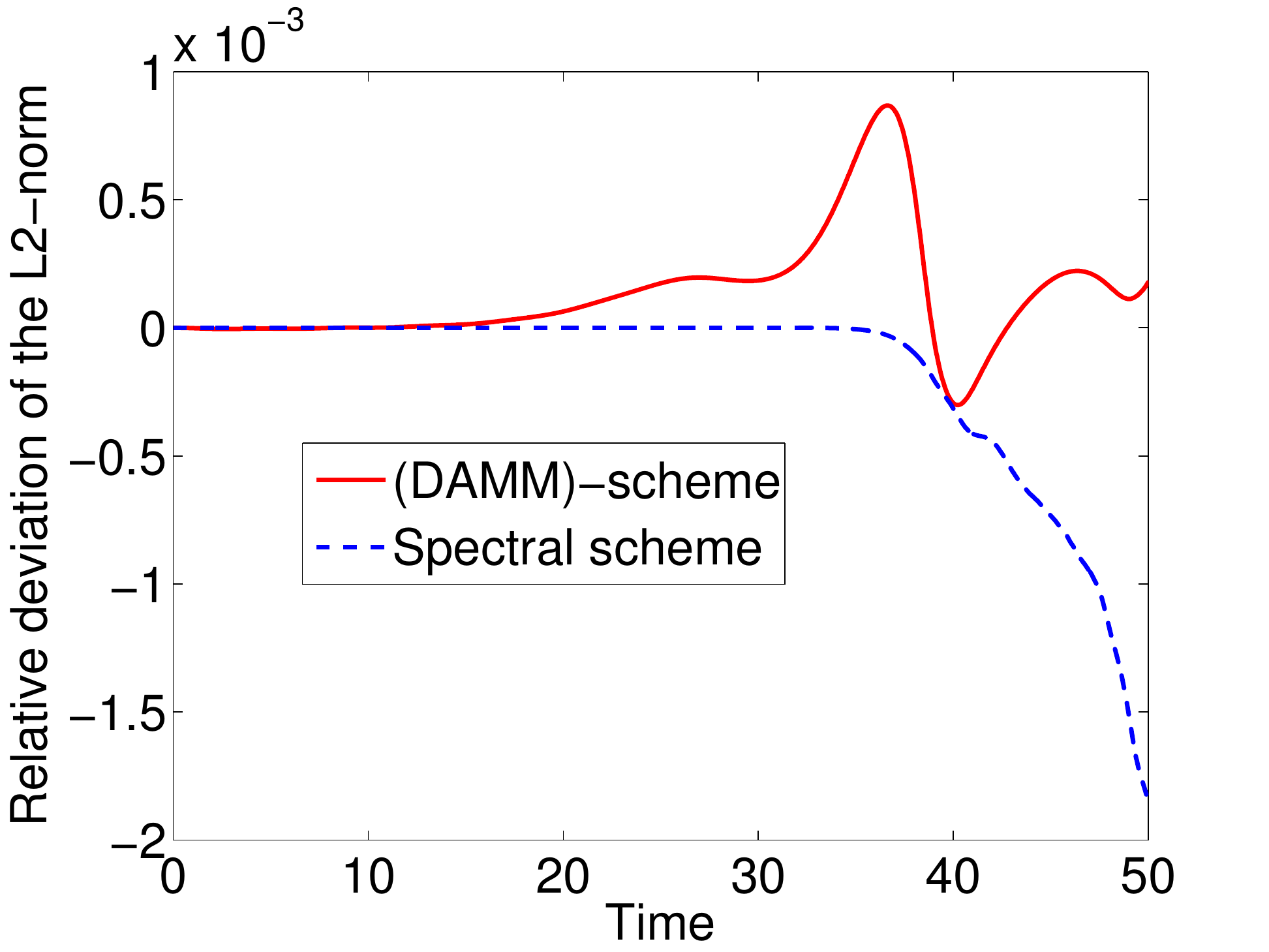}
	\caption{$L^2$-norm.}
	\end{subfigure} \\ 
\caption{(Two-stream instability for $\eps=1$ and $f^1_{in}$) Electric field versus time (A) and deviation over time for both (DAMM) and spectral schemes of several quantities (B), (C), (D), (E) and (F). Mesh size: $N_x=N_y=256$. $T=50$, $\Delta t=0.01$, $\eps=1$, and $\sigma = (\Delta x /L_x)^2$.}
\label{TS_cons}
\end{figure}

%%%%%%%%%%%%%%%%%%%%%%%%%%%%%%%%%%%%%%%%%%%%%%
\subsubsection{ Two-stream instability in the limit regime $\eps \to 0$}
%%%%%%%%%%%%%%%%%%%%%%%%%%%%%%%%%%%%%%%%%%%%%%

In order to conclude this section, we show that the AP-property of the (DAMM)-scheme can be useful when an equilibrium is reached. Recall that in the case of the Vlasov-Poisson system \eqref{Vlasov_Poisson_bracket}, passing to $\eps \to 0$ is equivalent to passing to $t \to \infty$. Thus, the (DAMM)-scheme seems suitable to study the long-time behavior of the non-linear two-stream instability. In this part, we modify the initial condition, taking

\begin{equation} \label{ini_ts_2}
f_{in}^2(x,v) = \frac{1}{\sqrt{2 \, \pi}} \, v^2 e^{-v^2/2}(1+\gamma \, \cos(k\,x))\,.
\end{equation}

\bigskip 

Although no rigorous proofs exist, the two-stream instability leads (in a certain weak sense) to a BGK (Bernstein-Greene-Kruskal) equilibrium after the growth phase. In Figure \ref{bgk_distri}, the qualitative behavior of such equilibrium is visible. We have plotted the initial condition \eqref{ini_ts_2} in the panel (A), then we resolve the Vlasov-Poisson system with the (DAMM)-scheme for $\eps=0$, $\sigma=(\Delta x/L_x)^2$, $\Delta t = 0.01$, $N_x=N_y=256$, $L_x=2\pi / k$, $L_v = 5$, and the initial condition \eqref{ini_ts_2}, with $k=0.5$ and $\gamma = 0.05$. From the first time iteration $n=3$ (panel (B)), the equilibrium seems to be attained and the filamentations are smoothed out. Note the formation of the separatrix which connects the saddle points at $v=0$ and $x=0=4 \,\pi$. Due to the topological conservation of the Vlasov-Poisson equation (see \cite{gamba}), the distribution function keeps over time the nature and the number of its extrema. Panels (C) and (D) represent the distribution function $f^{\eps}$ at the time iteration $n=15$ and $n=50$, respectively. Note that the separatrix, clearly visible in the panel (B), is progressively smoothed out due to the numerical dissipation of the scheme. Besides, the value of the central extremum in $(2 \pi, 0)$ remains essentially constant in time.Thus, the (DAMM)-scheme conserves the nature and the position of this latter, meaning that the particle trapping is well-reproduced by our scheme. In Figure \ref{bgk_distri2}, we have plotted the contours of the distribution function $f^{\eps}$ at the same times. We see clearly in the center the particle trapping on the panels (B), (C) and (D). 

\bigskip

In order to confirm this BGK  saturation, we shall check if the contours of the distribution function $f^0$ are aligned with the contours of the stream-function $\Psi^0 = v^2/2 - \varphi^0$, as one expects that in the limit $\eps \rightarrow 0$ $f^0$ depends only on $\Psi^0$. Thanks to the AP-property of our (DAMM)-scheme, we can obtain this equilibrium with a very low numerical cost, without too much numerical pollution. Few iterations are effectively needed to reach this equilibrium. To put into evidence the dependence $f^0(\Psi^0)$, we use the fitting proposed by Heath and al \cite{gamba}, namely
$$
f^0_{fit} = a \,(\Psi^0+\varphi_M) \, (\Psi^0 + \Psi^{\star})\, e^{-\beta \, \Psi^0}\,,
$$
where $a$ and $\beta$ are fitting parameters to be found numerically, $\varphi_{M}$ is the maximum of $\varphi^0$ and $\Psi^{\star}$ is defined by
$$
\Psi^{\star} = \frac{\varphi_M- \beta \, \Psi_M \varphi_M + 2\Psi_M-\beta \Psi_M^2}{\beta \, \varphi_M + \beta \, \Psi_M -1}\,,
$$
where $\Psi_M$ is the value at which $f^0$ attains its maximum. We choose $a=0.2948$ and $\beta=1.20$. From the numerical simulation, we extract $\varphi_M = 0.60$, $\Psi_M = 0.93$ and thus $\Psi^{\star}=0.90$. In Figure \ref{BGK_figures_curve}, we plot the evolution of $f^0(\Psi^0)$ as compared to the fitting distribution $f_{fit}^0$. Panel (A) represents $f^0(\Psi^0)$ at time $t=0$, clearly, there is no alignment between $f^0$ and $\Psi^0$, as expected. The panels (C) and (E) which zoom the panel (A) in two regimes confirm this affirmation, we see clearly the non-functional structure of the plot (multi-valued function). However, in panel (B), we track the same evolution but after fifty time iterations. One notes a very good correspondence between the numerical curve and the fitting one. The panels (D) and (F) show a good alignment of the points, showing that a BGK equilibrium is attained. Nevertheless, we observe an anormal inflexion of the curve near to the point $f^0(\Psi^0=0)$. A similar phenomenon was observed in \cite{CG}. In the panel (F), we examine $f^0(\Psi^0)$ near to its minimum. We pay attention here that there is no splitting phenomenon, confirming that the saturation is totally achieved.
 
\bigskip

%To carry on the study, we have plotted in Figure \ref{bgk_distri3} the deviation (from its initial value) of the mass (panel (A)), the $L^2$-norm (panel (B)), the momentum (panel (C)) and the total energy (panel (D)) from the previous simulation. The momentum is clearly conserved and stays null, as expected. The other quantities deviate weakly from their initial value, however after five iterations, the deviation stops, meaning that when the equilibrium is attained, there is no more degradation of the solution. This underlines the strong advantage of the (DAMM)-scheme with respect to the classical numerical procedures, where at each time iteration, the errors of the solution accumulate. 

To summarize, the (DAMM)-scheme permits, by passing to the limit $\eps \to 0$, to obtain a BGK equilibrium with a low number of iterations, permitting to control the accumulation of the errors. This is an essential advantage, as compared to standard schemes.

\begin{figure}[h]
\centering
	\begin{subfigure}{.49\textwidth}
  	\centering
  	\includegraphics[width=\linewidth,trim={0cm 0cm 0cm 0cm},clip]{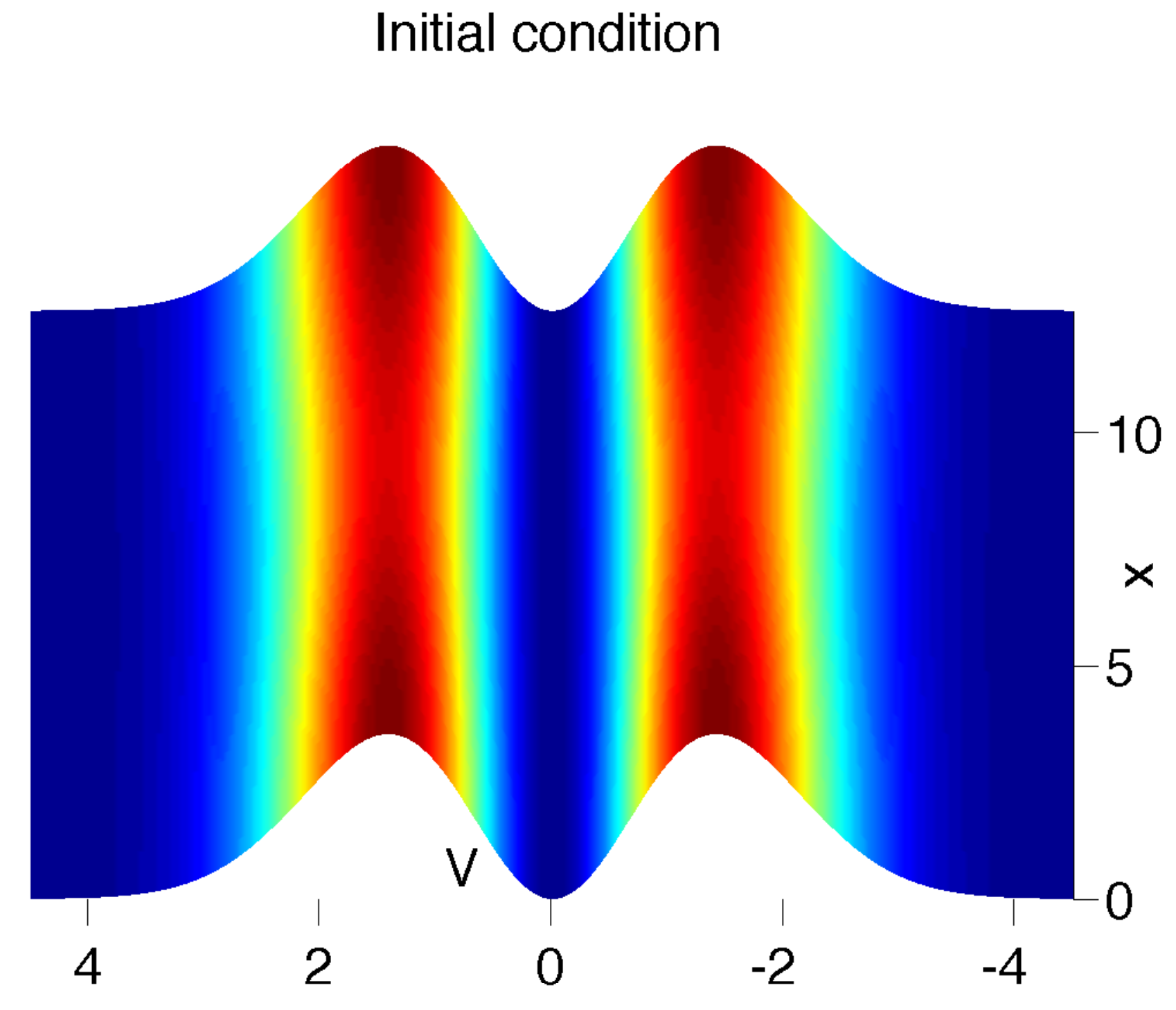}
	\caption{Initial condition}
	\end{subfigure}
    \begin{subfigure}{.49\textwidth}
  	\centering
  	\includegraphics[width=\linewidth,trim={0cm 0cm 0cm 0cm},clip]{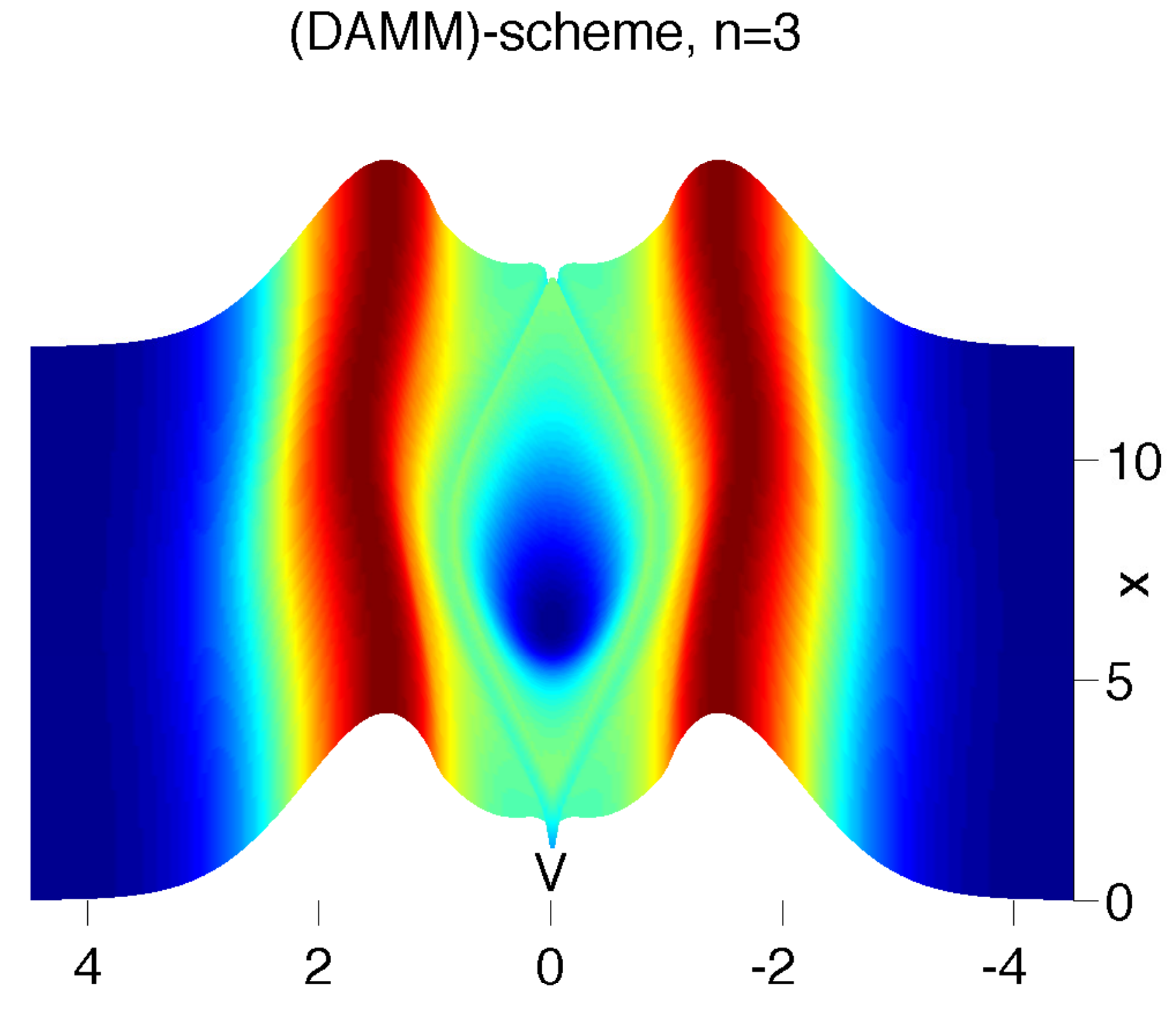}
	\caption{n=3}
	\end{subfigure} \\ \vspace{0.5cm}
	\begin{subfigure}{.49\textwidth}
  	\centering
  	\includegraphics[width=\linewidth,trim={0cm 0cm 0cm 0cm},clip]{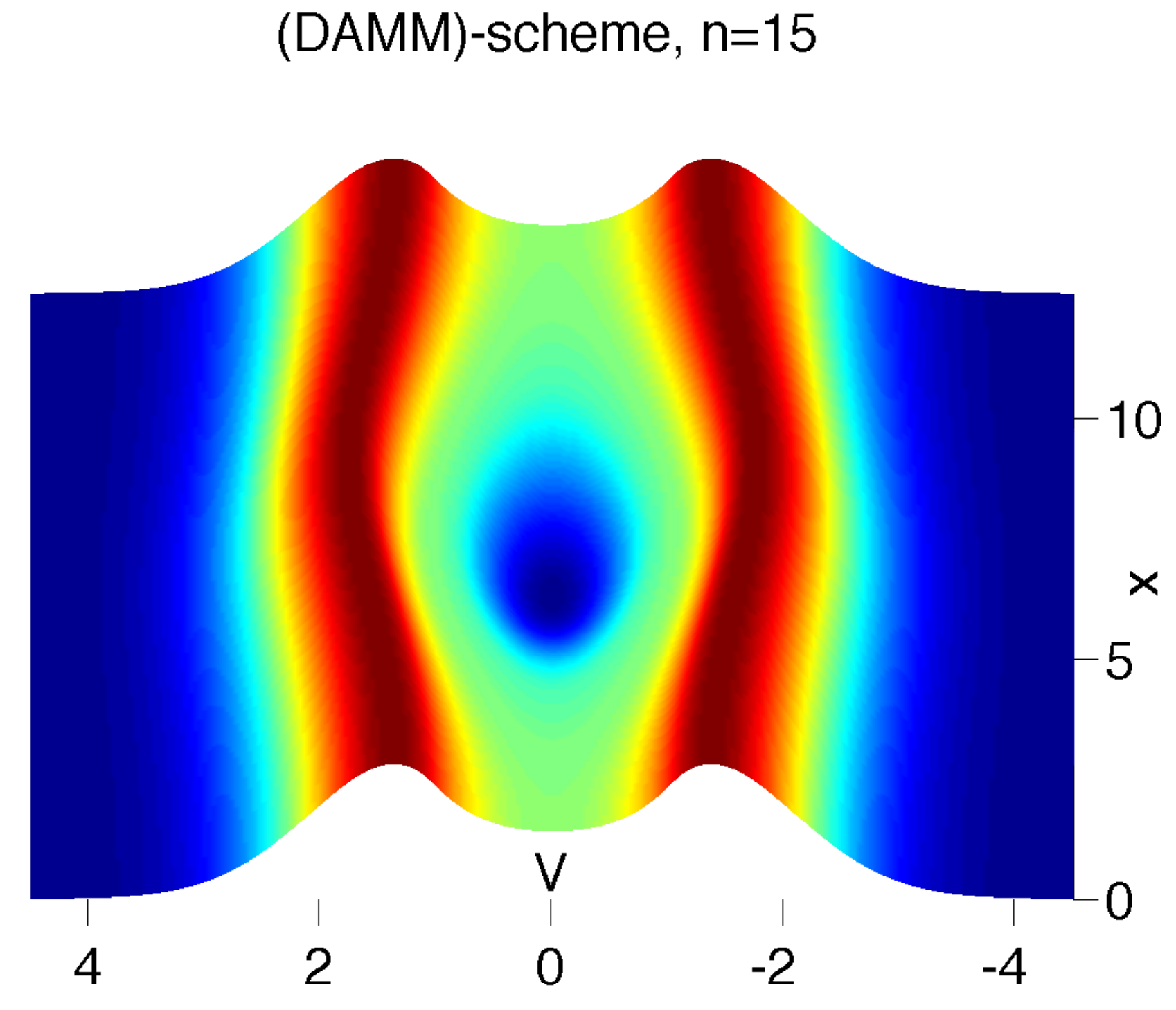}
	\caption{n=15}
	\end{subfigure}
     \begin{subfigure}{.49\textwidth}
   	\centering
   	\includegraphics[width=\linewidth,trim={0cm 0cm 0cm 0cm},clip]{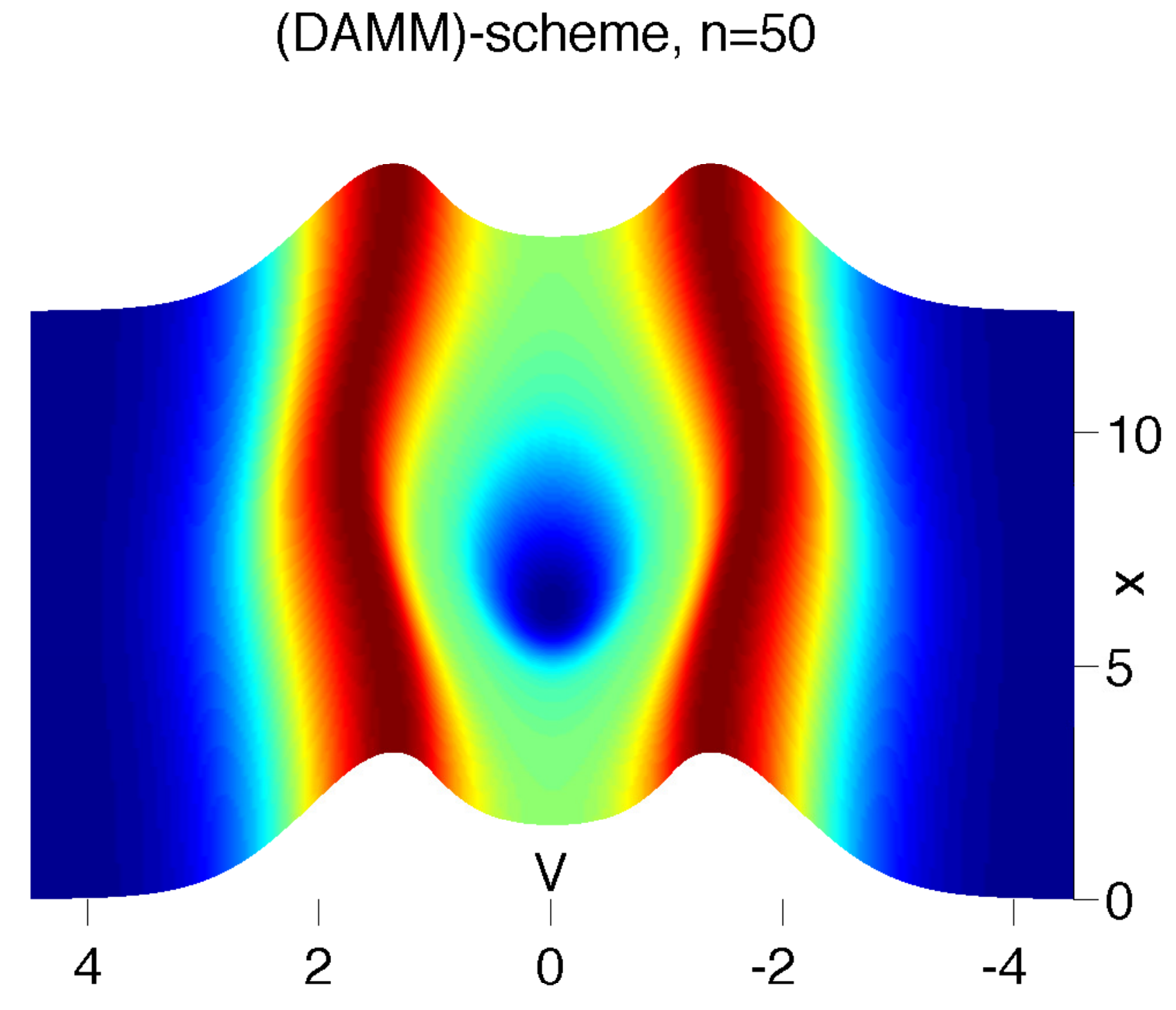}
 	\caption{n=50}
 	\end{subfigure} \vspace{0.5cm}
\caption{(Two-stream instability for $\eps=0$ and $f^2_{in}$) Distribution function $f^0(t,x,v)$ at different times for the two stream instability with $k=0.5$ and $\gamma=0.05$ via the (DAMM)-scheme. Parameters were $N_x = N_y=256$, $\Delta t = 0.1$, and $\sigma = (\Delta x/L_x)^2$.}
\label{bgk_distri}
\end{figure}

\begin{figure}[h]
\centering
	\begin{subfigure}{.49\textwidth}
  	\centering
  	\includegraphics[width=\linewidth,trim={0cm 0cm 0cm 0cm},clip]{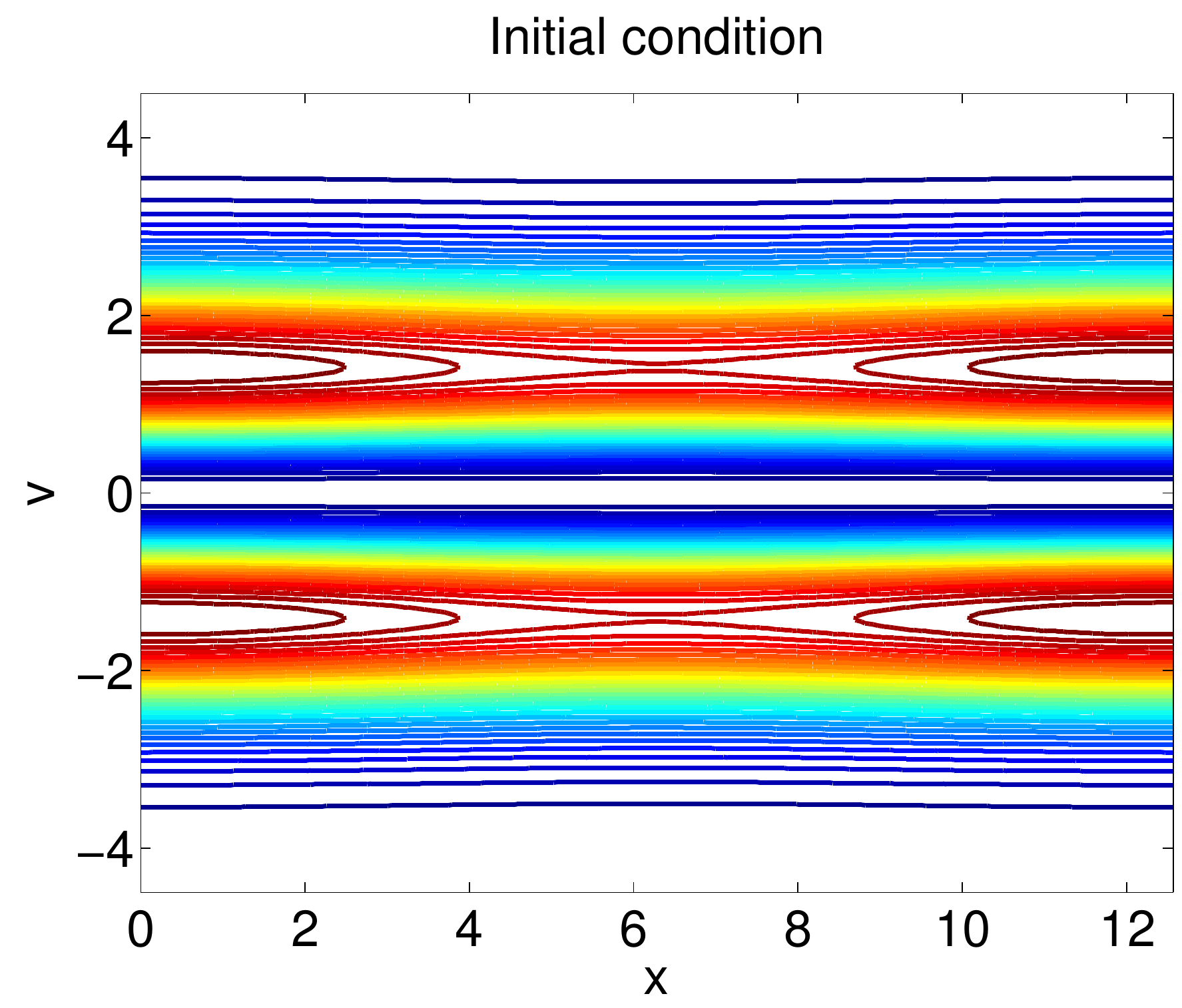}
	\caption{Initial condition}
	\end{subfigure}
    \begin{subfigure}{.49\textwidth}
  	\centering
  	\includegraphics[width=\linewidth,trim={0cm 0cm 0cm 0cm},clip]{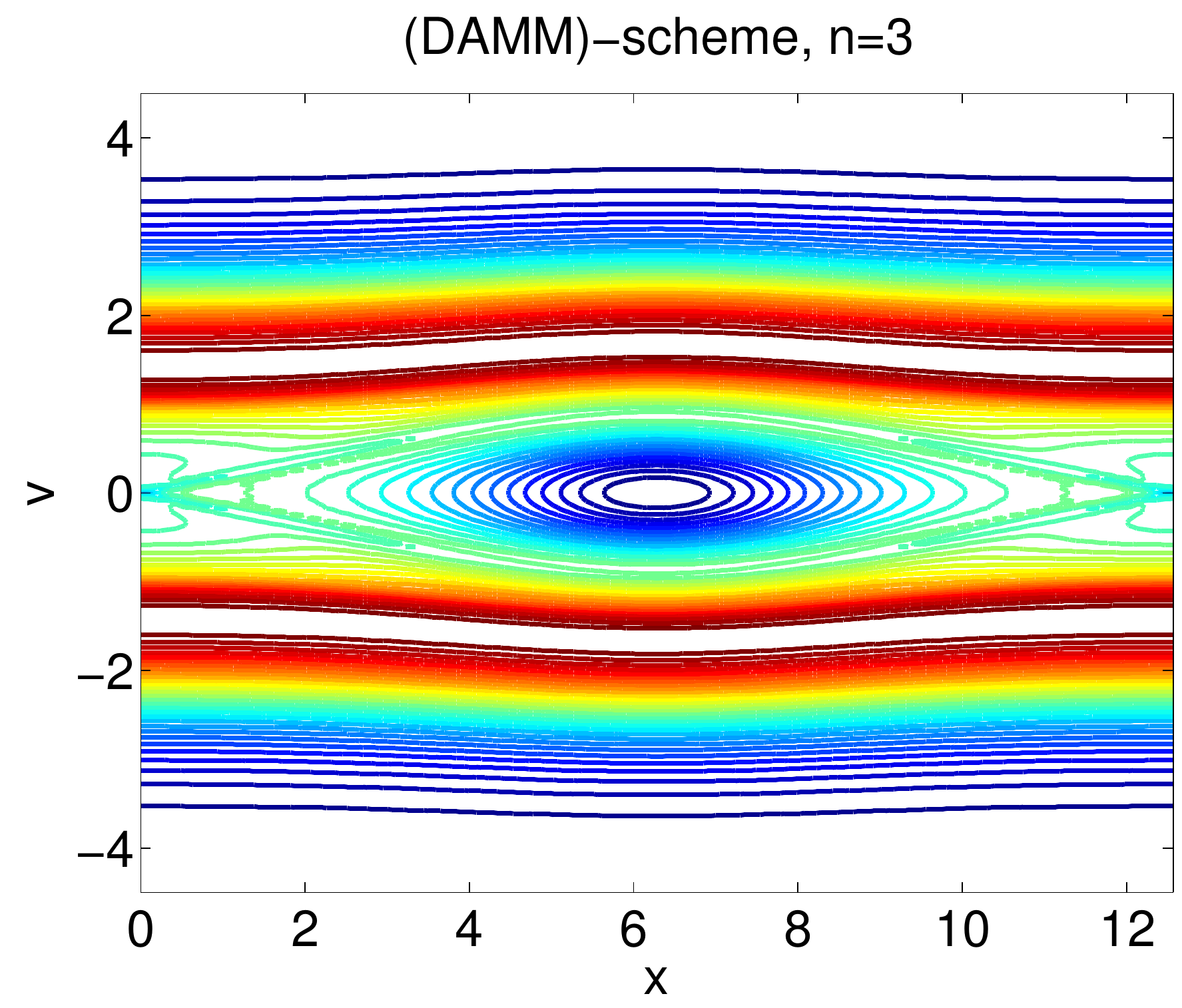}
	\caption{n=3}
	\end{subfigure} \\ \vspace{0.5cm}
	\begin{subfigure}{.49\textwidth}
  	\centering
  	\includegraphics[width=\linewidth,trim={0cm 0cm 0cm 0cm},clip]{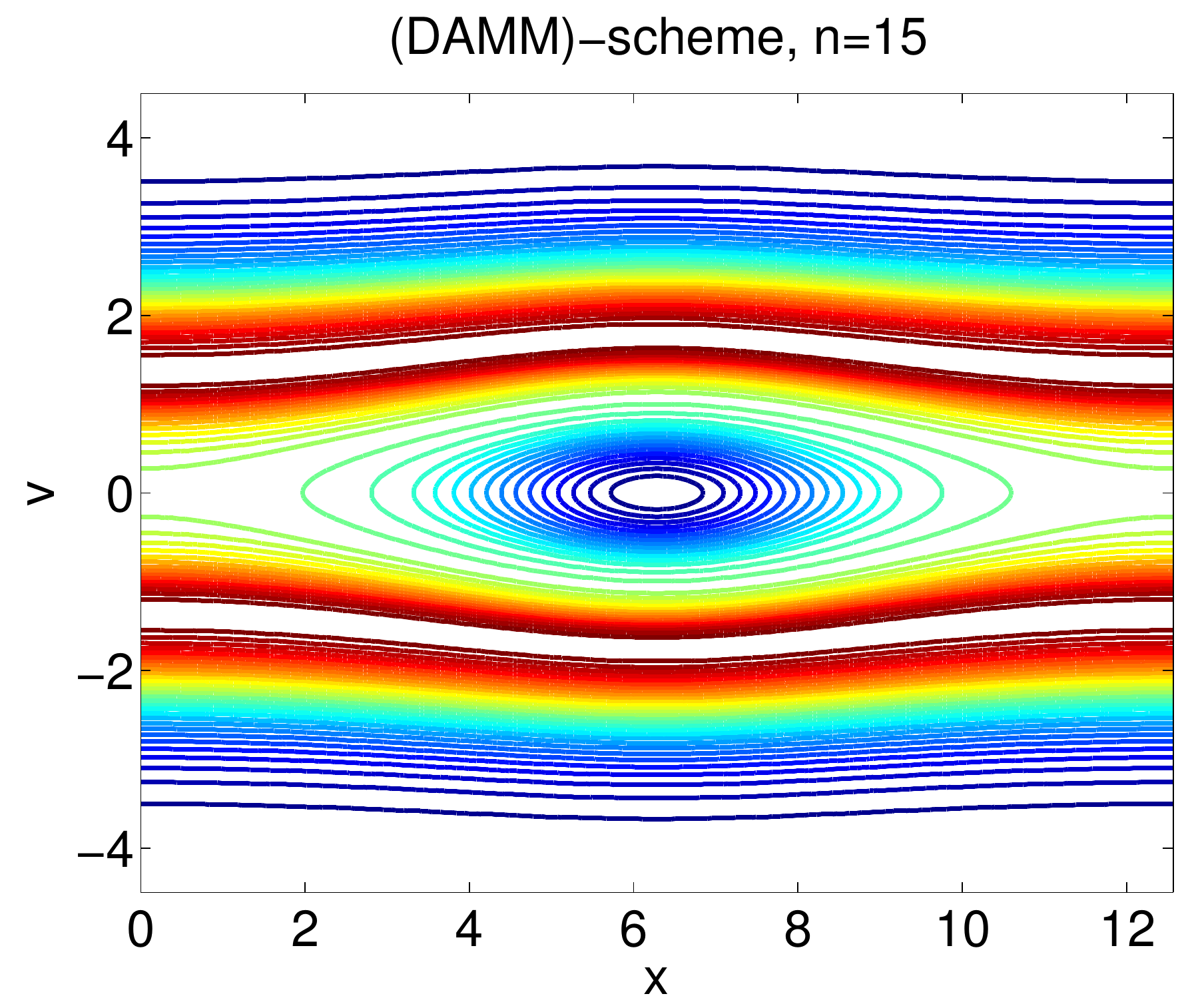}
	\caption{n=15}
	\end{subfigure}
     \begin{subfigure}{.49\textwidth}
   	\centering
   	\includegraphics[width=\linewidth,trim={0cm 0cm 0cm 0cm},clip]{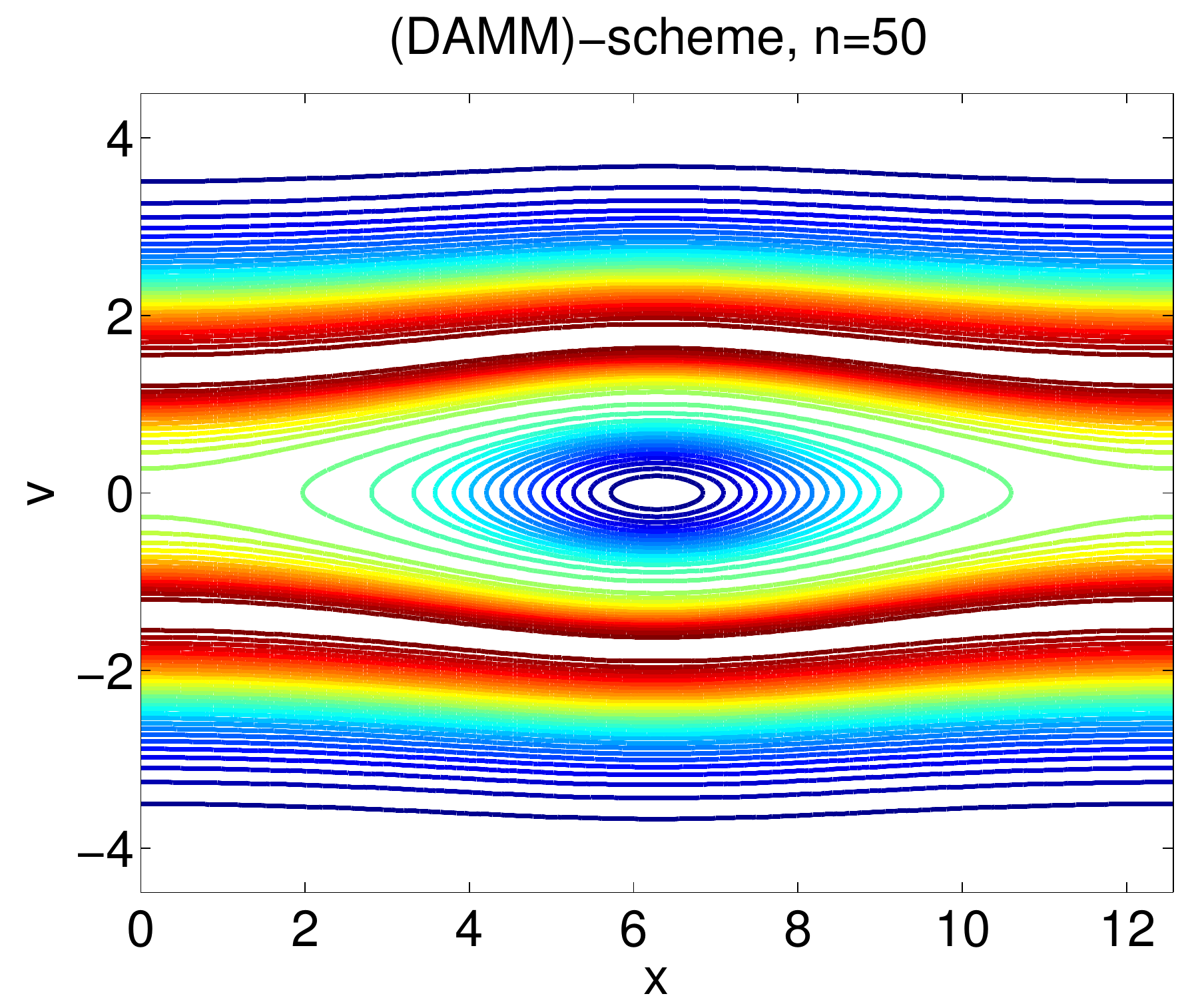}
 	\caption{n=50}
 	\end{subfigure} \vspace{0.5cm}
\caption{(Two-stream instability for $\eps=0$ and $f^2_{in}$) Contour plots of the distribution function $f^0(t,x,v)$ at different times for the two stream instability with $k=0.5$ and $\gamma=0.05$ via the (DAMM)-scheme. Parameters were $N_x = N_y=256$, $\Delta t = 0.1$, and $\sigma = (\Delta x/L_x)^2$.}
\label{bgk_distri2}
\end{figure}

\begin{figure}[ht]
\centering
	\begin{subfigure}{.49\textwidth}
  	\centering
  	\includegraphics[width=\linewidth,trim={0cm 0cm 0cm 0cm},clip]{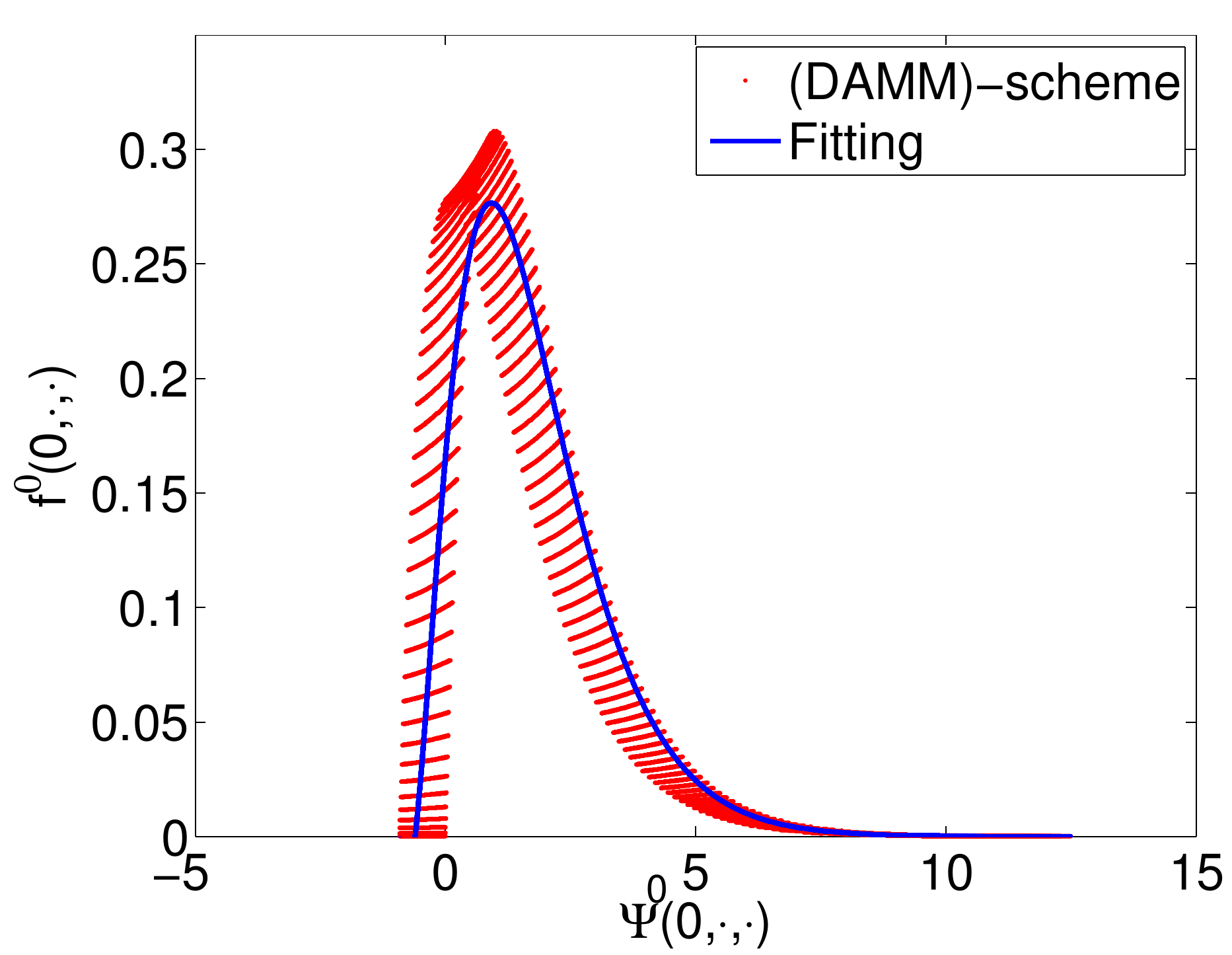}
	\caption{$n=0$.}
	\end{subfigure}
    \begin{subfigure}{.49\textwidth}
  	\centering
  	\includegraphics[width=\linewidth,trim={0cm 0cm 0cm 0cm},clip]{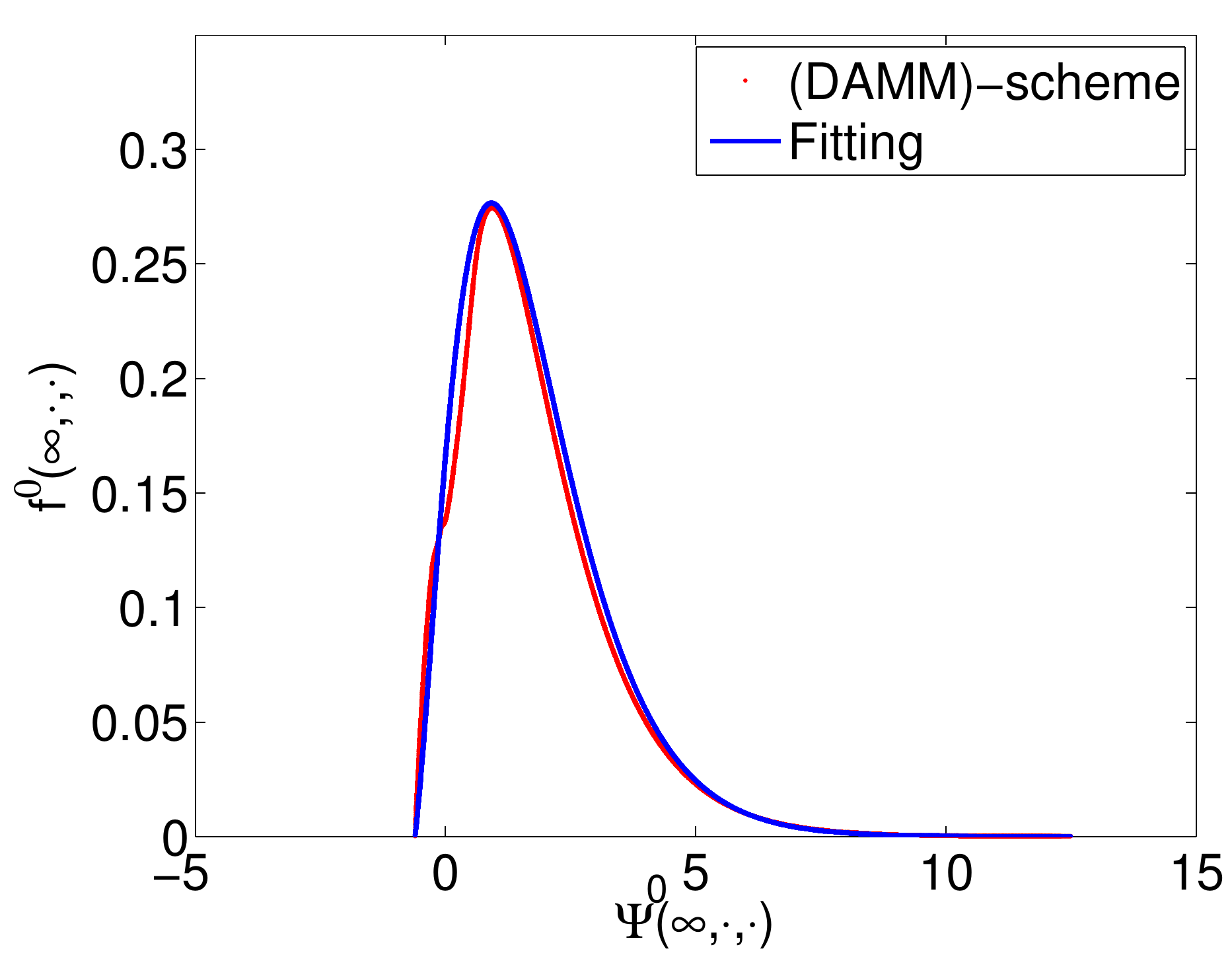}
	\caption{$n=50$ and $\eps=0$.}
	\end{subfigure} \\ 
	\vfill
	\vfill
	\begin{subfigure}{.49\textwidth}
  	\centering
  	\includegraphics[width=\linewidth,trim={0cm 0cm 0cm 0cm},clip]{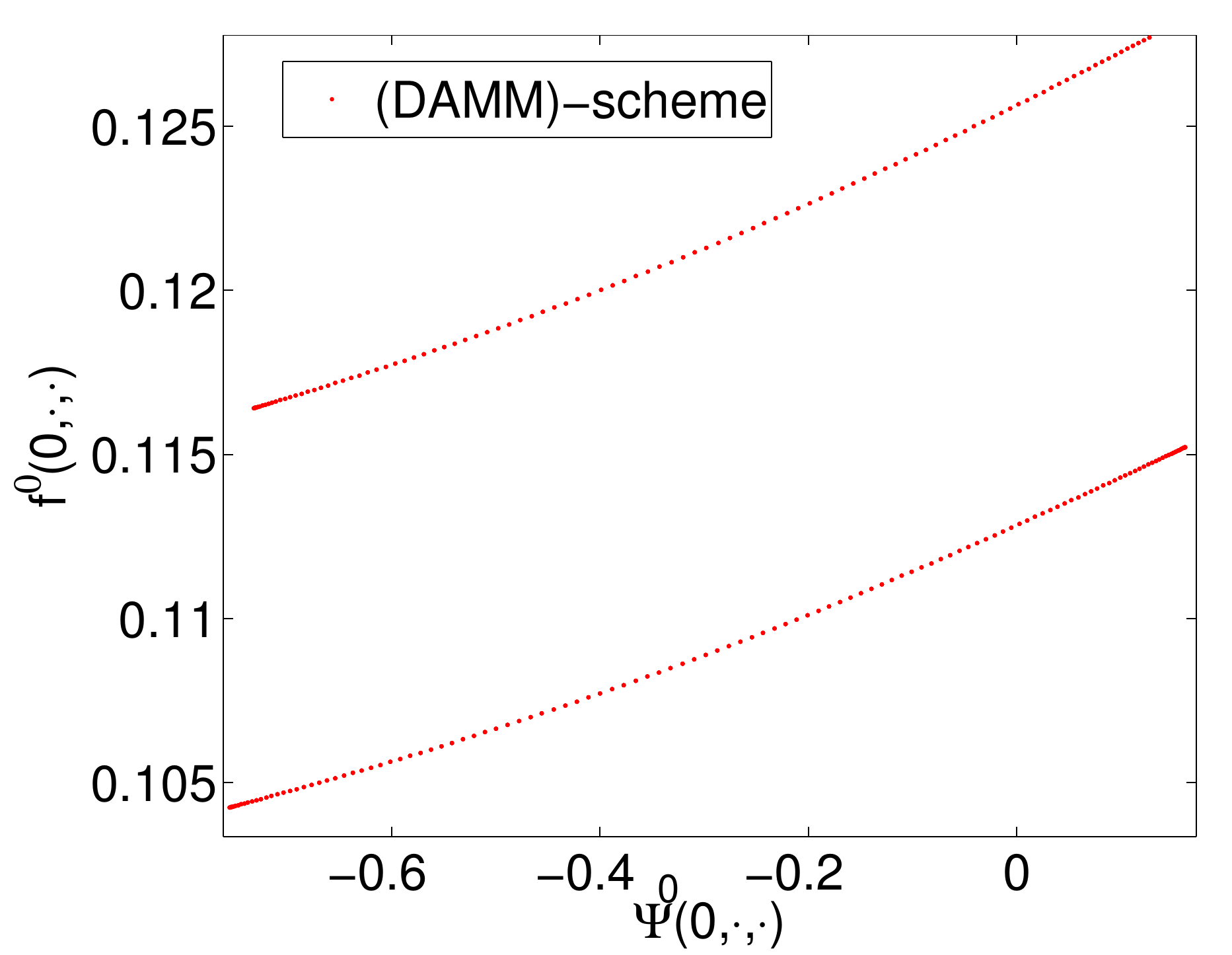}
	\caption{$n=0$.}
	\end{subfigure}
    \begin{subfigure}{.49\textwidth}
  	\centering
  	\includegraphics[width=\linewidth,trim={0cm 0cm 0cm 0cm},clip]{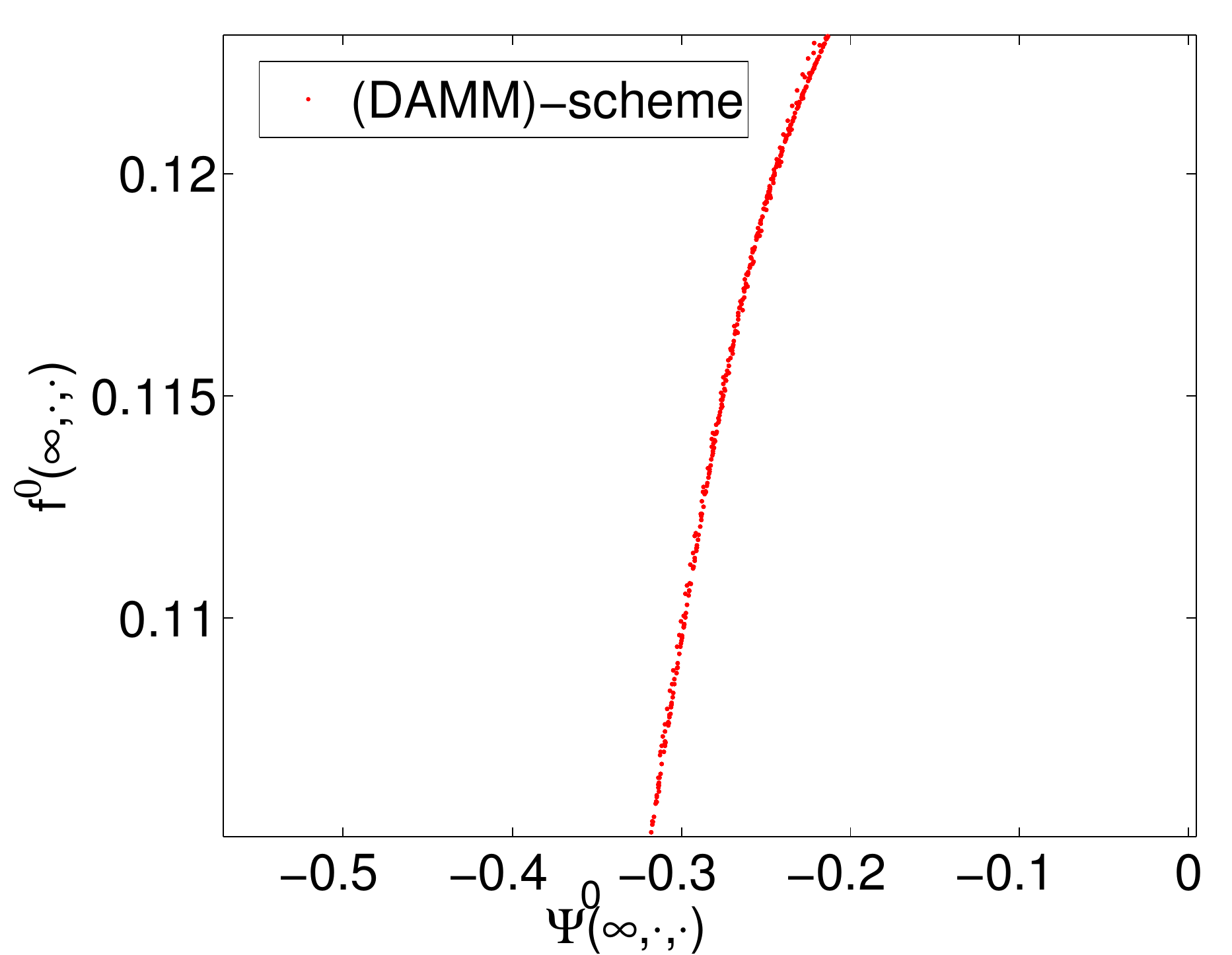}
	\caption{$n=50$ and $\eps=0$.}
	\end{subfigure} \\ 
	\vfill
	\vfill
	\begin{subfigure}{.49\textwidth}
  	\centering
  	\includegraphics[width=\linewidth,trim={0cm 0cm 0cm 0cm},clip]{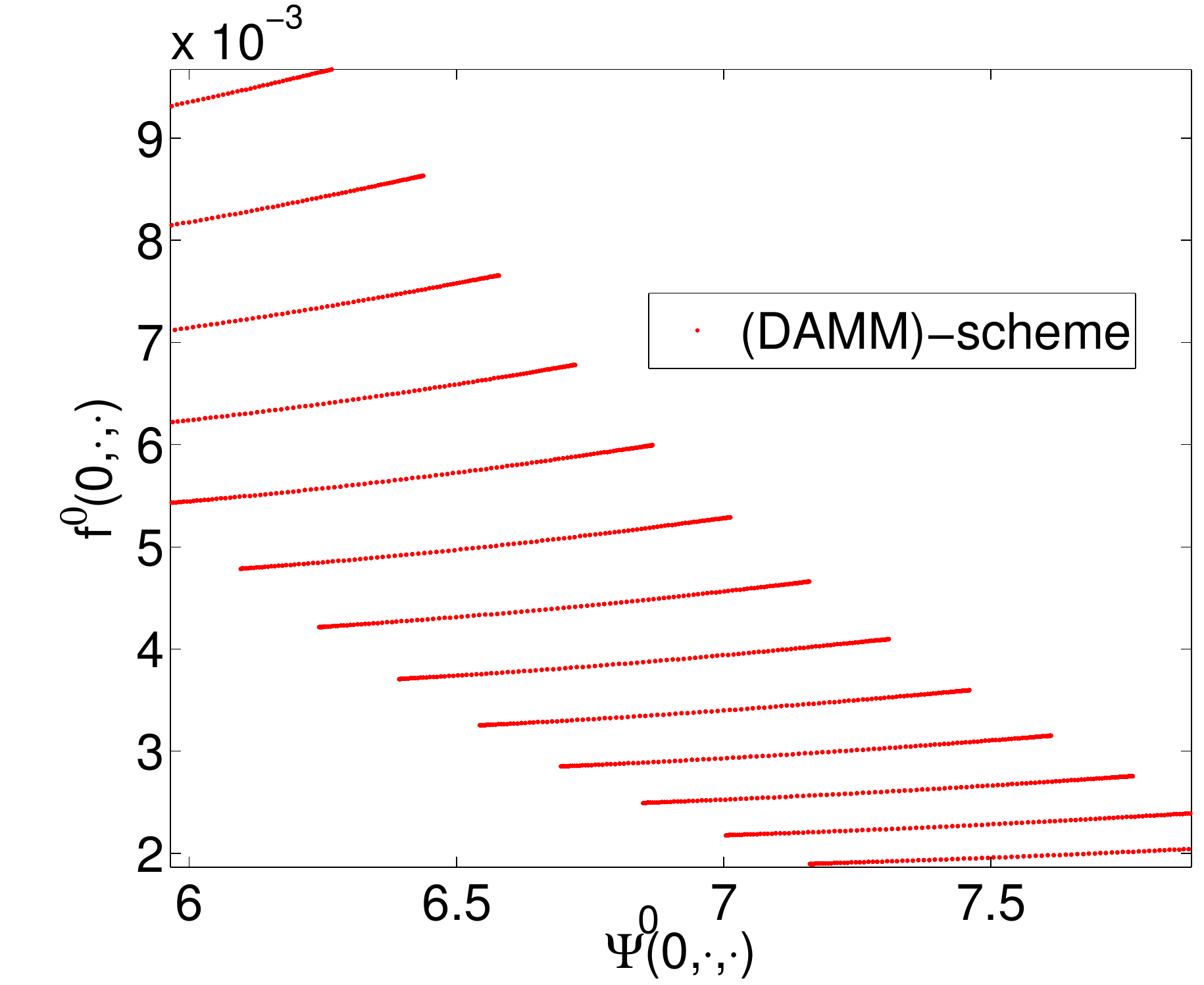}
	\caption{$n=0$.}
	\end{subfigure}
    \begin{subfigure}{.49\textwidth}
  	\centering
  	\includegraphics[width=\linewidth,trim={0cm 0cm 0cm 0cm},clip]{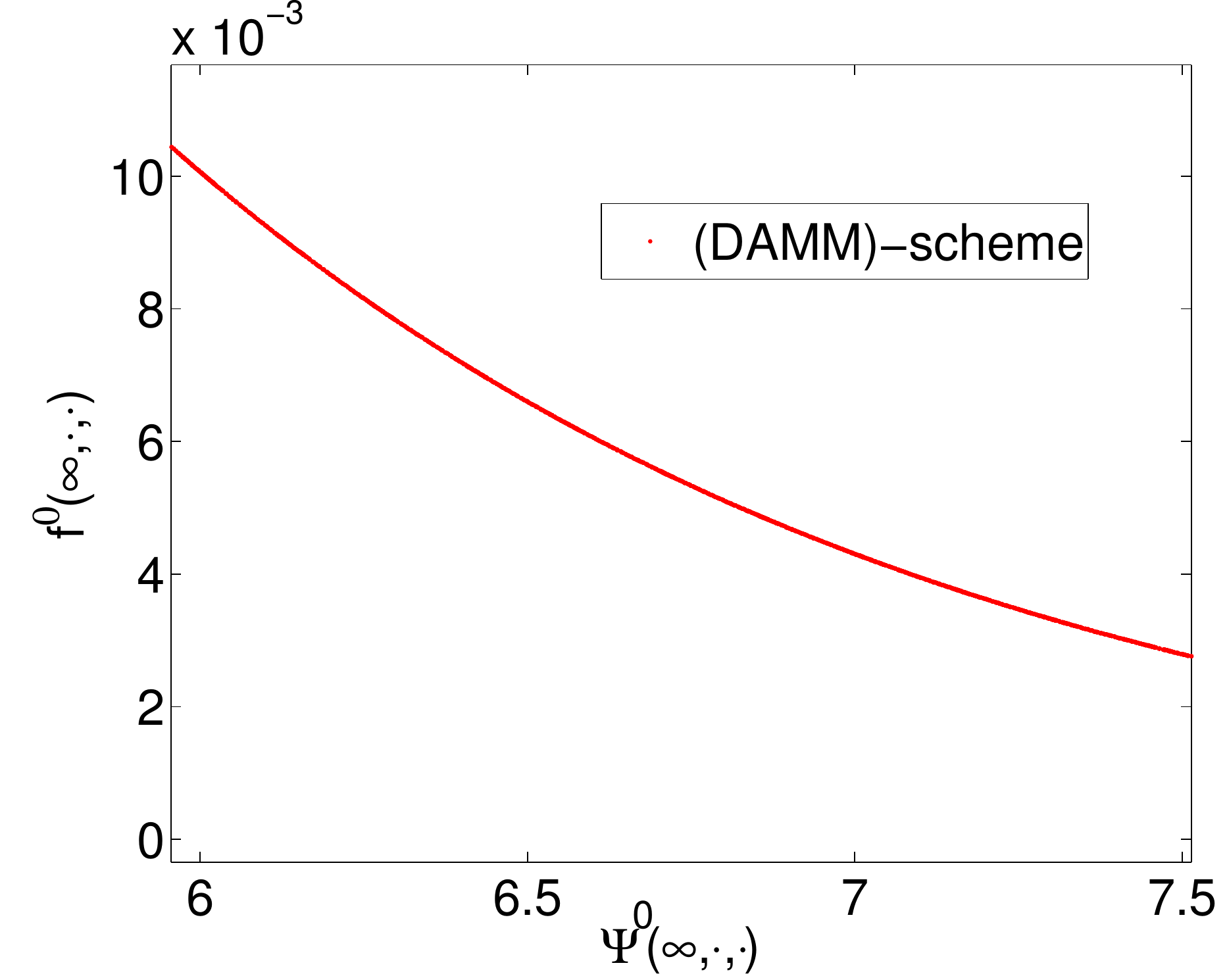}
	\caption{$n=50$ and $\eps=0$.}
	\end{subfigure} \\ 
\caption{(Two-stream instability for $\eps=0$ and $f^2_{in}$) Plot of the distribution function at times $t=0$ (panels (A), (C), (E)) and $t=50 \, \Delta t$ (panels (B), (D), (F)), with $\Delta t = 0.01$ as a function of $\Psi^0$. Mesh size: $N_x = N_y= 256$. Stabilization parameter : $\sigma = (\Delta x / L_x)^2$.}
\label{BGK_figures_curve}
\end{figure}

%\begin{figure}[h]
%\centering
%	\begin{subfigure}{.49\textwidth}
%  	\centering
%  	\includegraphics[width=\linewidth,trim={0cm 0cm 0cm 0cm},clip]{mass_BGK}
%	\caption{Mass.}
%	\end{subfigure}
%    \begin{subfigure}{.49\textwidth}
%  	\centering
%  	\includegraphics[width=\linewidth,trim={0cm 0cm 0cm 0cm},clip]{BGK_L2}
%	\caption{$L^2$-norm.}
%	\end{subfigure} \\ \vspace{0.5cm}
%	\begin{subfigure}{.49\textwidth}
%  	\centering
%  	\includegraphics[width=\linewidth,trim={0cm 0cm 0cm 0cm},clip]{momentum_BGK}
%	\caption{Momentum.}
%	\end{subfigure}
%     \begin{subfigure}{.49\textwidth}
%   	\centering
%   	\includegraphics[width=\linewidth,trim={0cm 0cm 0cm 0cm},clip]{energy_BGK}
% 	\caption{Total energy.}
% 	\end{subfigure} \vspace{0.5cm}
%% 	\begin{subfigure}{.49\textwidth}   	
%%	\centering
%%   	\includegraphics[width=\linewidth,trim={0cm 0cm 0cm 0cm},clip]{ts_damma_30}
%% 	\caption{t=30.}
%% 	\end{subfigure}
%%     \begin{subfigure}{.49\textwidth}
%%   	\centering
%%   	\includegraphics[width=\linewidth,trim={0cm 0cm 0cm 0cm},clip]{ts_damma_40}
%% 	\caption{t=40.}
%% 	\end{subfigure}
%\caption{(Two-stream instability for $\eps=0$ and $f^2_{in}$) Deviation over time for the (DAMM)-scheme of several quantities. Parameters were $N_x = N_y=256$, $\Delta t = 0.01$, and $\sigma = (\Delta x/L_x)^2$, and $\eps=0$.}
%\label{bgk_distri3}
%\end{figure}

%%%%%%%%%%%%%%%%%%%%%%%%%%%%%%%%%%%%%%%%%%%%%%%%%%%%%%%%%%%%%%%%%%%
\section{Concluding remarks and perspectives} The long-time behavior of the Vlasov-Poisson system is a challenging problem, requiring some investigations. Numerically, it is arduous to obtain a solution avoiding numerical pollution in such time asymptotics. We have developed an asymptotic-preserving scheme, based on a micro-macro decomposition coupled with a stabilization procedure in order to limit this problem. The analysis of the two-stream instability has shown the remarkable properties of the (DAMM)-scheme, permitting to attain a BGK-like equilibrium in few time iterations with low numerical costs and small errors. Nevertheless, the (DAMM)-scheme could be improved, notably through its stabilization part. The circle test case helped us a lot to better understand the choice of the stabilization parameter. But the Vlasov-Poisson test case shows that this parameter brokes the conservation properties of the system. One may imagine for a future work to replace the stabilization parameter by a more general operator which could improve the conservation properties of the Vlasov-Poisson equation.

\bigskip

%%%%%%%%%%%%%%%%%%%%%%%%%%%%%%%%%%%%%%%%%%%%%%%%%%%%%%%%%%%%%%%%%%%
\noindent {\bf Acknowledgments.} The authors would like to acknowledge support from the ANR PEPPSI  (Plasma Edge Physics and Plasma-Surface Interactions, 2013-2017). Furthermore, this work has been carried out within the framework of the EUROfusion Consortium and has received funding from the Euratom research and training program 2014-2018 under grant agreement No 633053. The views and opinions expressed herein do not necessarily reflect those of the European Commission.

%%%%%%%%%%%%%%

\begin{thebibliography}{99}
\footnotesize{

\bibitem{ara} 
\newblock A. Arakawa, 
\newblock Computational design for long-term numerical integration of the equations of fluid motion: two dimensional incompressible flow, 
\newblock \emph{Journal of Computational Physics}, \textbf{135} (1966), 119--143.

%\bibitem{berger} Berger, M.: Linear Algebra: I theory and Conditioning. New York University. Department of computer science, (2002).

\bibitem{bochev} 
\newblock P.B Bochev and R. B Lehoucq, 
\newblock Regularization and stabilization of discrete saddle-point variational problems,
\newblock \emph{Electronic Transactions on Numerical Analysis}, {\bf 22} (2006), 97--113.

\bibitem{BFR}
\newblock S. Boscarino, F. Filbet, and G. Russo,
\newblock High order semi-implicit schemes for time dependent partial differential equations, 
\newblock \emph{Journal of Scientic Computing}, {\bf 68} (2016), 975--1001.


\bibitem{mb} 
\newblock M. Bostan, 
\newblock Transport equations with disparate advection fields. Application to the gyrokinetic models in plasma physics, 
\newblock \emph{SIAM J. Sci. Comp.}, {\bf 31} (2008), 334--368.

\bibitem{mbflr} 
\newblock M. Bostan,   
\newblock The Vlasov-Poisson system with strong external magnetic field. Finite Larmor radius regime, 
\newblock \emph{Asymptot. Anal.}, {\bf 61} (2009), 91--123.

\bibitem{brezin} 
\newblock C. Brezinski, M. Redivo-Zaglia, G. Rodriguez and S. Seatzu, 
\newblock Multi-parameter regularization techniques for ill-conditioned linear systems, 
\newblock \emph{Numer. Math.}, {\bf 94} (2003), 203--228.

%\bibitem{brezzi}
%Brezzi, F., Fortin, M.: Mixed and Hybrid Finite Element Methods, Springer-Verlag. New-York (1991).

%\bibitem{brezzi2} 
%\newblock D. Boffi, F. Brezzi and M. Fortin,
%\newblock \emph{Mixed Finite Element Methods and Applications}, 
%\newblock Springer, New-York, 2003.

%\bibitem{Briz} Brizard A.J., Hahm T.S.: Foundations of nonlinear gyrokinetic theory. Rev. Mod. Phys. Vol. 79, 421--468 (2007).

\bibitem{calvetti} 
\newblock D. Calvetti, S. Morigi, L. Reichel and F. Sgallari, 
\newblock Tikhonov regularization and the L-curve for large discrete ill-posed problems,
\newblock \emph{J. Comput. Appl. Math.}, {\bf 123} (2000), 423--446.

\bibitem{CHENF}
\newblock F.F Chen,
\newblock \emph{Plasma Physics and Controlled Fusion}, 
\newblock Springer, New-York, 2006.

\bibitem{CG}
\newblock Y.Cheng, I.M. Gamba, and P.J Morrison,
\newblock Study of conservation and recurrence of Runge-Kutta discontinuous Galerkin schemes for Vlasov-Poisson systems,
\newblock \emph{Journal of Scientific Computing}, {\bf 56} (2013), 319--349.

\bibitem{crestetto2} 
\newblock A. Crestetto, N. Crouseilles and M. Lemou,
\newblock Asymptotic-Preserving scheme based on a Finite Volume/Particle-In-Cell coupling for Boltzmann- BGK-like equations in the diffusion scaling, submitted. 

\bibitem{crou3} 
\newblock N. Crouseilles and M. Lemou, 
\newblock An asymptotic preserving scheme based on a micro-macro decomposition for collisional Vlasov equations: diffusion and high-field scaling limits,  
\newblock \emph{Kinet. Relat. Models}, {\bf 4} (2011), no. 2, 441--477. 

\bibitem{crou}
\newblock N. Crouseilles, M. Mehrenberger and F. Vecil, 
\newblock Discontinuous Galerkin semi-Lagrangian method for Vlasov-Poisson, 
\newblock \emph{ESAIM: Proceedings}, {\bf 32} (2011), 211--230.

\bibitem{DDLNN} 
\newblock P. Degond, F. Deluzet, A. Lozinski, J. Narski, and C. Negulescu,
\newblock Duality based asymptotic-preserving method for highly anisotropic diffusion equations, 
\newblock \emph{Communications in Mathematical Sciences}, {\bf 10} (2012), no. 1, 1--31.

\bibitem{DLNN} 
\newblock P. Degond,  A. Lozinski, J. Narski and C. Negulescu,
\newblock An Asymptotic-Preserving method for highly anisotropic elliptic equations based on a micro-macro decomposition, 
\newblock \emph{Journal of Computational Physics}, {\bf 231} (2012), no. 7, 2724--2740. 

\bibitem{Degond}
\newblock P. Degond and M.Tang,
\newblock All speed scheme for the low mach number limit of the Isentropic Euler equation,
\newblock \emph{Communications in Computational Physics}, {\bf 10} (2011), 1--31.

\bibitem{dimarco} 
\newblock G. Dimarco, R. Loub\`{e}re, and M-H. Vignal, 
\newblock Study of a New Asymptotic Preserving Scheme for the Euler System in the Low Mach Number Limit, 
\newblock \emph{SIAM J. Sci. Comput.}, \textbf{39} (2016), 2099--2128.

\bibitem{engl}
\newblock H. W. Engl, M. Hanke, and A. Neubauer,
\newblock \emph{Regularization of Inverse Problems},
\newblock Kluwer Academic Publishers, Netherlands, 1996.

%\bibitem{Fijalkow} Fijalkow, E.: Numerical solution to the Vlasov equation : the 1D code. MAPMO. UMR 6628 (1998).

%\bibitem{filbet} F. Filbet, E. Sonnendru?cker, Comparison of Eulerian Vlasov solvers, Computer Physics Communications 150 (2003) 247?266.
\bibitem{FN}
\newblock B. Fedele and C. Negulescu,
\newblock Numerical study of an anisotropic Vlasov equation arising in plasma physics, preprint.

\bibitem{filbet} 
\newblock F. Filbet and S. Jin, 
\newblock A class of asymptotic-preserving schemes for kinetic equations and related problems with stiff sources, 
\newblock \emph{Journal of Computational Physics}, {\bf 229} (2010), 7625--7648. 

\bibitem{FP} 
\newblock F. Filbet and L. Pareschi, 
\newblock A numerical method for the accurate solution of the Fokker-Planck equation in the non-homogenous case,
\newblock \emph{Journal of Computational Physics} {\bf 179} (2002), 1--26.

\bibitem{frenod1}
\newblock E. Fr\'enod and E.Sonnendr\"ucker,  
\newblock Homogenization of the Vlasov Equation and of the Vlasov-Poisson System with a Strong External Magnetic Field,
\newblock \emph{Asymp. Anal.}, {\bf 18} (1998), 193--214.

\bibitem{gamba} 
\newblock  R.E. Heath, I.M Gamba, P.J. Morrison and C. Michler, 
\newblock A discontinuous Galerkin method for the Vlasov-Poisson system,
\newblock \emph{Journal of Computational Physics} {\bf 231} (2012), 1140--1174.

%\bibitem{frenod2}
%Fr\'enod, E., Sonnendr\"ucker, E.: Long Time Behavior of the Vlasov Equation with Strong External Magnetic Field. Math. Mod. Meth. Appl. Sciences, Vol. 10, No 4, pp 539--553, (2000).

%\bibitem{frenod3}
%Fr\'enod, E., Sonnendr\"ucker, E.:  The finite Larmor radius approximation. SIAM J. Math. Anal., Vol. 32, No 6, pp 1227--1247, (2001).

%\bibitem{Fri} Frieman E.A., Chen L.: Nonlinear gyrokinetic equations for low-frequency electromagnetic waves in general plasma equilibria. Physics of Fluids {\bf 25}, 502--5 (1982).

\bibitem{jin} 
\newblock S. Jin, L. Wang, 
\newblock An asymptotic preserving scheme for the Vlasov-Poisson-Fokker-Planck system in the high field regime,
\newblock \emph{Acta Math. Sci. Ser. B Engl.}, {\bf 31} (2011), 2219--2232.

\bibitem{klar}
\newblock A.Klar, 
\newblock An asymptotic-induced scheme for non-stationary transport equations in the diffusive limit, \newblock \emph{SIAM Journal of Numerical Analysis}, {\bf 35} (1998), 1073--1094.

%\bibitem{Xav} Garbet X., Idomura Y., Villard L., Watanabe T.: Gyrokinetic simulations of turbulent transport. Nuclear Fusion Vol. 50, No 4, (2010).
%
%\bibitem{Max}
%Ghendrih Ph., Hauray M., Nouri A.: Derivation of a gyrokinetic model. Existence and uniqueness of soecific stationary solutions. Kinetic and Related Models, Vol. 2, No 4, pp 707-725, (2009).
\bibitem{majda}
\newblock A. J. Majda and A. L Bertozzi,
\newblock \emph{Vorticity and Incompressible Flow},
\newblock Cambridge University Press, United Kingdom, 2002.
\bibitem{mak} 
McKinstrie, C. J., Giacone, R. E. and Startsev, E.A.: 
Accurate formulas for the Landau damping rates of electrostatic waves. 
\newblock \emph{Physics of Plasmas}, {\bf 6} (1999), 463--466.

\bibitem{mouhot} 
\newblock C.Villani and C. Mouhot, 
\newblock On Landau damping, 
\newblock \emph{Acta Math.}, { \bf 207} (2011), 29--201.

\bibitem{MN} 
\newblock A. Mentrelli and C. Negulescu, 
\newblock Asymptotic-Preserving scheme for highly anisotropic non-linear diffusion equations, 
\newblock \emph{Journal of Comp. Phys}, {\bf 231} (2012), 8229--8245. 

\bibitem{golse}
\newblock F. Golse and L. Saint-Raymond,
 \newblock The Vlasov-Poisson system with strong magnetic field,
 \newblock \emph{J. Math. Pures Appl.}, {\bf78} (2001), 791--817.

%\bibitem{golse2} Golse, F., Jin, S., Levermore, C.D., A domain decomposition analysis for a two-scale linear transport problem, M2AN Math. Model. Numer. Anal. {\bf 37} (2003), 869-892.

%\bibitem{nichol} Nicholson, D.R.: Introduction to Plasma Theory, John Wiley and Sons, Inc. 1983.
%
\bibitem{HM}
\newblock 
R.D. Hazeltine and J.D. Meiss,  
\newblock \emph{Plasma Confinement}, 
\newblock Dover Publications, Inc. Mineola, New-York, 2003.
%
%\bibitem{holmes} Holmes, M.H.: Introduction to Numerical Methods in Differential Equations, Springer. New York, (2007).

\bibitem{oneil} 
\newblock T. O'Neil, 
\newblock Collisionless damping of nonlinear plasma oscillations, 
\newblock \emph{Phys. Fluids}, {\bf 8} (1965), 2255--2262.

\bibitem{landau} 
\newblock L. Landau, 
\newblock On the vibration of the electronic plasma.  
\newblock \emph{English translation in J. Phys. (USSR)}, {\bf 10} (1946), 25. 

\bibitem{lebris} 
\newblock C. Le Bris, 
\newblock \emph{Syst\`{e}mes Multi-\'{e}chelles. Mod\'{e}lisation et Simulation}, 
\newblock Springer-Verlag, Berlin Heidelberg, 2005.

%\bibitem{lm}
%\newblock M. Lemou and L. Mieussens,
%\newblock A new asymptotic preserving scheme based on micro-macro formulation for linear kinetic equations in the diffusion limit,
%\newblock \emph{J. Differential Equations}, {\bf 249} (2010), 1620--1663. 

%\bibitem{leveque} LeVeque, R.J.: Finite Difference Methods for Ordinary and Partial Differential Equations. Siam, Philadelphia, (2007).

%\bibitem{negulescu1} Negulescu, C.: Modeling, Simulation and Mathematical Analysis of Magnetically Confined Plasmas. Porto-Ercole summer school, (2012).


\bibitem{Ruther}
\newblock R. J. Goldston, P.H. Rutherford,   
\newblock \emph{Plasma Physics},
\newblock Taylor $\&$ Francis Group, Philadelphia, 1995.

\bibitem{wein} 
\newblock E. Weinan,
\newblock \emph{Principles of Multiscale Modeling},
\newblock Cambridge University Press, United Kingdom, 2011.
%
%\bibitem{sarkar}
%Sarkar, S.: Bump-on-tail instability in space plasmas, Physics of Plasmas 22, 102109 (2015).
%
%\bibitem{pearson}
%Yger, A., Weil, J.A.: Math\'ematiques Appliqu\'ees $L3$. Pearson Education, (2009).
%
%\bibitem{sonne} E. Sonnendrucker, J. Roche,  P. Bertrand, and A. Ghizzo. : The Semi-Lagrangian Method for the Numerical  Resolution of the Vlasov Equation. J. Comput. Phys. 149 (1999) pp. 201-220.



}

\end{thebibliography}
\end{document}